%% file: thesis.tex
\documentclass{brandeis-dissertation}

\usepackage{amsmath, amsfonts, amsthm}
\usepackage[pdftex]{graphics}
\usepackage{hyperref}

\newcommand{\chisq}[1][{}]{\chi^2_{#1}}
\newcommand{\pdf}{p.d.f.\@}
\newcommand{\iid}{i.i.d.\@}
\newcommand{\random}[1]{\mathbf{#1}}
\newcommand{\prob}[1]{\mathbb{#1}}
\newcommand{\Lap}[1]{\Delta^{\mathbb{#1}}_x}
\newcommand{\varO}{\mathcal{O}}
\newcommand{\varD}{\mathcal{D}}
\newcommand{\vare}{\tilde{e}}
\newcommand{\varV}{\tilde{V}}
\newcommand{\varK}{\widetilde{K}}

\newcommand{\barSigmaM}{\bar{\Sigma}^M}
\newcommand{\barSigmainf}{\bar{\Sigma}^{\infty}}
\newcommand{\tildeSigmaM}{\tilde{\Sigma}^M}
\newcommand{\tildeSigmainf}{\tilde{\Sigma}^{\infty}}
\newcommand{\barGammaM}{\bar{\Gamma}^M}
\newcommand{\barGammainf}{\bar{\Gamma}^{\infty}}
\newcommand{\BbarGammaM}{\bar{\bar{\Gamma}}^M}
\newcommand{\BbarGammainf}{\bar{\bar{\Gamma}}^{\infty}}
\newcommand{\tildeGammaM}{\tilde{\Gamma}^M}
\newcommand{\tildeGammainf}{\tilde{\Gamma}^{\infty}}
\newcommand{\varvarphi}{\tilde{\varphi}}
\newcommand{\varpsi}{\tilde{\psi}}
\newcommand{\tildePsi}{\tilde{\Psi}}
\newcommand{\Lag}[2][M-N]{L^{(#1)}_{#2}}
\newcommand{\qLag}[2][2(M-N)]{L^{(#1)}_{#2}}
\newcommand{\mVandermonde}[1][\lambda]{\ensuremath{\tilde{V}^4}(#1)}
\newcommand{\Qj}[1][j]{\ensuremath{Q_{(#1)}}}
\newcommand{\sj}[1][j]{\ensuremath{s_{(#1)}}}
\newcommand{\ssj}[1][j]{\ensuremath{\tilde{s}_{(#1)}}}
\newcommand{\Pf}{\ensuremath{\mathrm{Pf}}}
\newcommand{\Ptt}[1][T]{\ensuremath{P_{#1}}}
\newcommand{\mPtt}[1][T]{\ensuremath{\widetilde{P}_{#1}}}
\newcommand{\Ttildepsi}{\tilde{\tilde{\psi}}}
\newcommand{\TtildeGammaM}{\tilde{\tilde{\Gamma}}^M}

\DeclareMathOperator{\E}{E}
\DeclareMathOperator{\var}{var}
\DeclareMathOperator{\cov}{cov}
\DeclareMathOperator{\Tr}{Tr}
\DeclareMathOperator{\GOE}{GOE}
\DeclareMathOperator{\GUE}{GUE}
\DeclareMathOperator{\GSE}{GSE}
\DeclareMathOperator{\arm}{arm}
\DeclareMathOperator{\leg}{leg}
\DeclareMathOperator{\Ai}{Ai}
\DeclareMathOperator{\Airy}{Airy}
\DeclareMathOperator{\remote}{remote}
\DeclareMathOperator{\local}{local}
\DeclareMathOperator{\rt}{right}
\DeclareMathOperator{\semi}{semi}

\newtheorem{prop}{Proposition}[chapter]
\newtheorem{lemma}{Lemma}[chapter]
\newtheorem{theorem}{Theorem}[chapter]

\theoremstyle{definition}
\newtheorem{definition}{Definition}[chapter]
\newtheorem{fact}{Fact}[chapter]

\theoremstyle{remark}
\newtheorem{remark}{Remark}[chapter]

\title{Spiked Models in Wishart Ensemble}
\author{Dong Wang}
\department{Department of Mathematics}
\adviser{Mark Adler}
\reader{Pierre van Moerbeke}
\reader{Ira Gessel}
\reader{Albion Lawrence}
\dean{Adam Jeffe}
\graduationyear{2008}
\graduationmonth{May}
\thesisabstract{\input{abstract}}
\acknowledgments{\input{acknowledgments}}

\begin{document}
\thesisfront


\input{introduction}

\input{complex_spiked_model}

\input{quaternionic_spiked_model}

\input{interpolations}

\input{asymptotic_analysis}
\bibliographystyle{plain}
\bibliography{bibliography}
\end{document}

%% file: abstract.tex
The spiked model is an important special case of the Wishart ensemble, and a natural generalization of the white Wishart ensemble. Mathematically, it can be defined on three kinds of variables: the real, the complex and the quaternion. For practical application, we are interested in the limiting distribution of the largest sample eigenvalue.

We first give a new proof of the result of Baik, Ben Arous and P\'{e}ch\'{e} for the complex spiked model, based on the method of multiple orthogonal polynomials by Bleher and Kuijlaars. Then in the same spirit we present a new result of the rank $1$ quaternionic spiked model, proven by combinatorial identities involving quaternionic Zonal polynomials ($\alpha = 1/2$ Jack polynomials) and skew orthogonal polynomials.

We find a phase transition phenomenon for the limiting distribution in the rank $1$ quaternionic spiked model as the spiked population eigenvalue increases, and recognize the seemingly new limiting distribution on the critical point as the limiting distribution of the largest sample eigenvalue in the real white Wishart ensemble.

Finally we give conjectures for higher rank quaternionic spiked model and the real spiked model.

%% file: acknowledgments.tex
Great thanks are due to my advisor, Mark Adler, for his guidance and support throughout my development as a research mathematician, and particularly the process of producing this thesis. He has helped me in forming the big picture of mathematics and even helped me in details like the proof reading of this thesis.

I also thank Jinho Baik for information on related literature, Ira Gessel for help in combinatorial, and Pierre van Moerbeke for the course in random matrix.

Among graduate peers, Hsin-Hong Lai not only elucidates algebraic geometry to me from his erudition, but also set a model of persistence, devotion to mathematics, and a healthy dose of idealism; Junbo Wang and Chao-Jen Wang, who are also my roommates, chat a lot with me on every topic, and particular to Junbo who takes me shopping with his car; Alex Chris taught me a lot about the English language; others who helped me in some way are too many to mention. I thank them all.

The understanding of my family is crucial for my graduate study abroad. At last, I am grateful to Sha Liu, my fianc\'{e}e, whose support and encouragement are never in short supply.

%% file: introduction.tex
\chapter{Introduction}

\label{introduction}

\section{Wishart distribution}

The Wishart distribution is a multivariate generalization of the $\chisq$ distribution. The $\chisq$ distribution is defined by the normal distribution. The normal distribution is well known as a distribution of one real variable. However, in this thesis we study statistics of three kinds of variables: real, complex and quaternion. Thus we begin the thesis with a review of the normal distribution of all the three kinds of variables. All results for real variables are standard, see e.g. \cite{Muirhead82}; for complex variables, see \cite{Goodman63}; for quaternion variables, see \cite{Andersson75}.

\subsection{Normal distribution}

The normal distribution of a real variable with mean $\mu$ and variance $\sigma^2$ is defined by the probability density function (\pdf)
\begin{equation}
P(x) = \frac{1}{\sqrt{2\pi}\sigma} e^{-\frac{(x-\mu)^2}{2\sigma^2}}.
\end{equation}

When we study the distribution of a complex variable $z$, we can view it as two possibly dependent real variables $x$ and $y$, which are its real and imaginary parts: $z = x+iy$. The definitions of mean and variance is similar to those of real variables:
\begin{align}
\E(z) = & \E(x) + \E(y), \label{eq:definition_of_mean_on_C} \\
\var(z) = & \E((z - \E(z)) \overline{(z - \E(z))}) = (\E(x^2) - \E^2(x)) + (\E(y^2) - \E^2(y)) \notag \\
 = & \var(x) + \var(y). \label{eq:definition_of_variance_on_C}
\end{align}
The normal distribution of a complex variable $z$ with mean $\mu$ and variance $\sigma^2$ is defined by the \pdf\ 
\begin{equation}
P(z) = \frac{1}{\pi\sigma^2} e^{-\frac{(x-\Re(\mu))^2 + (y-\Im(\mu))^2}{\sigma^2}} = \frac{1}{\pi\sigma^2} e^{-\frac{\lvert z-\mu \rvert^2}{\sigma^2}},
\end{equation}
so that $x$ and $y$ are independent real variables in normal distribution with means $\Re(\mu)$ and $\Im(\mu)$ respectively, and identical variance $\sigma^2/2$.

A quaternion variable $u$ has $4$ real parts: $u = x+iy+jz+kw$. The definitions of mean and variance is similar to \eqref{eq:definition_of_mean_on_C} and \eqref{eq:definition_of_variance_on_C}: \footnote{In this thesis we use the same notations of conjugation and norm for both complex and quaternion variables. For quaternions, $\overline{x+iy+jz+kw} = x-iy-jz-kw$ and $\lvert x+iy+jz+kw \rvert = \sqrt{x^2+y^2+z^2+w^2}$.}
\begin{align}
\E(u) = & \E(x) + i\E(y) + j\E(z) + k\E(w), \\
\var(u) = & \E((u - \E(u)) \overline{(u - \E(u))}) = \var(x) + \var(y) + \var(z) + \var(w).
\end{align}
The normal distribution of a quaternionic variable $u$ with mean $\mu$ and variance $\sigma^2$ is defined by the \pdf\
\begin{equation}
P(u) = \frac{1}{\pi^2\sigma^4/4} e^{-\frac{\lvert u-\mu \rvert^2}{\sigma^2/2}},
\end{equation}
so that the $4$ parts, $x$, $y$, $z$ and $w$ are independent real variables in normal distribution with means the corresponding parts of $\mu$ and identical variance $\sigma^2/4$.

\subsection{$\chisq$ distribution}

The $\chisq$ distribution, like the normal distribution, can be defined for all the three kinds of variables. Let $\random{x}$ be a random variable in normal distribution with mean $0$ and variance $\sigma^2$, and we take $k$ independent measurements of $\random{x}$, with results $x_1$, \dots, $x_k$, which are accordingly random variables with identical independent (\iid) $(0, \sigma^2)$ normal distribution. If we define the random variable 
\begin{equation} \label{eq:definition_of_s}
s = \sum^k_{j=1} \lvert x_j \rvert^2,
\end{equation}
then $s/\sigma^2$ is in the $\chisq[k]$ distribution, i.e., chi-square distribution with parameter $k$.

No matter for what kind of variables $\chisq[k]$ distribution is defined, it is the distribution of a real random variable with support $[0, \infty)$. However, the {\pdf}s of $\chisq[k]$ are not identical for the three kinds of variables:
\begin{itemize}
\item The \pdf\ of $\chisq[k]$ for a real variable:
\begin{equation}
P(x) = \frac{(1/2)^{k/2}}{\Gamma(k/2)} x^{\frac{k}{2}-1} e^{-\frac{x}{2}}.
\end{equation}

\item The \pdf\ of $\chisq[k]$ for a complex variable:
\begin{equation}
P(x) = \frac{1}{\Gamma(k)} x^{k-1} e^{-x}.
\end{equation}

\item The \pdf\ of $\chisq[k]$ for a quaternion variable:
\begin{equation}
P(x) = \frac{4^k}{\Gamma(2k)} x^{2k-1} e^{-2x}.
\end{equation}
\end{itemize}

\subsection{Multivariate normal distribution}

The normal distribution has multivariate versions for all the three kinds of variables. The $n$-variate random variable is represented by an $n$-dimensional column vector $X = (x_1, \dots, x_n)^T$, where $x_j$'s are possibly correlated random variables. The variable in the multivariate normal distribution has a mean $\mu$, which is an $n$-dimensional column vector, and a covariance matrix $\Sigma$, which is a positive defined symmetric, Hermitian or quaternionic Hermitian matrix \footnote{The definition of a quaternionic Hermitian matrix $(a_{ij})_{1 \leq i,j \leq N}$ is similar to that of a Hermitian matrix: (1) Diagonal $a_{ii}$'s are real numbers. (2) Strictly upper-triangular entries $a_{ij}$'s with $i<j$ are arbitrary quaternions. (3) $a_{ij} = \overline{a_{ji}}$ for strictly lower triangular entries $a_{ij}$ with $i>j$.}, depending on which kind of variables we consider. They are given by
\begin{align}
\mu_j = & \E(x_j), \\
\Sigma_{ij} = & \cov(x_i, x_j),
\end{align}
where for any kind of variables,
\begin{equation}
\cov(x,y) = \E((x-\E(x)) \overline{(y-\E(y))}),
\end{equation}
with the overline meaning conjugation for complex and quaternion variables and the identity for real variables.

Conversely, $\mu$ and $\Sigma$ determines the \pdf\ of $X$, by slightly different formulas for the three kinds of variables:
\begin{itemize}
\item The \pdf\ of an $n$-variate real normal variable:
\begin{equation}
P(X) = \frac{1}{(2\pi)^{n/2}\det(\Sigma)^{1/2}} e^{-\frac{1}{2} (X-\mu)^T\Sigma^{-1}(X-\mu)}.
\end{equation}

\item The \pdf\ of an $n$-variate complex normal variable:
\begin{equation}
P(X) = \frac{1}{\pi^n\det(\Sigma)} e^{-\overline{(X-\mu)}^T\Sigma^{-1}(X-\mu)}.
\end{equation}

\item The \pdf\ of an $n$-variate quaternion normal variable:
\begin{equation}
P(X) = \frac{1}{(\pi/2)^{2n}\det(\Sigma)^2} e^{-2\overline{(X-\mu)}^T\Sigma^{-1}(X-\mu)}.
\end{equation}
\end{itemize}

We need to take notice that the determinant is not well defined for quaternion matrices due to the noncommutativity of quaternions. Here since $\Sigma$ is a quaternionic Hermitian matrix with positive real eigenvalues $\sigma_1$, \dots, $\sigma_n$, we define 
\begin{equation}
\det(\Sigma) = \prod^n_{j=1} \sigma_j.
\end{equation}

\subsection{Wishart distribution}

Now we can define the Wishart distribution, which is similar to the $\chisq$ distribution. Let $\random{X}$ be an $N$-variate random variable in the normal distribution with mean $0$ and covariance matrix $\Sigma$, and we take $M$ independent measurements of $\random{X}$ with result $X_1$, \dots, $X_M$, which are accordingly random variables with \iid\ $(0, \Sigma)$ normal distribution. If we let the $N \times M$ matrix $X$ be the juxtaposition of $X_j$'s: $X = (X_1 : , \dots, : X_M)$, then we say that the $N^2$-variate random variable
\begin{equation} \label{eq:definition_of_S}
S = \frac{1}{M} X\bar{X}^T
\end{equation}
has the Wishart distribution $W_N(M, \Sigma)$ with $M$ degrees of freedom and covariance matrix $\Sigma$.

Given $M$ and $\Sigma$, we have explicit formulas of {\pdf}s of $S$ for all the three kinds of variables. Since the spaces of $N \times N$ real symmetric, Hermitian and quaternionic Hermitian matrices are Euclidean spaces with dimensions $N(N-1)/2$, $N^2$ and $N(2N-1)$ respectively, we take the usual definition of the measure: For $S = (s_{pq})$ real symmetric, $dS = \prod_{1 \leq q \leq p \leq N} ds_{pq}$. For $S = (s_{pq})$ Hermitian, $dS = \prod^N_{r=1} ds_{rr} \prod_{1 \leq q < p \leq N}d\Re s_{pq}d\Im s_{pq}$. For $S = (s_{pq})$ quaternionic Hermitian and $s_{pq} = x_{pq}+iy_{pq}+jz_{pq}+kw_{pq}$ for off-diagonal entries, $dS = \prod^N_{r=1}s_{rr} \prod_{1 \leq q < p \leq N} dx_{pq}dy_{pq}dz_{pq}dw_{pq}$. To make the support of $S$ be the full positive definite cone of the space of real symmetric matrices, Hermitian matrices or quaternionic Hermitian matrices, we require that $M \geq N$.
\begin{itemize}
\item The \pdf\ of the real Wishart distribution $W_N(M, \Sigma)$:
\begin{multline} \label{eq:real_Wishart_distribution}
P(S)dS = \frac{1} {2^{MN/2} \pi^{N(N-1)/4} (\det\Sigma)^{M/2} \prod^N_{j=1}\Gamma((M-j+1)/2)} \\
e^{-\frac{1}{2}\Tr(\Sigma^{-1}S)} (\det S)^{(M-N-1)/2} dS.
\end{multline}

\item The \pdf\ of the complex Wishart distribution $W_N(M, \Sigma)$:
\begin{equation}
P(S)dS = \frac{1} {\pi^{N(N-1)/2} (\det\Sigma)^M \prod^N_{j=1}\Gamma(M-j+1)} e^{-\Tr(\Sigma^{-1}S)} (\det S)^{M-N} dS.
\end{equation}

\item The \pdf\ of the quaternion Wishart distribution $W_N(M, \Sigma)$:
\begin{multline} \label{eq:quaternion_Wishart_distribution}
P(S)dS = \frac{2^{2MN}}{\pi^{N(N-1)} (\det\Sigma)^{2M} \prod^N_{j=1}\Gamma(2(M-j+1))} \\
e^{-2\Re\Tr(\Sigma^{-1}S)} (\det S)^{2(M-N)+1} dS.
\end{multline}
\end{itemize}
Here we should take note that although $\Tr(\Sigma^{-1}S)$ is automatically real for $\Sigma$ and $S$ to be real symmetric or complex Hermitian matrices, we need to take the real part explicitly in the quaternionic case due to the noncommutativity.

\section{Spiked models of Wishart ensemble}

In statistics, the eigenvalues $\sigma_1$, \dots, $\sigma_N$ of the covariance matrix $\Sigma$ in the multivariate normal distribution are called population eigenvalues. They are of importance in principal component analysis.

Let $\random{X}$ be a centralized $N$-variate random variable, which means its mean is $0$. Under the assumption that $\random{X}$ has the (multivariate) normal distribution, how can we determine its population eigenvalues by results of measurements?

In the $N=1$ case, it is equivalent to find the variance $\sigma^2$. If we make $k$ independent measurements and get results $x_1$, \dots, $x_k$, we have the random variable $s$ defined in \eqref{eq:definition_of_s}, and it is easy to find that $s/k$ approaches $\sigma^2$ almost surely as $k \rightarrow \infty$.

For general $N$, if we make $M$ measurements and get results $X_1$, \dots, $X_N$, we have the $N \times N$ random matrix $S$ defined in \eqref{eq:definition_of_S}, which is called the sample covariance matrix in statistics. The multivariate counterparts of $s$ are the eigenvalues $\lambda_1$, \dots, $\lambda_N$ of $S$, which are called sample eigenvalues in statistics. A celebrated result of Anderson \cite{Anderson63} states that if $M \gg N$, the sample eigenvalues are good approximations of the population eigenvalues.

However, if $M$ is not much greater than $N$, say, both $N$ and $M$ are large, and $M/N = \gamma^2 \geq 1$, then the sample eigenvalues fail to approximate the population eigenvalues. For example, if $\sigma_j$'s are identically $1$, Mar\v{c}enko and Pastur proved \cite{Marcenko-Pastur63}
\begin{prop}[Mar\v{c}enko-Pastur law]
When $\Sigma = I$, as $M, N \rightarrow \infty$ such that $M/N \rightarrow \gamma^2 \geq 1$, the limiting density of the sample eigenvalues $\lambda_i$ in the complex Wishart ensemble is given by 
\begin{equation}
\frac{1}{N} \#\{ \lambda_j | \lambda_j \leq x \} \rightarrow H(x),
\end{equation}
where ($b_1 = (1 - \gamma^{-1})^2$ and $b_2 = (1 + \gamma^{-1})^2$)
\begin{equation} \label{eq:definition_of_H}
H(x) = 
\begin{cases}
0 & x < b_1 \\
\int^x_{b_1} \frac{\gamma^2}{2\pi t} \sqrt{(t - b_1)(b_2 - t)} & b_1 \leq x \leq b_2 \\
1 & x > b_2
\end{cases}.
\end{equation}
\end{prop}
Actually the proof of Mar\v{c}enko and Pastur is only for real and complex variables. However, their proof can be transplanted to the quaternion case without much difficulty.

In this thesis, if the covariance matrix $\Sigma$ is given, we call the distribution of the sample eigenvalues $\lambda_1$, \dots, $\lambda_N$ the Wishart ensemble. It is easy to see that the Wishart ensemble is completely determined by the population eigenvalues $\sigma_1$, \dots, $\sigma_N$ and the number of measurements $M$. The Mar\v{c}enko-Pastur law gives the density of the sample eigenvalues in the $\Sigma = I$ Wishart ensemble, which is commonly called the white Wishart ensemble \footnote{In the random matrix theory literature, the real white Wishart ensemble is called the Laguerre orthogonal ensemble (LOE), and the complex and quaternionic white Wishart ensembles are called the Laguerre unitary ensemble (LUE) and the Laguerre symplectic ensemble (LSE).}.

The general problem of getting information about $\sigma_j$'s from properties of the Wishart ensemble is far too difficult. We begin with the hypothesis testing problem: For $\Sigma$ and $\Sigma'$ with different population eigenvalues, can we tell the difference between the corresponding Wishart ensembles? In the simplest case, the eigenvalues of $\Sigma$ are identically $1$, and we get the white Wishart ensemble; most eigenvalues of $\Sigma'$ are $1$ while the other $r$ sample eigenvalues are $1+\alpha_1$, \dots, $1+\alpha_r$, with $\alpha_j$'s real numbers greater than $-1$, and we call the corresponding Wishart ensemble the spiked model of rank $r$. Now the question is: Can we tell the spiked model from the white Wishart ensemble?

The density of sample eigenvalues fails to detect the difference, since the proof of Mar\v{c}enko and Pastur implies the stronger result \cite{Marcenko-Pastur63}:
\begin{prop}
Let $r$ be a fixed positive integer and $\alpha_1$, \dots, $\alpha_r$ fixed real numbers greater than $-1$. When $M, N \rightarrow \infty$ such that $M/N \rightarrow \gamma^2 \geq 1$, and the population eigenvalues are $1+\alpha_1$, \dots, $1+\alpha_r$ and all others identically $1$, for all the three kinds of variables, the limiting density of the sample eigenvalues $\lambda_j$'s in the rank $k$ spiked model is given by
\begin{equation}
\frac{1}{N} \#\{ \lambda_j | \lambda_j \leq x \} \rightarrow H(x),
\end{equation}
where $H(x)$ is defined in \eqref{eq:definition_of_H}, the same as that in the white Wishart ensemble.
\end{prop}

However, if some population eigenvalues are large, the limiting distribution of the largest sample eigenvalue may change. First, we have a complete result for the limiting distribution of the largest eigenvalue in the white Wishart ensemble. Unlike the limiting density, these limiting distributions for the three kinds of variables are different \cite{Forrester93}, \cite{Johansson00}, \cite{Johnstone01}.

\begin{prop}[GOE, GUE and GSE Tracy-Widom distributions] \label{prop:definition_of_three_white_Wisharts}
When $\Sigma = I$, as $M, N \rightarrow \infty$ such that $M/N \rightarrow \gamma^2 \geq 1$, the largest sample eigenvalue $\max(\lambda)$ approaches $(1+\gamma^{-1})^2$ almost surely for all the three kinds of variables. However, the limiting distributions for these three kinds of variables are different.
\begin{itemize}
\item For real variables, the limiting distribution of $\max(\lambda)$ is the GOE Tracy-Widom distribution, after proper rescaling: \footnote{By $M \rightarrow \infty$, we mean ``$M,N \rightarrow \infty$ and $M/N = \gamma^2$.}
\begin{equation}
\lim_{M \rightarrow \infty} \prob{P} \left( (\max(\lambda) - (1+\gamma^{-1})^2) \cdot \frac{\gamma M^{2/3}}{(1+\gamma)^{4/3}} < T \right) = F_{\GOE}(T).
\end{equation}

\item For complex variables, the limiting distribution of $\max(\lambda)$ is the GUE Tracy-Widom distribution, after proper rescaling:
\begin{equation}
\lim_{M \rightarrow \infty} \prob{P} \left( (\max(\lambda) - (1+\gamma^{-1})^2) \cdot \frac{\gamma M^{2/3}}{(1+\gamma)^{4/3}} < T \right) = F_{\GUE}(T).
\end{equation}

\item For quaternion variables, the limiting distribution of $\max(\lambda)$ is the GSE Tracy-Widom distribution, after proper rescaling:
\begin{equation}
\lim_{M \rightarrow \infty} \prob{P} \left( (\max(\lambda) - (1+\gamma^{-1})^2) \cdot \frac{\gamma (2M)^{2/3}}{(1+\gamma)^{4/3}} < T \right) = F_{\GSE}(T).
\end{equation}
\end{itemize}
\end{prop}

We have explicit formulas for these probability functions \cite{Tracy-Widom94}, \cite{Tracy-Widom96}:
\begin{align}
F_{\GOE}(\xi) = & e^{-\frac{1}{2} \int^{\infty}_{\xi} (x-\xi)q^2(x) dx} e^{-\frac{1}{2} \int^{\infty}_{\xi} q(x) dx}, \\
F_{\GUE}(\xi) = & e^{-\int^{\infty}_{\xi} (x-\xi)q^2(x) dx}, \label{eq:analytic_formula_of_F_GUE} \\
F_{\GSE}(\xi) = & \frac{1}{2} e^{-\frac{1}{2} \int^{\infty}_{\xi} (x-\xi)q^2(x) dx} \left( e^{-\frac{1}{2} \int^{\infty}_{\xi} q(x) dx} + e^{\frac{1}{2} \int^{\infty}_{\xi} q(x) dx} \right), \label{eq:analytic_formula_of_F_GSE}
\end{align}
where $q(x)$ is a solution to the Painlev\'{e} equation
\begin{equation} \label{eq:Painleve_equation_II}
q''(x) = xq(x) + 2q^3(x),
\end{equation}
with 
\begin{equation} \label{eq:boundary_condition_of_Painleve}
q(x) \sim \Ai(x) \quad \text{as} \quad x \rightarrow +\infty.
\end{equation}
Here $\Ai(x)$ is the Airy function, whose definition will be given in \eqref{eq:definition_of_Airy_function_by_contour}.

It is worth noticing that although all previous statements for the three kinds of variables are parallel, the three-fold symmetry breaks down and our statements for the three kinds of variables are going to bifurcate. Despite their similarity in analytic form, $F_{\GUE}$ is most naturally defined by a Fredholm determinant while $F_{\GOE}$ and $F_{\GSE}$ are most naturally defined by Fredholm Pfaffians, and the derivations of $F_{\GOE}$ and $F_{\GSE}$ are somehow similar to each other, and requires more work than the derivation of $F_{\GUE}$ \cite{Mehta04}.

Now the goal is to get the limiting distribution of the largest eigenvalue in spiked models of the three kinds of variables. For the complex variables, Baik, Ben Arous and P\'{e}ch\'{e} got the complete result for any finite rank $k$ \cite{Baik-Ben_Arous-Peche05}. In this thesis we derive their result by a different approach and are going to get that limiting distribution for the rank $1$ quaternionic spiked model. It is desirable to find the counterpart limiting distribution for the rank $1$ real spiked model. However, it seems that the symmetry between real and quaternion variables breaks again, and such a result is unattainable by the method in this thesis.

\section{(Generalized) Zonal polynomials}

From the {\pdf}s of the Wishart distributions \eqref{eq:real_Wishart_distribution}--\eqref{eq:quaternion_Wishart_distribution} for the three kinds of variables, we can get the {\pdf}s of sample eigenvalues $\lambda = (\lambda_1, \dots, \lambda_N)$ by the Weyl integration formula, or more directly, by calculating Jacobians. Since we do not need explicit formulas of the normalization constants, we simply write ``$C$'' from now on. The derivation for real variables is in \cite{Muirhead82}
\begin{itemize}
\item The \pdf\ of the sample eigenvalues in the real Wishart ensemble:
\begin{equation} \label{eq:pdf_of_real_Wishart_ensemble}
P(\lambda) = \frac{1}{C} \lvert V(\lambda) \rvert \prod^N_{j=1} \lambda^{(M-N-1)/2}_j \int_{O(N)} e^{-\frac{M}{2} \Tr(\Sigma^{-1}OSO^{-1})} dO.
\end{equation}

\item The \pdf\ of the sample eigenvalues in the complex Wishart ensemble:
\begin{equation} \label{eq:pdf_of_complex_Wishart_ensemble}
P(\lambda) = \frac{1}{C} V(\lambda)^2 \prod^N_{j=1} \lambda^{M-N}_j \int_{U(N)} e^{-M\Tr(\Sigma^{-1}USU^{-1})} dU.
\end{equation}

\item The \pdf\ of the sample eigenvalues in the quaternion Wishart ensemble:
\begin{equation} \label{eq:pdf_of_quaternion_Wishart_ensemble}
P(\lambda) = \frac{1}{C} V(\lambda)^4 \prod^N_{j=1} \lambda^{2(M-N)+1}_j \int_{Sp(N)} e^{-2M\Re\Tr(\Sigma^{-1}QSQ^{-1})} dQ.
\end{equation}
\end{itemize}
Here $V(\lambda) = \prod_{1 \leq i < j \leq N} (\lambda_i - \lambda_j)$ is the Vandermonde, and the integrals in \eqref{eq:pdf_of_real_Wishart_ensemble}--\eqref{eq:pdf_of_quaternion_Wishart_ensemble} are over orthogonal, unitary and compact symplectic groups with Haar measures respectively. To go further in our analysis, we need to evaluate these integrals.

It is well know among statisticians that we can expand the integral in \eqref{eq:pdf_of_real_Wishart_ensemble} by Zonal polynomials \cite{Muirhead82}. To define Zonal polynomials, and their counterparts for complex and quaternion variables, we need some preliminary definitions.

\begin{definition} \cite{Macdonald95}
A partition $\kappa$ of $k$ is a sequence $\kappa = (\kappa_1, \kappa_2, \dots, \kappa_l)$ where $\kappa_j \geq 0$ are weakly decreasing and $\sum^l_{j=1}\kappa_j = k$. We denote this by $\kappa \vdash k$.
\end{definition}
For example, $\kappa = (2,2,1)$ is a partition of $5$. The number of nonzero parts of $\kappa$ is called the length of $\kappa$, denoted as $l(\kappa)$. If we drop the weakly decreasing condition, we call the sequence a general partition.

If $\kappa$ and $\kappa'$ are two general partitions of $k$, we say $\kappa < \kappa'$ if for some index $j$, $\kappa_i = \kappa'_i$ for $i < j$ and $\kappa_j < \kappa'_j$. For example,
\begin{equation}
(2,1,1,1) < (2,2,1) < (3,2).
\end{equation}

If $\kappa \vdash k$ is a general partition with $l(\kappa) = l$, we define the monomial of degree $k$
\begin{equation}
x^{\kappa} = x^{\kappa_1}_1 x^{\kappa_2}_2 \dots x^{\kappa_l}_l
\end{equation}
and say $x^{\kappa'}$ is of higher weight than $x^{\kappa}$ if $\kappa' > \kappa$.

We need another definition of Laplacians $\Lap{R}$, $\Lap{C}$ and $\Lap{H}$ \cite{Sekiguchi76}:
\begin{definition}
For $N$ variables $x = (x_1, \dots, x_N)$, we define
\begin{align}
\Lap{R} = & \sum^N_{j=1} x^2_j\frac{\partial^2}{\partial^2 x^2_j} + \sum^N_{\substack{i=1 \\ i \neq j}} \frac{x^2_j}{x_j - x_i} \frac{\partial}{\partial x_j}, \\
\Lap{C} = & \sum^N_{j=1} x^2_j\frac{\partial^2}{\partial^2 x^2_j} + 2\sum^N_{\substack{i=1 \\ i \neq j}} \frac{x^2_j}{x_j - x_i} \frac{\partial}{\partial x_j}, \\
\Lap{H} = & \sum^N_{j=1} x^2_j\frac{\partial^2}{\partial^2 x^2_j} + 4\sum^N_{\substack{i=1 \\ i \neq j}} \frac{x^2_j}{x_j - x_i} \frac{\partial}{\partial x_j}.
\end{align}
\end{definition}

Now we can give a definition of Zonal polynomials and their counterparts, complex Zonal polynomials and quaternionic Zonal polynomials \cite{Macdonald86}.
\begin{definition}
For $N$ variables $x = (x_1, \dots, x_N)$, a nonnegative integer $k$ and a partition $\kappa \vdash k$, we have the unique Zonal polynomial $Z_{\kappa}(x)$, complex Zonal polynomial $C_{\kappa}(x)$ and quaternionic Zonal polynomial $Q_{\kappa}(x)$, which are all symmetric, homogeneous polynomials of degree $k$ in $x_j$'s such that
\begin{itemize}
\item The highest weight term in $Z_{\kappa}(x)$ ($C_{\kappa}(x)$, $Q_{\kappa}(x)$) is $x^{\kappa}$.

\item $Z_{\kappa}(x)$ ($C_{\kappa}(x)$, $Q_{\kappa}(x)$) is an eigenfunction of the Laplacian $\Lap{R}$ ($\Lap{C}$, $\Lap{H}$).

\item \begin{equation} \label{eq:three_decompositions}
\sum_{\kappa \vdash k} Z_{\kappa}(x) = \sum_{\kappa \vdash k} C_{\kappa}(x) = \sum_{\kappa \vdash k} Q_{\kappa}(x) = (x_1 + \dots + x_N)^k.
\end{equation}
\end{itemize}
\end{definition}
Form the highest weight property of $Z_{\kappa}(x)$ ($C_{\kappa}(x)$, $Q_{\kappa}(x)$), we have
\begin{fact} \label{fact:variable_number_reduction}
For any $N$ and $\kappa \vdash k$, 
\begin{equation}
Z_{\kappa}(x) = C_{\kappa}(x) = Q_{\kappa}(x) = 0 \quad \text{if} \quad l(\kappa) > N.
\end{equation}
\end{fact}

Latter we apply the notation $Z_{\kappa}(X)$ to mean $Z_{\kappa}(x_1, \dots, x_N)$ if $X$ is an $N \times N$ matrix with eigenvalues $x_1, \dots, x_N$, and similarly for $C_{\kappa}(X)$ and $Q_{\kappa}(X)$.

It can be derived from their eigenfunction property that \cite{Muirhead82}, \cite{Macdonald86}, \cite{Macdonald95}, \cite{Hanlon-Stanley-Stembridge92}
\begin{prop}
Given $X$ and $Y$ to be $N \times N$ symmetric matrices, Hermitian matrices or quaternionic Hermitian matrices, and $I_N$ to be the $N \times N$ identity matrix, we have
\begin{align}
\int_{O(N)} Z_{\kappa}(XOYO^{-1}) dO = & \frac{Z_{\kappa}(X) Z_{\kappa}(Y)}{Z_{\kappa}(I_N)}, \\
\int_{U(N)} C_{\kappa}(XUYU^{-1}) dU = & \frac{C_{\kappa}(X) C_{\kappa}(Y)}{C_{\kappa}(I_N)}, \\
\int_{Sp(N)} \Re Q_{\kappa}(XQYQ^{-1}) dQ = & \frac{Q_{\kappa}(X) Q_{\kappa}(Y)}{Q_{\kappa}(I_N)}.
\end{align}
\end{prop}
Then by \eqref{eq:three_decompositions} we have
\begin{equation}
\begin{split}
\int_{O(N)} e^{\rho \Tr(XOYO^{-1})} dO = & \sum^{\infty}_{k=0} \frac{\rho^k}{k!} \int_{O(N)} \Tr(XOYO^{-1})^k dO \\
= & \sum^{\infty}_{k=0} \frac{\rho^k}{k!} \sum_{\substack{\kappa \vdash k \\ l(\kappa) \leq N}} \int_{O(N)} Z_{\kappa}(XOYO^{-1}) dO \\
= & \sum^{\infty}_{k=0} \frac{\rho^k}{k!} \sum_{\substack{\kappa \vdash k \\ l(\kappa) \leq N}} \frac{Z_{\kappa}(X) Z_{\kappa}(Y)}{Z_{\kappa}(I_N)},
\end{split}
\end{equation}
and similarly
\begin{align}
\int_{U(N)} e^{\rho \Tr(XUYU^{-1})} dU = & \sum^{\infty}_{k=0} \frac{\rho^k}{k!} \sum_{\substack{\kappa \vdash k \\ l(\kappa) \leq N}} \frac{C_{\kappa}(X) C_{\kappa}(Y)}{C_{\kappa}(I_N)}, \\
\int_{Sp(N)} e^{\rho \Re\Tr(XQYQ^{-1})} dQ = & \sum^{\infty}_{k=0} \frac{\rho^k}{k!} \sum_{\substack{\kappa \vdash k \\ l(\kappa) \leq N}} \frac{Q_{\kappa}(X) Q_{\kappa}(Y)}{Q_{\kappa}(I_N)}.
\end{align}

Now we can state the series formulas for joint {\pdf}s for the three kinds of variables of the Wishart ensemble:
\begin{prop}
Let the centralized $N$-variate normal random variable have the covariance matrix $\Sigma$ with population eigenvalues $\sigma_1$, \dots, $\sigma_N$, and the number of measurements be $M \geq N$. The \pdf\ of the sample eigenvalues $\lambda = (\lambda_1, \dots, \lambda_N)$ in the Wishart ensemble is
\begin{itemize}
\item For the real variable case
\begin{multline} \label{eq:Zonal_expanssion}
P(\lambda) = \frac{1}{C} \lvert V(\lambda) \rvert \prod^N_{j=1} \lambda^{(M-N-1)/2}_j \\
\sum^{\infty}_{k=0} \frac{(-M/2)^k}{k!} \sum_{\substack{\kappa \vdash k \\ l(\kappa) \leq N}} \frac{Z_{\kappa}(\sigma^{-1}_1, \dots, \sigma^{-1}_N) Z_{\kappa}(\lambda_1, \dots, \lambda_N)}{Z_{\kappa}(I_N)}.
\end{multline}

\item For the complex variable case
\begin{multline}
P(\lambda) = \frac{1}{C} V(\lambda)^2 \prod^N_{j=1} \lambda^{M-N}_j \\
\sum^{\infty}_{k=0} \frac{(-M)^k}{k!} \sum_{\substack{\kappa \vdash k \\ l(\kappa) \leq N}} \frac{C_{\kappa}(\sigma^{-1}_1, \dots, \sigma^{-1}_N) C_{\kappa}(\lambda_1, \dots, \lambda_N)}{C_{\kappa}(I_N)}.
\end{multline}

\item For the quaternion variable case
\begin{multline} \label{eq:quaternionic_Zonal_expanssion}
P(\lambda) = \frac{1}{C} V(\lambda)^4 \prod^N_{j=1} \lambda^{2(M-N)+1}_j \\
\sum^{\infty}_{k=0} \frac{(-2M)^k}{k!} \sum_{\substack{\kappa \vdash k \\ l(\kappa) \leq N}} \frac{Q_{\kappa}(\sigma^{-1}_1, \dots, \sigma^{-1}_N) Q_{\kappa}(\lambda_1, \dots, \lambda_N)}{Q_{\kappa}(I_N)}.
\end{multline}
\end{itemize}
\end{prop}

\begin{remark}
The Zonal, complex Zonal and quaternionic Zonal polynomials are Jack polynomials with the parameter $\alpha = 2$, $1$ and $1/2$ \cite{Macdonald95}. In particular, complex Zonal polynomials ($\alpha = 1$ Jack polynomials) are essentially Schur polynomials, see \eqref{eq:complex_Zonal_is_Schur}.
\end{remark}

Although in formulas \eqref{eq:Zonal_expanssion}--\eqref{eq:quaternionic_Zonal_expanssion} we get rid of integrals over Lie groups, the number of degree $k$ terms grows very fast as $k$ increases, since the only restriction $l(k) \leq N$ is rather weak when we consider the large $N$. For the general Wishart ensemble these formulas are still impractical. However, they can be much more powerful in spiked models.

Since in the spiked model, lots of $\sigma_j$'s are identically $1$, we can shift coordinates to make them identically $0$, and then by fact \ref{fact:variable_number_reduction}, for rank $r$ spiked model we need only consider (complex, quaternionic) Zonal polynomials with index $l(\kappa) \leq r$, since after the coordinate shift, $N-r$ variables are $0$ and the $N$-variable polynomials is equivalent to an $r$-variable one. For real variables the procedure is (eigenvalues of $S$ are $\lambda_1$, \dots, $\lambda_N$, and eigenvalues of $\Sigma$ are $\alpha_1$, \dots, $\alpha_r$, $1$, \dots, $1$)
\begin{equation}
\begin{split}
& \int_{O(N)} e^{-\frac{M}{2} \Tr(\Sigma^{-1}OSO^{-1})} dO \\
= & \int_{O(N)} e^{-\frac{M}{2} \Tr(IOSO^{-1})} e^{\frac{M}{2} \Tr((I-\Sigma^{-1})OSO^{-1})} dO \\
= & \prod^N_{j=1} e^{-\frac{M}{2} \lambda_j} \int_{O(N)} e^{\frac{M}{2} \Tr((I-\Sigma^{-1})OSO^{-1})} dO \\
= & \prod^N_{j=1} e^{-\frac{M}{2} \lambda_j} \sum^{\infty}_{k=0} \frac{(M/2)^k}{k!} \sum_{\substack{\kappa \vdash k \\ l(\kappa) \leq r}} \frac{Z_{\kappa}(\frac{\alpha_1}{1+\alpha_1}, \dots, \frac{\alpha_r}{1+\alpha_r}) Z_{\kappa}(\lambda_1, \dots, \lambda_N)}{Z_{\kappa}(I_N)},
\end{split}
\end{equation}
since $r$ eigenvalues of $I-\Sigma^{-1}$ are $\frac{\alpha_1}{1+\alpha_1}, \dots, \frac{\alpha_r}{1+\alpha_r}$ and the other $N-r$ eigenvalues are $0$. In this way we can simplify formulas for joint {\pdf}s of sample eigenvalues for the three kinds of variables in the spiked model:
\begin{prop}
Let the centralized $N$-variate normal random variable have the covariance matrix $\Sigma$ with $r$ population eigenvalues $1+\alpha_1$, \dots, $1+\alpha_r$ and the other $N-r$ population eigenvalues identically $1$, and the number of measurements be $M \geq N$. The \pdf\ of the sample eigenvalues $\lambda = (\lambda_1, \dots, \lambda_N)$ in the rank $r$ spiked model is
\begin{itemize}
\item For the real variable case
\begin{multline} \label{eq:Zonal_expanssion_spiked}
P(\lambda) = \frac{1}{C} \lvert V(\lambda) \rvert \prod^N_{j=1} \lambda^{(M-N-1)/2}_j e^{-\frac{M}{2}\lambda_j} \\
\sum^{\infty}_{k=0} \frac{(-M/2)^k}{k!} \sum_{\substack{\kappa \vdash k \\ l(\kappa) \leq r}} \frac{Z_{\kappa}(\frac{\alpha_1}{1+\alpha_1}, \dots, \frac{\alpha_r}{1+\alpha_r}) Z_{\kappa}(\lambda_1, \dots, \lambda_N)}{Z_{\kappa}(I_N)}.
\end{multline}

\item For the complex variable case
\begin{multline} \label{eq:complex_Zonal_expanssion_spiked}
P(\lambda) = \frac{1}{C} V(\lambda)^2 \prod^N_{j=1} \lambda^{M-N}_j e^{-M\lambda_j} \\
\sum^{\infty}_{k=0} \frac{(-M)^k}{k!} \sum_{\substack{\kappa \vdash k \\ l(\kappa) \leq N}} \frac{C_{\kappa}(\frac{\alpha_1}{1+\alpha_1}, \dots, \frac{\alpha_r}{1+\alpha_r}) C_{\kappa}(\lambda_1, \dots, \lambda_N)}{C_{\kappa}(I_N)}.
\end{multline}

\item For the quaternion variable case
\begin{multline} \label{eq:quaternionic_Zonal_expanssion_spiked}
P(\lambda) = \frac{1}{C} V(\lambda)^4 \prod^N_{j=1} \lambda^{2(M-N)+1}_j e^{-2M\lambda_j} \\
\sum^{\infty}_{k=0} \frac{(-2M)^k}{k!} \sum_{\substack{\kappa \vdash k \\ l(\kappa) \leq N}} \frac{Q_{\kappa}(\frac{\alpha_1}{1+\alpha_1}, \dots, \frac{\alpha_r}{1+\alpha_r}) Q_{\kappa}(\lambda_1, \dots, \lambda_N)}{Q_{\kappa}(I_N)}.
\end{multline}
\end{itemize}
\end{prop}

\section{Statement of Results}

\subsection{Complex spiked model}

To demonstrate the result, we need the language of Fredholm determinant. If $K(x, y)$ is the kernel of an integral operator from $L^2(\mathbb{R})$ to $L^2(\mathbb{R})$, then the Fredholm determinant of the integral operator, which we represent by the same notation as its kernel, is \cite{Gohberg-Krein69}
\begin{multline} \label{eq:definition_of_Fredholm_determinant}
\det \left( I - K(x,y) \right) = \\
1 - \frac{1}{1!} \int^{\infty}_{-\infty} K(x_1, x_1)dx_1 + \frac{1}{2!} \int^{\infty}_{-\infty}\int^{\infty}_{-\infty} 
\begin{vmatrix}
K(x_1, x_1) & K(x_1, x_2) \\
K(x_2, x_1) & K(x_2, x_2)
\end{vmatrix}
dx_1dx_2 - \dots \\
+ \frac{(-1)^n}{n!} \int^{\infty}_{-\infty} \dots \int^{\infty}_{-\infty} 
\begin{vmatrix}
K(x_1, x_1) & \dots & K(x_1, x_n) \\
\vdots & \dots & \vdots \\
K(x_n, x_1) & \dots & K(x_n, x_n)
\end{vmatrix}
dx_1 \dots dx_n + \dots.
\end{multline}
Then we can state the theorem for the complex spiked model \cite{Baik-Ben_Arous-Peche05}

\begin{theorem} \label{theorem:spiked_model_theorem_for_complex}
In the rank $r$ complex spiked model, let non-trivial population eigenvalues be $1+a_1 < \dots < 1+a_s$ with multiplicities respectively $r_1$, \dots, $r_s$, so that $\sum^s_{j=1} r_j = r$.
\begin{enumerate}
\item If $-1 < a_s < \gamma^{-1}$, then the distribution of the largest sample eigenvalue is the same as that of the complex white Wishart ensemble in proposition \ref{prop:definition_of_three_white_Wisharts}.

\item If $a_s = \gamma^{-1}$, then the limit and the fluctuation scale are the same as those of the complex white Wishart ensemble, but the distribution function is
\begin{equation}
\lim_{M \rightarrow \infty} \prob{P} \left( \left( \max(\lambda) - \left( \frac{\gamma+1}{\gamma} \right)^2 \right) \cdot \frac{\gamma M^{2/3}}{(\gamma+1)^{4/3}} < T \right) = F_{\GUE r_s}(T).
\end{equation}

\item If $a_s = a > \gamma^{-1}$, then the limit and the fluctuation scale are changed as well as the distribution function, which is a finite GUE distribution
\begin{multline}
\lim_{M \rightarrow \infty} \prob{P} \left( \left( \max(\lambda) - (1+a)\left( 1 + \frac{1}{\gamma^2 a} \right) \right) \cdot \frac{\sqrt{M}}{(1+a)\sqrt{1 - \frac{1}{\gamma^2 a^2}}} < T \right) = \\
G_{r_s}(T).
\end{multline}
\end{enumerate}
\end{theorem}

We need to explain the distribution function $F_{\GUE r_s}$ and $G_{r_s}$. First, let us revisit the distribution $F_{\GUE}$, which can be defined by (We abbreviate the indicator function $\chi_{[T, \infty)}(x)$ as $\chi(x)$.) \cite{Tracy-Widom94}
\begin{equation} \label{eq:determinantal_formula_of_Airy_kernel}
F_{\GUE}(T) = \det \left( 1 - \chi(\xi)K_{\Airy}(\xi, \eta)\chi(\eta) \right),
\end{equation}
where the kernel $K_{\Airy}$ is given by the Airy function
\begin{equation} \label{eq:definition_of_the_scalar_Airy_kernel}
K_{\Airy}(\xi, \eta) = \int^{\infty}_0 \Ai(x+t)\Ai(y+t) dt,
\end{equation}
with \cite{Abramowitz-Stegun64}
\begin{equation} \label{eq:definition_of_Airy_function_by_contour}
\Ai(x) = \frac{1}{2\pi} \int^{\infty e^{\pi i/6}}_{\infty e^{5\pi i/6}} e^{ixz + i\frac{z^3}{3}} dz.
\end{equation}
The integral sign $\int^{\infty e^{\pi i/6}}_{\infty e^{5\pi i/6}}$ means the integral is along an infinite arc from the direction $\infty e^{5\pi i/6}$ to the direction $\infty e^{\pi i/6}$. (We borrow the notation from the real projective geometry on $\mathbb{R}P^2$.) It is easy to see that as $x \rightarrow +\infty$, $\lvert \Ai(x) \rvert \rightarrow 0$ faster than any exponential decay. The equivalency of the Fredholm determinant representation and the formula \eqref{eq:analytic_formula_of_F_GUE} is established by Tracy and Widom.

$F_{\GUE t}$ with $t = 1, 2, \dots$ are variations of $F_{\GUE}$, and are defined as \cite{Baik-Ben_Arous-Peche05}
\begin{equation}
F_{\GUE t}(T) = \det \left( 1 - \chi(\xi) \left( K_{\Airy}(\xi, \eta) + \sum^t_{j=1} t^{(j)}(\xi) s^{(j)}(\eta) \right) \chi(\eta) \right),
\end{equation}
where $s^{(j)}$ and $t^{(j)}$ are variations of the Airy function, with
\begin{align}
s^{(j)}(x) = & \frac{1}{2\pi} \int^{\infty e^{\pi i/6}}_{\infty e^{5\pi i/6}} e^{ixz + i\frac{z^3}{3}} \frac{1}{(iz)^j} dz, \\
t^{(j)}(x) = & \frac{1}{2\pi} \int^{\infty e^{\pi i/6}}_{\infty e^{5\pi i/6}} e^{ixz + i\frac{z^3}{3}} (-iz)^{j-1} dz,
\end{align}
and we require that the infinite arc for the integral of $s^{(j)}$ is below the pole $z = 0$. We can prove that as $x \rightarrow +\infty$, $\lvert t^{(j)}(x) \rvert \rightarrow 0$ faster than any exponential decay, and $s^{(j)}$ grows slower than any exponential growth. Especially,
\begin{align}
s^{(1)}(x) = & 1 - \int^{\infty}_x \Ai(t) dt, \label{eq:expression_of_s_1} \\
t^{(1)}(x) = & \Ai(x). \label{eq:expression_of_t_1}
\end{align}

$G_t$ with $t = 1, 2, \dots$ are defined as
\begin{equation}
G_t(T) = \det \left( 1 - \chi(\xi) \left( \sum^{t-1}_{j=0} \frac{1}{j! \sqrt{2\pi}} H_j(\xi)e^{-\frac{\xi^2}{4}} H_j(\eta)e^{-\frac{\eta^2}{4}} \right) \chi(\eta) \right),
\end{equation}
where $H_j$'s are Hermite polynomials \cite{Szego75}, with $\deg(H_j) = j$ and
\begin{equation}
\int^{\infty}_{-\infty} H_i(x)H_j(x) e^{-\frac{x^2}{2}} dx = j! \sqrt{2\pi} \delta_{ij}.
\end{equation}

\begin{remark}
Since $H_0(x) = 1$ is a constant, we can easily see that $G_1$ is the Gaussian distribution 
\begin{equation} \label{eq:G_1_is_Gaussian}
G_1(T) = \int^T_{-\infty} \frac{1}{\sqrt{2\pi}} e^{-\frac{t^2}{2}} dt.
\end{equation}
\end{remark}

\begin{remark}
We call $F_{\GUE}$, $F_{\GUE t}$ and $G_t$ distribution functions, because they are monotone increasing, as $T \rightarrow +\infty$, the values of these functions approach $1$, and as $T \rightarrow -\infty$, the values approach $0$. Their monotonicity is ensured by their definitions. In the $T \rightarrow +\infty$ direction, we can easily verify the limit property from the determinantal representation. However, in the $T \rightarrow -\infty$ direction, it is a more serious problem.

Since $G_t$ is defined by a finite rank kernel with relatively simple functions, we can verify $\lim_{T \rightarrow -\infty} G_t = 0$ by direct calculation. For $F_{\GUE}$, the analytic formula \eqref{eq:analytic_formula_of_F_GUE}, whose derivation from the determinantal formula \eqref{eq:determinantal_formula_of_Airy_kernel} is highly non-trivial, yields the limit result as $T \rightarrow -\infty$. For $F_{\GUE t}$, we refer to Baik's result \cite{Baik06}.
\end{remark}

\subsection{Rank $1$ quaternionic spiked model}

To state the result for the rank $1$ quaternionic spiked model, we need the definition of Fredholm determinant for a matrix integral operator. If $K$ is an matrix integral operator from $L^2(\mathbb{R})^n$ to $L^2(\mathbb{R})^n$, i.e., $K = (K_{ij}(x,y))_{1 \leq i,j \leq n}$, with $K_{ij}(x,y)$ an integral operator from $L^2(\mathbb{R})^n$ to $L^2(\mathbb{R})^n$, then
\begin{multline} \label{eq:definition_of_matrix_integral_Fredholm_det}
\det(I - K) = 1 + \sum^{\infty}_{m=1} (-1)^m \sum_{\substack{0 \leq r_j \leq m \\ \sum^n_{j=1} r_j = m}} \frac{1}{r_1! \dots r_n!} \\
\int^{\infty}_{-\infty} \dots \int^{\infty}_{-\infty} \prod^{r_1}_{j=1} dx^{(1)}_j \dots \prod^{r_n}_{j=1} dx^{(n)}_j \det \left( K_{kl}(x^{(k)}_i, x^{(l)}_j)_{\substack{1 \leq i \leq r_k \\ 1 \leq l \leq r_l}} \right)_{1 \leq k,l \leq m}.
\end{multline}
Then we can state the theorem for the rank $1$ quaternionic spiked model

\begin{theorem} \label{theorem:main_theorem_for_quater_rank_1}
In the rank $1$ quaternionic spiked model,

\begin{enumerate}
\item If $-1 < a < \gamma^{-1}$, then the distribution of the largest
sample eigenvalue is the same as that of the quaternionic white
Wishart ensemble in proposition \ref{prop:definition_of_three_white_Wisharts}.

\item If $a = \gamma^{-1}$, then the limit and the
fluctuation scale are the same as those of the quaternionic white
Wishart ensemble, but the distribution function is
\begin{equation}
\lim_{M \rightarrow \infty} \mathbb{P} \left( \left( \max(\lambda) -
\left( \frac{\gamma+1}{\gamma} \right)^2 \right) \cdot \frac{\gamma
(2M)^{2/3}} {(1+\gamma)^{4/3}} < T \right) = F_{\GSE1}(T).
\end{equation}

\item If $a > \gamma^{-1}$, then the limit and the
fluctuation scale are changed as well as the distribution function,
which is a Gaussian:
\begin{multline} \label{eq:Gaussian_distribution_for_quaternionic}
\lim_{M \rightarrow \infty} \mathbb{P} \left( \left( \max(\lambda) -
(a+1)\left( 1+\frac{1}{\gamma^2a} \right) \right) \cdot
\frac{\sqrt{2M}} {(a+1)
 \sqrt{1 - \frac{1}{\gamma^2a^2}}} < T \right) = \\
\int^T_{-\infty} \frac{1}{\sqrt{2\pi}} e^{-\frac{t^2}{2}} dt.
\end{multline}
\end{enumerate}
\end{theorem}

Here $F_{\GSE1}$ is a variation of $F_{\GSE}$, and we first give a definition of $F_{\GSE}$ by Fredholm determinant of a matrix integral operator \cite{Forrester-Nagao-Honner99}
\begin{equation} \label{eq:determinantal_definition_of_F_GSE}
F_{\GSE}(T) = \sqrt{\det(1 - \widehat{P}(\xi, \eta))},
\end{equation}
and
\begin{equation}
\widehat{P}(\xi, \eta) = \chi(\xi) \begin{pmatrix}
\widehat{S}_4(\xi, \eta) &
\widehat{SD}_4(\xi, \eta) \\
\widehat{IS}_4(\xi, \eta) & \widehat{S}_4(\eta, \xi,)
\end{pmatrix} \chi(\eta),
\end{equation}
where
\begin{align}
\widehat{S}_4(\xi, \eta) = & \frac{1}{2}K_{\Airy}(\xi, \eta) -
\frac{1}{4}\Ai(\xi)\int^{\infty}_{\eta} \Ai(t)dt, \\
\widehat{SD}_4(\xi, \eta) = & -\frac{1}{2}
\frac{\partial}{\partial\eta} K_{\Airy}(\xi, \eta) -
\frac{1}{4}\Ai(\xi)\Ai(\eta), \\
\widehat{IS}_4(\xi, \eta) = & -\frac{1}{2} \int^{\infty}_{\xi}
K_{\Airy}(t, \eta) dt + \frac{1}{4} \int^{\infty}_{\xi} \Ai(t)dt
\int^{\infty}_{\eta} \Ai(t)dt.
\end{align}
The equivalency of \eqref{eq:determinantal_definition_of_F_GSE} and the formula \eqref{eq:analytic_formula_of_F_GSE} is established in \cite{Tracy-Widom96} and \cite{Tracy-Widom05}.

Now we can define $F_{\GSE1}$ as
\begin{equation}
F_{\GSE1}(T) = \sqrt{\det(1 - \overline{\overline{P}}(\xi, \eta))},
\end{equation}
and
\begin{equation}
\overline{\overline{P}}(\xi, \eta) = \chi(\xi)
\begin{pmatrix}
\overline{\overline{S}}_4(\xi, \eta) &
\overline{\overline{SD}}_4(\xi, \eta) \\
\overline{\overline{IS}}_4(\xi, \eta) &
\overline{\overline{S}}_4(\eta, \xi,)
\end{pmatrix} \chi(\eta),
\end{equation}
where
\begin{align}
\overline{\overline{S}}_4(\xi, \eta) = & \widehat{S}_4(\xi, \eta) +
\frac{1}{2}\Ai(\xi), \\
\overline{\overline{SD}}_4(\xi, \eta) = & \widehat{SD}_4(\xi, \eta), \\
\overline{\overline{IS}}_4(\xi, \eta) = & \widehat{IS}_4(\xi, \eta)
- \frac{1}{2} \int^{\infty}_{\xi} \Ai(t)dt + \frac{1}{2}
\int^{\infty}_{\eta} \Ai(t)dt.
\end{align}

The matrix kernel $\overline{\overline{P}}(\xi, \eta)$ seems to be new in the literature. However, the distribution function $F_{\GSE 1}$ is not new:

\begin{theorem} \label{theorem:F_GSE1=F_GOE}
\begin{equation}
F_{\GSE 1}(T) = F_{\GOE}(T).
\end{equation}
\end{theorem}

\section{Structure of the thesis}

In chapter \ref{complex_spiked_model} we reproduce Baik, Ben Arous and P\'{e}ch\'{e}'s result on the limiting distribution of the largest sample eigenvalue in the complex spiked model, by the method of multiple orthogonal polynomials suggested by Bleher and Kuijlaars \cite{Bleher-Kuijlaars05}. In chapter \ref{quaternionic_spiked_model_rank_1} we use the same idea and the method of skew multiple orthogonal polynomials, to find the limiting distribution of the largest sample eigenvalue in the rank $1$ quaternionic spiked model, with the help of a combinatorial result of $\alpha = 1/2$ Jack polynomials (quaternionic Zonal polynomials) \cite{Macdonald95}.

To get the limiting distribution, we need technical results of asymptotic analysis. We put all such results involving contour integral in chapter \ref{asymptotic_analysis}, and when we prove theorems in chapter \ref{complex_spiked_model} and \ref{quaternionic_spiked_model_rank_1}, we quote the results therein.

The limiting distribution of the largest sample eigenvalue in the complex spiked model has a phase transition phenomenon, and in the rank $1$ case it is an interpolation from $F_{\GUE}$ to Gaussian via $F^2_{\GOE}$ \cite{Baik-Ben_Arous-Peche05}. In chapter \ref{interpolations} we prove theorem \ref{theorem:F_GSE1=F_GOE}, which together with theorem \ref{theorem:main_theorem_for_quater_rank_1} gives the interpolation result for the rank $1$ quaternionic spiked model, which is from $G_{\GSE}$ to Gaussian via $F_{\GOE}$. We also give conjectures for more general phase transition phenomena.

%% file: complex_spiked_model.tex
\chapter{Complex spiked model}

\label{complex_spiked_model}

In this chapter, we consider the complex Wishart ensemble unless otherwise stated.

\section{Determinantal joint \pdf\ formula}

For the complex Wishart ensemble, we have an advantage that the complex Zonal polynomial $C_{\kappa}(x)$ is the same as the Schur polynomial $s_{\kappa}(x)$ up to a constant multiple. To be precise, we have \cite{Stanley89}
\begin{equation} \label{eq:complex_Zonal_is_Schur}
C_{\kappa}(x) = \frac{k!}{H(\kappa)} s_{\kappa}(x),
\end{equation}
where $H(\kappa)$ is the product of hook lengths of $\kappa$. If $\kappa = (\kappa_1, \kappa_2, \dots)$ and $l(\kappa) = l$, then
\begin{equation}
H(\kappa) = \prod^l_{i=1} \prod^{\kappa_i}_{j=1} (\arm_{\kappa}(i,j) + \leg_{\kappa}(i,j) + 1),
\end{equation}
where
\begin{align}
\arm_{\kappa}(i,j) = & \kappa_i - j, \\
\leg_{\kappa}(i,j) = & \min \{i' \mid k_{i'} < j \}.
\end{align}

The nomenclature of $\arm_{\kappa}$ and $\leg_{\kappa}$ is most clearly shown by the Young diagram of the partition $\kappa$. For example, in the Young diagram of the partition $\kappa = (4,3,3,2,1)$ which is drawn in figure \ref{figure:A_SAMPLE_OF_YOUNG_TABLEAU}, $\arm_{\kappa}(2,1)$ is the number of squares to the right of the $1$-st square in the $2$-nd row, and $\leg_{\kappa}(2,1)$ is the number of squares below it. Usually we call $\arm_{\kappa}(i,j) + \leg_{\kappa}(i,j) + 1 = h_{\kappa}(i,j)$ the hook length of the $(i,j)$ square in the Young diagram of the partition $\kappa$.

\begin{figure}[h]
\centering
\includegraphics{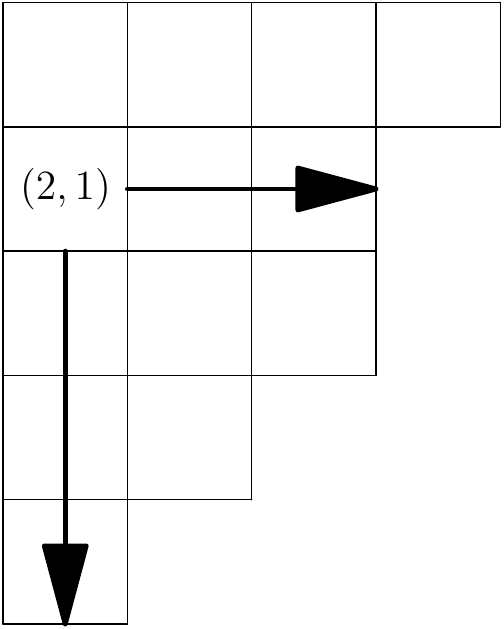}
\label{figure:A_SAMPLE_OF_YOUNG_TABLEAU}
\caption{Young tableau of partition $(4,3,3,2,1)$}
\end{figure}

Let us consider the rank $1$ spiked model first. We assume the $N$ population eigenvalues are $(N-1)$ $1$'s and the other one $1+a$. With $M$ measurements, by \eqref{eq:complex_Zonal_expanssion_spiked}, the joint \pdf\ of sample eigenvalues is
\begin{equation} \label{eq:rank_1_pdf_in_complex_case}
P(\lambda) = \frac{1}{C} V(\lambda)^2 \prod^N_{j=1} \lambda^{M-N}_j e^{-M\lambda_j} \sum^{\infty}_{k=0} \frac{M^k}{k!} \frac{C_{(k)}(\frac{a}{1+a}) C_{(k)}(\lambda_1, \dots, \lambda_N)}{C_{(k)}(I_N)}.
\end{equation}
Since we require the index $\kappa \vdash k$ of complex Zonal polynomials to satisfy $l(\kappa) \leq 1$, for any $k$, there is only one qualified partition $\kappa = (k)$. Therefore the joint \pdf\ of sample eigenvalues for the rank $1$ spiked model is much simplified.

Here we notice that $H((k)) = k!$, so that $C_{(k)}(x) = s_{(k)}(x)$. We have
\begin{align}
s_{(k)}(\frac{a}{1+a}) = & \left( \frac{a}{1+a} \right)^k, \label{eq:Schur_univariable} \\
s_{(k)}(I_N) = & \frac{(N+k-1)!}{(N-1)!k!}. \label{eq:Schur_IN}
\end{align}
Therefore
\begin{equation}
P(\lambda) = \frac{1}{C} V(\lambda)^2 \prod^N_{j=1} \lambda^{M-N}_j e^{-M\lambda_j} \sum^{\infty}_{k=0} \frac{\left( \frac{a}{1+a}M \right)^k}{(N+k-1)!} s_{(k)}(\lambda_1, \dots, \lambda_N).
\end{equation}

Schur polynomials have the well known determinantal representation \cite{Macdonald95}: For any partition $\kappa = (\kappa_1, \kappa_2, \dots)$ with $l(\kappa)=l$,
\begin{equation}
s_{\kappa}(x_1, \dots, x_N) = \frac{
\begin{vmatrix}
1 & 1 & \dots & 1 \\
x_1 & x_2 & \dots & x_N \\
\vdots & \vdots & \dots & \vdots \\
x^{N-l-1}_1 & x^{N-l-1}_2 & \dots & x^{N-l-1}_N \\
x^{N-l+\kappa_l}_1 & x^{N-l+\kappa_l}_2 & \dots & x^{N-l+\kappa_l}_N \\
\vdots & \vdots & \dots & \vdots \\
x^{N-1+\kappa_1}_1 & x^{N-1+\kappa_1}_2 & \dots & x^{N-1+\kappa_1}_N 
\end{vmatrix}}
{\begin{vmatrix}
1 & 1 & \dots & 1 \\
x_1 & x_2 & \dots & x_N \\
\vdots & \vdots & \dots & \vdots \\
x^{N-1}_1 & x^{N-1}_2 & \dots & x^{N-1}_N
\end{vmatrix}},
\end{equation}
where the denominator is the Vandermonde and the numerator is different from the Vandermonde only at the last $l$ rows, with the power of entries of the $j$-th last row increased by $\kappa_j$.

We have
\begin{equation}
s_{(k)}(\lambda_1, \dots, \lambda_N) = \frac{
\begin{vmatrix}
1 & 1 & \dots & 1 \\
\lambda_1 & \lambda_2 & \dots & \lambda_N \\
\vdots & \vdots & \dots & \vdots \\
\lambda^{N-2}_1 & \lambda^{N-2}_2 & \dots & \lambda^{N-2}_N \\
\lambda^{N-1+k}_1 & \lambda^{N-1+k}_2 & \dots & \lambda^{N-1+k}_N 
\end{vmatrix}}{V(\lambda)},
\end{equation}
and 
\begin{equation} \label{eq:raw_formula_for_Schur_sum}
\sum^{\infty}_{k=0} \frac{\left( \frac{a}{1+a}M \right)^k}{(N+k-1)!} s_{(k)}(\lambda_1, \dots, \lambda_N) = \frac{
\begin{vmatrix}
1 & 1 & \dots & 1 \\
\lambda_1 & \lambda_2 & \dots & \lambda_N \\
\vdots & \vdots & \dots & \vdots \\
\lambda^{N-2}_1 & \lambda^{N-2}_2 & \dots & \lambda^{N-2}_N \\
\vare \left( \frac{a}{1+a}M\lambda_1 \right) & \vare \left( \frac{a}{1+a}M\lambda_2 \right) & \dots & \vare \left( \frac{a}{1+a}M\lambda_N \right)
\end{vmatrix}}{V(\lambda)},
\end{equation}
where
\begin{equation}
\begin{split}
\vare \left( \frac{a}{1+a}M\lambda_j \right) = & \sum^{\infty}_{k=0} \frac{1}{(N+k-1)!} \left( \frac{a}{1+a}M \right)^k \lambda^{N+k-1}_j \\
= & \left( \frac{a}{1+a}M \right)^{-(N-1)} \left( e^{\frac{a}{1+a}M\lambda_j} - \sum^{N-2}_{k=0} \frac{1}{k!} \left( \frac{a}{1+a}M\lambda_j \right)^k \right).
\end{split}
\end{equation}
By row operations, we can change the term $\vare \left( \frac{a}{1+a}M\lambda_j \right)$ in \eqref{eq:raw_formula_for_Schur_sum} into $\left( \frac{a}{1+a}M \right)^{-(N-1)} e^{\frac{a}{1+a}M\lambda_j}$, and we have
\begin{equation} \label{eq:joint_pdf_for_r_eq_1}
\begin{split}
P(\lambda) = & \frac{1}{C} V(\lambda)^2 \prod^N_{j=1} \lambda^{M-N}_j e^{-M\lambda_j} \frac{
\begin{vmatrix}
1 & 1 & \dots & 1 \\
\lambda_1 & \lambda_2 & \dots & \lambda_N \\
\vdots & \vdots & \dots & \vdots \\
\lambda^{N-2}_1 & \lambda^{N-2}_2 & \dots & \lambda^{N-2}_N \\
e^{\frac{a}{1+a}M\lambda_1} & e^{\frac{a}{1+a}M\lambda_2} & \dots & e^{\frac{a}{1+a}M\lambda_N}
\end{vmatrix}}{V(\lambda)} \\
= & \frac{1}{C}  V(\lambda) 
\begin{vmatrix}
1 & 1 & \dots & 1 \\
\lambda_1 & \lambda_2 & \dots & \lambda_N \\
\vdots & \vdots & \dots & \vdots \\
\lambda^{N-2}_1 & \lambda^{N-2}_2 & \dots & \lambda^{N-2}_N \\
e^{\frac{a}{1+a}M\lambda_1} & e^{\frac{a}{1+a}M\lambda_2} & \dots & e^{\frac{a}{1+a}M\lambda_N}
\end{vmatrix}
\prod^N_{j=1} \lambda^{M-N}_j e^{-M\lambda_j}.
\end{split}
\end{equation}

We can get similar result for the spiked model with rank $r$ with the same idea, special values of Schur polynomials like \eqref{eq:Schur_univariable} and \eqref{eq:Schur_IN} and more laborious work:

\begin{prop} \label{prop:determinantal_representation_of_joint_pdf}
In a rank $k$ spiked model, let $r$ population eigenvalues be $1+\alpha_1$, \dots, $1+\alpha_r$ and other $N-r$ population eigenvalues be $1$. Some $\alpha_j$'s may not be distinct, but they assume values $a_1 < \dots < a_s$ with $r_1$ of them being $a_1$, \dots, $r_s$ of them being $a_s$ and $\sum^l_{j=1} r_j = r$. If we take $M$ measurements, the joint \pdf\ of sample eigenvalues is
\begin{multline} \label{eq:determinantal_representation_of_joint_pdf}
P(\lambda) = \frac{1}{C}  V(\lambda) 
\begin{vmatrix}
1 & 1 & \dots & 1 \\
\lambda_1 & \lambda_2 & \dots & \lambda_N \\
\vdots & \vdots & \dots & \vdots \\
\lambda^{N-r-1}_1 & \lambda^{N-r-1}_2 & \dots & \lambda^{N-r-1}_N \\
p_1\left( \frac{a_1}{1+a_1}M\lambda_1 \right) & p_1\left( \frac{a_1}{1+a_1}M\lambda_2 \right) & \dots & p_1\left( \frac{a_1}{1+a_1}M\lambda_N \right) \\
\vdots & \vdots & \dots & \vdots \\
p_{r_1}\left( \frac{a_1}{1+a_1}M\lambda_1 \right) & p_{r_1}\left( \frac{a_1}{1+a_1}M\lambda_2 \right) & \dots & p_{r_1}\left( \frac{a_1}{1+a_1}M\lambda_N \right) \\
\vdots & \vdots & \dots & \vdots \\
p_{r_s}\left( \frac{a_s}{1+a_s}M\lambda_1 \right) & p_{r_s}\left( \frac{a_s}{1+a_s}M\lambda_2 \right) & \dots & p_{r_s}\left( \frac{a_s}{1+a_s}M\lambda_N \right)
\end{vmatrix} \\
\prod^N_{j=1} \lambda^{M-N}_j e^{-M\lambda_j},
\end{multline}
where the determinant is similar to a Vandermonde, the only difference being in the last $r$ rows: If $\sum^{s'-1}_{i=1} r_i < r' \leq \sum^{s'}_{i=1} r_i$ and $r' - \sum^{s'-1}_{i=1} r_i = t'$, then in the $N-r+r'$ row, the entries are of the form $p_{t'}\left( \frac{a_{s'}}{1+a_{s'}}M\lambda_j \right)$, where 
\begin{equation} \label{eq:definition_of_pj}
p_j(x) = x^{j-1}e^x.
\end{equation}
\end{prop}

Later in this thesis we will denote the second determinant in \eqref{eq:determinantal_representation_of_joint_pdf} by $\varV(\lambda)$. We will not give the inductive proof of proposition \ref{prop:determinantal_representation_of_joint_pdf} like the formula \eqref{eq:joint_pdf_for_r_eq_1} for the $r=1$ case, because it can be proven much easier by the Harish-Chandra-Itzykson-Zuber (HCIZ) formula, see e.g. \cite{Mehta04}:
\begin{lemma}[HCIZ formula]
Given two $N \times N$ Hermitian matrices $X$ and $Y$, each having distinct eigenvalues $x_1$, \dots, $x_N$ and $y_1$, \dots, $y_N$, we have
\begin{equation}
\int_{U(N)} e^{\Tr(XUYU^{-1})} dU = \frac{1}{C} \frac{\det(e^{x_iy_j})_{1 \leq i,j \leq N}}{V(x)V(y)}.
\end{equation}
\end{lemma}
\begin{proof}[Sketch of proof of proposition \ref{prop:determinantal_representation_of_joint_pdf} by HCIZ formula]

By the HCIZ formula, we can simplify the joint \pdf\ \eqref{eq:pdf_of_complex_Wishart_ensemble} of generic Wishart ensemble (i.e., all population eigenvalues are distinct) as
\begin{equation} \label{eq:pdf_of_complex_Wishart_ensemble_generic}
\begin{split}
P(\lambda) = & \frac{1}{C}  V(\lambda) \det(e^{-M\sigma^{-1}_i\lambda_j})_{1 \leq i,j \leq N} \prod^N_{j=1} \lambda^{M-N}_j \\
= & \frac{1}{C(l_1, \dots, l_N)}  V(\lambda) \det(e^{Ml_i\lambda_j})_{1 \leq i,j \leq N} \prod^N_{j=1} \lambda^{M-N}_j e^{-M\lambda_j},
\end{split}
\end{equation}
where we denote $l_i = 1- \sigma^{-1}_i$, and we regard $\{ l_i \}$ as a set of parameters, so that the constant $C$ is a function of $l_i$'s. 

If $\sigma_i$'s are not distinct, which is equivalent to that $l_i$'s are not distinct, then $\det(e^{Ml_i\lambda_j})_{1 \leq i,j \leq N} = 0$, and heuristically, to make $P(\lambda)$ a \pdf\ whose total probability is $1$, $C(l_1, \dots, l_N)$ must be also $0$. In that case, formula \eqref{eq:pdf_of_complex_Wishart_ensemble_generic} becomes $\frac{0}{0}$, and we can apply l'H\^{o}pital's rule to these multiple $l_j$'s to get a reasonable formula. 

For example, if $l_1 = \dots = l_{N-r} = 0$ and $l_{N-r+1}$, \dots, $l_N$ are distinct numbers other than $0$, then we have
\begin{equation} \label{eq:pdf_of_complex_Wishart_ensemble_one_cluster}
\begin{split}
P(\lambda) = & \left. \frac{\varD \det(e^{Ml_i\lambda_j})_{1 \leq i,j \leq N}}{\varD C(l_1, \dots, l_N)} \right\rvert_{l_1 = \dots = l_{N-r} = 0} V(\lambda) \prod^N_{j=1} \lambda^{M-N}_j e^{-M\lambda_j} \\
= & \frac{
\begin{vmatrix}
1 & 1 & \dots & 1 \\
\lambda_1 & \lambda_2 & \dots & \lambda_N \\
\vdots & \vdots & \dots & \vdots \\
\lambda^{N-r-1}_1 & \lambda^{N-r-1}_2 & \dots & \lambda^{N-r-1}_N \\
e^{Ml_{N-r+1}\lambda_1} & e^{Ml_{N-r+1}\lambda_2} & \dots & e^{Ml_{N-r+1}\lambda_N} \\
\vdots & \vdots & \dots & \vdots \\
e^{Ml_N\lambda_1} & e^{Ml_N\lambda_2} & \dots & e^{Ml_N\lambda_N}
\end{vmatrix}}{\varD C(l_1, \dots, l_N)} V(\lambda) \prod^N_{j=1} \lambda^{M-N}_j e^{-M\lambda_j},
\end{split}
\end{equation}
where $\varD$ is the differential operator
\begin{equation}
\varD = \frac{\partial^{(N-r)(N-r-1)/2}}{\partial l_2 \partial l^2_3 \dots \partial l^{N-r-1}_{N-r}}.
\end{equation}
The formula \eqref{eq:pdf_of_complex_Wishart_ensemble_one_cluster} agrees with \eqref{eq:determinantal_representation_of_joint_pdf}, if we regard $\varD C(l_1, \dots, l_N)$ as the constant. For the case that $l_{N-r+1}$, \dots, $l_{N}$ are not distinct, we apply l'H\^{o}pital's rule also to them and still get \eqref{eq:determinantal_representation_of_joint_pdf}.
\end{proof}

\section{Multiple orthogonal Laguerre polynomials}

Our goal is to find the probability that the largest sample eigenvalue is less than a certain value, which is
\begin{equation} \label{eq:integral_of_pdf}
\begin{split}
\prob{P}(\max(\lambda)<T) = & \int^T_0 \dots \int^T_0 P(\lambda) d\lambda_1 \dots \lambda_N \\
= & \frac{1}{C} \int^T_0 \dots \int^T_0 V(\lambda)\varV(\lambda) \prod^N_{j=1} \lambda^{M-N}_j e^{-M\lambda_j} d\lambda_1 \dots \lambda_N.
\end{split}
\end{equation}
To evaluate the integral of determinants \eqref{eq:integral_of_pdf}, we change it into a determinant of integrals \cite{de_Bruijn55}:
\begin{lemma}[de Bruijn's] \label{lemma:de_Bruijn_for_complex}
For any two sets of functions $f_1$, \dots, $f_N$ and $g_1$, \dots, $g_N$ defined over $[a,b]$, we have
\begin{multline} \label{eq:de_Bruijn_for_complex}
\int^b_a \dots \int^b_a 
\begin{vmatrix}
f_1(x_1) & \dots & f_1(x_N) \\
\vdots & \dots & \vdots \\
f_N(x_1) & \dots & f_N(x_N)
\end{vmatrix}
\begin{vmatrix}
g_1(x_1) & \dots & g_1(x_N) \\
\vdots & \dots & \vdots \\
g_N(x_1) & \dots & g_N(x_N)
\end{vmatrix}
dx_1 \dots dx_N \\
= n! \det \left( \int^b_a f_i(x)g_j(x) dx \right)_{1 \leq i,j \leq N}.
\end{multline}
\end{lemma}
The verification of the integral formula is straightforward.

Before the application of lemma \ref{lemma:de_Bruijn_for_complex}, we are going to do some preparation to \eqref{eq:determinantal_representation_of_joint_pdf}. Let $\varphi_0$, \dots, $\varphi_{N-1}$ be polynomials of degree $0$, \dots, $N-1$ respectively, i.e., $\varphi_j(x)$ is a linear combination of functions $1$, $x$, \dots, $x^{j}$, which are functions defining the first $j+1$ rows in the Vandermonde matrix $V(\lambda)$. Similarly, let $\varvarphi_j$ ($0 \leq j \leq N-1$) be a linear combination of the $j+1$ functions which define the first $j+1$ rows in $\varV(\lambda)$, so that $\varvarphi_0$, \dots, $\varvarphi_{N-r-1}$ are polynomials but $\varvarphi_{N-r}$, \dots, $\varvarphi_{N-1}$ are not. We require that for $0 \leq i,j \leq N-1$,
\begin{equation} \label{eq:orthonormality}
\int^{\infty}_0 \varphi_i(x)\varvarphi_j(x)x^{M-N}e^{-Mx} dx = \delta_{ij}.
\end{equation}
These orthogonal conditions cannot determine $\varphi_j$ and $\varvarphi_j$ uniquely, since we can multiply a constant $C$ to $\varphi_j$ and $1/C$ to $\varvarphi_j$. For $0 \leq j \leq N-r-1$, the conditions for $\varphi_j$ and $\varvarphi_j$ are symmetric, so $\varvarphi_j = C\varphi_j$, a constant multiple of $\varphi_j$, and \eqref{eq:orthonormality} gives
\begin{equation} 
\int^{\infty}_0 C\varphi_i(x)\varphi_j(x) x^{M-N}e^{-Mx} dx = \delta_{ij},
\end{equation}
so that we can choose for arbitrary $C_j$
\begin{align}
\varphi_j(x) = & C_j \Lag{j}(Mx), \label{eq:Laguerre_form_of_varphi} \\
\varvarphi_j(x) = & \frac{j! M^{M-N+1}}{(M-N+j)! C_j} \Lag{j}(Mx), \label{eq:Laguerre_form_of_varvarphi}
\end{align}
since
\begin{equation}
\int^{\infty}_0 \Lag{i}(x)\Lag{j}(x) x^{M-N}e^{-x} dx = \delta_{ij} \frac{(M-N+j)!}{j!}.
\end{equation}

For notational simplicity, we denote the inner product of two functions $f(x)$ and $g(x)$ on $[0, \infty)$ by $\langle \cdot, \cdot \rangle_2$:
\begin{equation}
\langle f(x), g(x) \rangle_2 = \int^{\infty}_0 f(x)g(x) x^{M-N}e^{-Mx} dx.
\end{equation}

For $j \geq N-r$, we index $\varphi_j$ and $\varvarphi_j$ by a triple $(r', s', t')$, which appears in the statement of proposition \ref{prop:determinantal_representation_of_joint_pdf}, and is defined as
\begin{equation} \label{eq:definition_of_index_trituple}
r' \in \{ 1, \dots, r \}, \quad r' = \left( \sum^{s'-1}_{j=1} r_j \right) +t', \quad 1 \leq t' \leq r_{s'}.
\end{equation}
It is Bleher and Kuijlaars' observation that \cite{Bleher-Kuijlaars05}
\begin{prop}
Given $(r', s', t')$ defined as \eqref{eq:definition_of_index_trituple}, we have
\begin{multline} \label{eq:definition_of_varphi}
\varphi_{N-r+r'-1}(x) = \frac{C_{r'}}{x^{M-N}2\pi i} \\
\oint_{\Sigma} e^{Mxz} \frac{(z-1)^{N-r}}{z^{M-r+r'}} \left( \prod^{s'-1}_{j=1} \left( z - \frac{1}{1+a_j} \right)^{r_j} \right) \left( z - \frac{1}{1+a_{s'}} \right)^{t'-1} dz,
\end{multline}
and
\begin{multline} \label{eq:definition_of_varvarphi}
\varvarphi_{N-r+r'-1}(x) = \frac{Me^{Mx}}{(1+a_{s'})C_{r'}2\pi i} \\
\oint_{\Gamma} e^{-Mxz} \frac{z^{M-r+r'-1}}{(z-1)^{N-r}} \left( \prod^{s'-1}_{j=1} \left( z - \frac{1}{1+a_j} \right)^{-r_j} \right) \left( z - \frac{1}{1+a_{s'}} \right)^{-t'} dz,
\end{multline}
where $C_{r'}$ is an arbitrary constant, $\Sigma$ is a contour around $z=0$ and is to the left of the points $z=1$ and $z=\frac{1}{1+a_j}$ ($j = 1, \dots, s'$), and $\Gamma$ is a contour containing the points $z=1$ and $z=\frac{1}{1+a_j}$ ($j = 1, \dots, s'$), and is to the right of $z=0$.
\end{prop}

\begin{proof}[Sketch of proof]
By the residue theorem, we can check that $\varphi_{N-r+r'-1}(x)$ defined in \eqref{eq:definition_of_varphi} is a polynomial of degree $N-r+r'-1$, and $\varvarphi_{N-r+r'-1}(x)$ defined in \eqref{eq:definition_of_varvarphi} is also in the correct form: For any $r'$,
\begin{multline} \label{eq:expanssion_of_varvarphi}
\varvarphi_{N-r+r'-1}(x) = c_01 + c_1 x + \dots + c_{N-r-1}x^{N-r-1} \\
+ c_{N-r}p_1 \left( \frac{a_1}{1+a_1}Mx \right) + \dots + c_{N-r+r'-1}p_{t'} \left( \frac{a_{s'}}{1+a_{s'}}Mx \right),
\end{multline}
where the $j$-th term is a constant $c_j$ times the function that defines the $j$-th row in the determinant in \eqref{eq:determinantal_representation_of_joint_pdf}. Then we can check that for $k = 0, \dots, N-r-1$,
\begin{equation}
\begin{split}
& \langle \varphi_{N-r+r'-1}(x), x^k \rangle_2 \\
= & \int^{\infty}_0 x^k \frac{C_r'}{2\pi i} \oint_{\Sigma} e^{Mx(z-1)} \frac{(z-1)^{N-r}}{z^{M-r+r'}} \left( \prod^{s'-1}_{j=1} \left( z - \frac{1}{1+a_j} \right)^{r_j} \right) \left( z - \frac{1}{1+a_{s'}} \right)^{t'-1} dzdx \\
= & \frac{C_{r'}}{2\pi i} \oint_{\Sigma} \left( \int^{\infty}_0 x^ke^{Mx(z-1)} dx \right) \frac{(z-1)^{N-r}}{z^{M-r+r'}} \left( \prod^{s'-1}_{j=1} \left( z - \frac{1}{1+a_j} \right)^{r_j} \right) \left( z - \frac{1}{1+a_{s'}} \right)^{t'-1} dz \\
= & \frac{C_{r'}}{2\pi i} \oint_{\Sigma} \frac{k!}{(M(1-z))^{k+1}} \frac{(z-1)^{N-r}}{z^{M-r+r'}} \left( \prod^{s'-1}_{j=1} \left( z - \frac{1}{1+a_j} \right)^{r_j} \right) \left( z - \frac{1}{1+a_{s'}} \right)^{t'-1} dz \\
= & \frac{k!C_{r'}}{(-M)^{k+1}2\pi i} \oint_{\Sigma} (z-1)^{N-r-k-1} \left( \prod^{s'-1}_{j=1} \left( z - \frac{1}{1+a_j} \right)^{r_j} \right) \left( z - \frac{1}{1+a_{s'}} \right)^{t'-1} \frac{dz}{z^{M-r+r'}},
\end{split}
\end{equation}
which is $0$ by the residue theorem. Similarly, we can prove for $j \leq s'-1$, $i \leq r_j$ or $j=s'$, $i<t'$, using \eqref{eq:definition_of_pj}, that
\begin{equation}
\left\langle \varphi_{N-r+r'-1}(x), p_i\left( \frac{a_j}{1+a_j}Mx \right) \right\rangle_2 = 0,
\end{equation}
so that by linearity, for $0 \leq j < N-r+r'-1$, using \eqref{eq:expanssion_of_varvarphi}, we have
\begin{equation}
\langle \varphi_{N-r+r'-1}(x), \varvarphi_j(x) \rangle_2 = 0,
\end{equation}
In the same way, we can check that for $0 \leq j < N-r+r'-1$,
\begin{align}
\langle x^j, \varvarphi_{N-r+r'-1}(x) \rangle_2 = & 0, \label{eq:orthogonality_of_monomial_with_varvarphi} \\
\langle \varphi_j(x), \varvarphi_{N-r+r'-1}(x) \rangle_2 = & 0.
\end{align}

We can compute the leading term of $\varphi_{N-r+r'-1}(x)$, i.e., the $x^{N-r+r'-1}$ term, from \eqref{eq:definition_of_varphi}, which is
\begin{multline} \label{eq:leading_term_of_varphi}
\frac{C_{r'}}{x^{M-N}2\pi i} \oint_{\Sigma} \frac{(Mxz)^{M-r+r'-1}}{(M-r+r'-1)!} \frac{(z-1)^{N-r}}{z^{M-r+r'}} \\
\left( \prod^{s'-1}_{j=1} \left( z - \frac{1}{1+a_j} \right)^{r_j} \right) \left( z - \frac{1}{1+a_{s'}} \right)^{t'-1} dz = \\
(-1)^{N-r+r'-1} \frac{M^{M-r+r'-1}C_{r'}}{(M-r+r'-1)!} \left( \prod^{s'-1}_{j=1} \frac{1}{(1+a_j)^{r_j}} \right) \frac{1}{(1+a_{s'})^{t'-1}}x^{N-r+r'-1}.
\end{multline}
Therefore by \eqref{eq:orthogonality_of_monomial_with_varvarphi} and \eqref{eq:leading_term_of_varphi}, if we denote $\tilde{C}_{r'} = \left( \prod^{s'-1}_{j=1} \frac{1}{(1+a_j)^{r_j}} \right) \frac{1}{(1+a_{s'})^{t'}}$,
\begin{equation}
\begin{split}
& \langle \varphi_{N-r+r'-1}(x), \varvarphi_{N-r+r'-1}(x) \rangle_2 \\
= & \left\langle (-1)^{N-r+r'-1} \frac{(-M)^{M-r+r'-1}C_{r'}}{(M-r+r'-1)!} \tilde{C}_{r'-1} x^{N-r+r'-1}, \varvarphi_{N-r+r'-1}(x) \right\rangle_2 \\
= & (-1)^{N-r+r'-1} \frac{M^{M-r+r'}}{(M-r+r'-1)!}\frac{\tilde{C}_{r'-1}}{1+a_{s'}} \int^{\infty}_0 x^{M-r+r'-1} \\ 
& \frac{1}{2\pi i} \oint_{\Gamma} e^{-Mxz} \frac{z^{M-r+r'-1}}{(z-1)^{N-r}} \left( \prod^{s'-1}_{j=1} \left( z - \frac{1}{1+a_j} \right)^{-r_j} \right) \left( z - \frac{1}{1+a_{s'}} \right)^{-t'} dzdx \\
= & (-1)^{N-r+r'-1} \frac{M^{M-r+r'}\tilde{C}_{r'}}{(M-r+r'-1)!} \frac{1}{2\pi i} \oint_{\Gamma} \left( \int^{\infty}_0 x^{M-r+r'-1}e^{-Mxz} dx \right) \\
& \frac{z^{M-r+r'-1}}{(z-1)^{N-r}} \left( \prod^{s'-1}_{j=1} \left( z - \frac{1}{1+a_j} \right)^{-r_j} \right) \left( z - \frac{1}{1+a_{s'}} \right)^{-t'} dz \\
= & (-1)^{N-r+r'-1}\tilde{C}_{r'} \frac{1}{2\pi i} \oint_{\Gamma} \frac{1}{z(z-1)^{N-r}} \left( \prod^{s'-1}_{j=1} \left( z - \frac{1}{1+a_j} \right)^{-r_j} \right) \left( z - \frac{1}{1+a_{s'}} \right)^{-t'} dz \\
= & (-1)^{N-r+r'-1}\tilde{C}_{r'} \frac{1}{2\pi i} \left[ \oint_{\Gamma'} \frac{1}{z(z-1)^{N-r}} \left( \prod^{s'-1}_{j=1} \left( z - \frac{1}{1+a_j} \right)^{-r_j} \right) \left( z - \frac{1}{1+a_{s'}} \right)^{-t'} dz \right. \\
& \left. -\oint_{\Gamma''} \frac{1}{z(z-1)^{N-r}} \left( \prod^{s'-1}_{j=1} \left( z - \frac{1}{1+a_j} \right)^{-r_j} \right) \left( z - \frac{1}{1+a_{s'}} \right)^{-t'} dz \right],
\end{split}
\end{equation}
where $\Gamma'$ is a contour including $0$, $1$ and $\frac{1}{1+a_j}$ ($j = 1, \dots, s$), and $\Gamma''$ is a contour including $0$ and excluding $1$ and $\frac{1}{1+a_j}$. We can deform $\Gamma'$ to infinity to see that
\begin{equation}
\frac{1}{2\pi i} \oint_{\Gamma'} \frac{1}{z(z-1)^{N-r}} \left( \prod^{s'-1}_{j=1} \left( z - \frac{1}{1+a_j} \right)^{-r_j} \right) \left( z - \frac{1}{1+a_{s'}} \right)^{-t'} dz = 0,
\end{equation}
and by the residue theorem see that
\begin{equation}
\frac{1}{2\pi i} \oint_{\Gamma''} \frac{1}{z(z-1)^{N-r}} \left( \prod^{s'-1}_{j=1} \left( z - \frac{1}{1+a_j} \right)^{-r_j} \right) \left( z - \frac{1}{1+a_{s'}} \right)^{-t'} dz = \frac{(-1)^{N-r+r'}} {\tilde{C}_{r'}},
\end{equation}
and we check that $\langle \varphi_{N-r+r'-1}(x), \varvarphi_{N-r+r'-1}(x) \rangle_2 = 1$. 
\end{proof}

Although the integral formulas for $\varphi_{N-r+r'-1}(x)$ and $\varvarphi_{N-r+r'-1}(x)$ seems strange, if we compare $\varvarphi_{N-r+r'-1}(x)$ with the integral representation of Laguerre polynomials \cite{Szego75}
\begin{equation} \label{eq:first_integral_formula_of_Laguerre}
\Lag{j}(Mx) = \frac{e^{Mx}}{2\pi i} \oint_{\Gamma} e^{-Mxz} \frac{z^{M-N+j}}{(z-1)^{j+1}} dz,
\end{equation}
and $\varphi_{N-r+r'-1}(x)$ with another integral representation of Laguerre polynomials
\begin{equation} \label{eq:second_integral_formula_of_Laguerre}
\Lag{j}(Mx) = \frac{(M-N+j)!}{j!M^{M-N}} \frac{1}{x^{M-N}2\pi i} \oint_{\Sigma} e^{Mxz} \frac{(z-1)^j}{z^{M-N+j+1}} dz,
\end{equation}
and find that they are variations of Laguerre polynomials. The $\varphi_{N-r+r'-1}(x)$'s are called multiple Laguerre polynomials of type II, and the $\varvarphi_{N-r+r'-1}(x)$'s are equivalent to---although they are not polynomials---multiple Laguerre polynomials of type I, by Bleher and Kuijlaars \cite{Bleher-Kuijlaars05}.

We know that
\begin{equation}
V(\lambda) = \frac{1}{C} 
\begin{vmatrix}
\varphi_0(\lambda_1) & \dots & \varphi_0(\lambda_N) \\
\vdots & \dots & \vdots \\
\varphi_{N-1}(\lambda_1) & \dots & \varphi_{N-1}(\lambda_N)
\end{vmatrix},
\end{equation}
and
\begin{equation}
\varV(\lambda) = \frac{1}{C} 
\begin{vmatrix}
\varvarphi_0(\lambda_1) & \dots & \varvarphi_0(\lambda_N) \\
\vdots & \dots & \vdots \\
\varvarphi_{N-1}(\lambda_1) & \dots & \varvarphi_{N-1}(\lambda_N)
\end{vmatrix},
\end{equation}
so that we can write the formula \eqref{eq:determinantal_representation_of_joint_pdf} for the joint \pdf\ as
\begin{equation}
P(\lambda) = \frac{1}{C} 
\begin{vmatrix}
\varphi_0(\lambda_1) & \dots & \varphi_0(\lambda_N) \\
\vdots & \dots & \vdots \\
\varphi_{N-1}(\lambda_1) & \dots & \varphi_{N-1}(\lambda_N)
\end{vmatrix}
\begin{vmatrix}
\varvarphi_0(\lambda_1) & \dots & \varvarphi_0(\lambda_N) \\
\vdots & \dots & \vdots \\
\varvarphi_{N-1}(\lambda_1) & \dots & \varvarphi_{N-1}(\lambda_N)
\end{vmatrix}
\prod^N_{j=1} \lambda^{M-N}_j e^{-M\lambda_j},
\end{equation}
and by formulas \eqref{eq:integral_of_pdf} and \eqref{eq:de_Bruijn_for_complex}, we have
\begin{equation} \label{eq:intermediate_determinantal_formula}
\prob{P}(\max(\lambda) < T) = \frac{1}{C} \det \left( \int^T_0 \varphi_i(x)\varvarphi_j(x) x^{M-N}e^{-Mx} dx \right)_{0 \leq i,j \leq N-1},
\end{equation}
where we choose the $f_i$ and $g_j$ in \eqref{eq:de_Bruijn_for_complex} to be
\begin{align}
f_i(x) = & \varphi_{i-1}(x) x^{\frac{M-N}{2}}e^{-\frac{M}{2}x}, \\
g_j(x) = & \varvarphi_{j-1}(x) x^{\frac{M-N}{2}}e^{-\frac{M}{2}x}.
\end{align}

By \eqref{eq:orthonormality}, we have
\begin{equation}
\int^T_0 \varphi_i(x)\varvarphi_j(x) x^{M-N}e^{-Mx} dx = \delta_{ij} - \int^{\infty}_T \varphi_i(x)\varvarphi_j(x) x^{M-N}e^{-Mx} dx,
\end{equation}
and so
\begin{multline} \label{eq:identity_of_determinants}
\det \left( \int^T_0 \varphi_i(x)\varvarphi_j(x) x^{M-N}e^{-Mx} dx \right)_{0 \leq i,j \leq N-1} = \\
\det \left( I_N - \left( \int^{\infty}_T \varphi_i(x)\varvarphi_j(x) x^{M-N}e^{-Mx} dx \right)_{0 \leq i,j \leq N-1} \right).
\end{multline}
To evaluate the determinant in \eqref{eq:identity_of_determinants}, we apply the formula
\begin{multline} \label{eq:changed_determinantal_formula_for_complex}
\det \left( I_N - \left( \int^b_a f_i(x)g_j(x) dx \right)_{1 \leq i,j \leq N} \right) = \\
\det \left( I - \chi_{[a,b]}(x) \left( \sum^N_{j=1} f_j(x)g_j(y) \right) \chi_{[a,b]}(y) \right),
\end{multline}
where $\chi_{[a,b]}$ is the indicator function for the interval $[a,b]$, and the latter determinant is the Fredholm determinant for integral operator. Since the integral kernel on the left hand side of \eqref{eq:changed_determinantal_formula_for_complex} is of finite rank, we can check the identity directly. Consider the parametrized determinant
\begin{equation} \label{eq:graded_determinant_formula}
\det \left( I_N - t\left( \int^b_a f_i(x)g_j(x) dx \right)_{1 \leq i,j \leq N} \right) = 1+ c_1t + c_2t^2 + \dots + c_Nt^N,
\end{equation}
we can compute the coefficients degree by degree ($K(x,y) = \sum^N_{j=1} f_j(x)g_j(y)$):
\begin{align}
c_1 = & -\sum^N_{j=1} \int^b_a f_j(x)g_j(x) dx = -\int^b_a K(x_1, x_1) dx_1, \\
c_2 = & \sum^N_{i,j=1} \left( \int^b_a f_i(x)g_i(x) dx \int^b_a f_j(x)g_j(x) dx - \int^b_a f_i(x)g_j(x) dx \int^b_a f_j(x)g_i(x) dx \right) \notag \\
= & \frac{1}{2!} \int^b_a\int^b_a 
\begin{vmatrix}
K(x_1, x_1) & K(x_1, x_2) \\
K(x_2, x_1) & K(x_2, x_2)
\end{vmatrix} dx_1dx_2, \\
& \dots \dots \dots \notag
\end{align}
and we find the right-hand side of \eqref{eq:graded_determinant_formula} is the same as the right-hand side of \eqref{eq:definition_of_Fredholm_determinant} for $t=1$.

Now using \eqref{eq:intermediate_determinantal_formula}, \eqref{eq:identity_of_determinants} and \eqref{eq:changed_determinantal_formula_for_complex} we get the Fredholm determinantal formula for the probability that the largest eigenvalue is less than $T$: ($\chi(x)$ means $\chi_{[T,\infty)}(x)$)
\begin{multline}
\prob{P}(\max(\lambda) < T) = \\
\frac{1}{C} \det \left( I - \chi(x) \left( \sum^{N-1}_0 \varphi_j(x)\varvarphi_j(y)x^{\frac{M-N}{2}}y^{\frac{M-N}{2}}e^{-\frac{M}{2}(x+y)} \right) \chi(x) \right).
\end{multline}
Furthermore, we can determine the constant $C=1$, since letting $T \rightarrow \infty$, $\prob{P}(\max(\lambda) < T) \rightarrow 1$ and the Fredholm determinant on the right hand side also approaches $1$. To undertake real calculations, we need methods to evaluate of the Fredholm determinant other than the formula \eqref{eq:definition_of_Fredholm_determinant}. First, we have the result on taking the limit, see, e.g. \cite{Lax02}:
\begin{prop} \label{prop:trace_norm_convergence_implies_det_convergence}
If a series of integral operators $K_n$ approaches $K$ in trace norm, then
\begin{equation}
\lim_{n \rightarrow \infty} \det(I - K_n) = \det(I - K).
\end{equation}
\end{prop}

Since as $M \rightarrow \infty$, we expect that the fluctuation scale of the largest eigenvalue shrinks, depending on $M$. We will consider for any $M$ the probability $\prob{P}(\max(\lambda) < p+qT)$, where $p$ and $q$ may depend on $M$. (It turns out that $p$ is a constant.) If we denote
\begin{align}
K_{2a}(x,y) = & \sum^{N-r-1}_{j=0} \varphi_j(x)\varvarphi_j(y)x^{\frac{M-N}{2}}y^{\frac{M-N}{2}}e^{-\frac{M}{2}(x+y)}, \label{eq:raw_form_of_K_2a} \\
K_{2b}(x,y) = & \sum^{N-1}_{j=N-r} \varphi_j(x)\varvarphi_j(y)x^{\frac{M-N}{2}}y^{\frac{M-N}{2}}e^{-\frac{M}{2}(x+y)}, \label{eq:raw_form_of_K_2b} \\
K_2(x,y) = & K_{2a}(x,y) + K_{2b}(x,y), \\
\intertext{and}
\varK_*(\xi, \eta) = & qK_*(p+q\xi,p+q\eta), \label{eq:relation_of_original_and_scaled_kernels}
\end{align}
where $*$ stands for $2$, $2a$ or $2b$, then we have the determinantal formula
\begin{equation}
\begin{split}
\prob{P}(\max(\lambda) < p+qT) = & \det \left( I - \chi_{[p+qT, \infty)}(x) K_2(x,y) \chi_{[p+qT, \infty)}(y) \right) \\ 
= & \det \left( I - \chi(\xi) \varK_2(\xi,\eta) \chi(\eta) \right).
\end{split}
\end{equation}

$K_{2a}(x,y)$, which is composed of Laguerre polynomials, is the kernel for the LUE, and we can write it in an integral form:
\begin{theorem}
\begin{multline} \label{eq:integral_form_of_Airy_kernel}
K_{2a}(x,y) = -M^2 \frac{y^{\frac{M-N}{2}} e^{\frac{M}{2}y}}{x^{\frac{M-N}{2}} e^{\frac{M}{2}x}} \int^{\infty}_0 \left( \frac{1}{2\pi i}\oint_{\Gamma} e^{-M(y+t)z} \frac{z^{M-r}}{(z-1)^{N-r}} dz \right) \\
\left( \frac{1}{2\pi i}\oint_{\Sigma} e^{M(x+t)w} \frac{(w-1)^{N-r}}{w^{M-r}} dw \right) dt.
\end{multline}
\end{theorem}

\begin{proof}
Because of \eqref{eq:Laguerre_form_of_varphi}, \eqref{eq:Laguerre_form_of_varvarphi}, \eqref{eq:first_integral_formula_of_Laguerre}, \eqref{eq:second_integral_formula_of_Laguerre} and \eqref{eq:raw_form_of_K_2a}, we have the integral representation
\begin{equation} \label{eq:raw_formula_for_involution}
\begin{split}
& K_{2a}(x,y) \\
= & \sum^{N-r-1}_{j=0} \frac{j!}{(M-N+j)!}M^{M-N+1} \Lag{j}(Mx)\Lag{j}(My) x^{\frac{M-N}{2}} y^{\frac{M-N}{2}} e^{-\frac{M}{2}(x+y)} \\
= & \frac{M}{(2\pi i)^2} \frac{y^{\frac{M-N}{2}} e^{\frac{M}{2}y}}{x^{\frac{M-N}{2}} e^{\frac{M}{2}x}} \sum^{N-r-1}_{j=0}  \oint_{\Gamma}dz\oint_{\Sigma}dw e^{-Myz} \frac{z^{M-N+j}}{(z-1)^{j+1}} e^{Mxw} \frac{(w-1)^j}{w^{M-N+j+1}}.
\end{split}
\end{equation}
We can write the sum of integrands in \eqref{eq:raw_formula_for_involution} as
\begin{equation}
\begin{split}
& \sum^{N-r-1}_{j=0} e^{Mxz} \frac{z^{M-N+j}}{(z-1)^{j+1}} e^{-Myw} \frac{(w-1)^j}{w^{M-N+j+1}} \\
= & e^{Mxz}e^{-Myw} \frac{z^{M-N}}{w^{M-N}} \frac{1}{(z-1)w} \sum^{N-r-1}_{j=0} \left( \frac{z(w-1)}{(z-1)w} \right)^j \\
= & e^{Mxz}e^{-Myw} \frac{z^{M-N}}{w^{M-N}} \frac{1}{(z-1)w} \frac{1 - \left( \frac{z(w-1)}{(z-1)w} \right)^{N-r}}{1 - \frac{z(w-1)}{(z-1)w}} \\
= & \frac{1}{z-w} e^{Mxz}z^{M-N} e^{-Myw}\frac{1}{w^{M-N}} \\
& - \frac{1}{z-w} e^{Mxz} \frac{z^{M-r}}{(z-1)^{N-r}} e^{-Myw}\frac{(w-1)^{N-r}}{w^{M-r}}.
\end{split}
\end{equation}
By residue theorem, since $\Gamma$ and $\Sigma$ are disjoint, and for the variable $z$, the pole $z=w$ is outside of $\Gamma$,
\begin{equation}
\oint_{\Gamma}dz\oint_{\Sigma}dw \frac{1}{z-w} e^{Mxz}z^{M-N} e^{-Myw}\frac{1}{w^{M-N}} = 0.
\end{equation}
On the other side, since $\Re(w-z)$ is always less than $0$,
\begin{equation}
\frac{1}{z-w} = M \int^{\infty}_0 e^{tM(w-z)} dt,
\end{equation}
so that we have
\begin{equation} \label{eq:last_formula_for_involution}
\begin{split}
& \frac{M}{(2\pi i)^2} \oint_{\Gamma}dz\oint_{\Sigma}dw \frac{1}{z-w} e^{-Myz} \frac{z^{M-r}}{(z-1)^{N-r}} e^{Mxw}\frac{(w-1)^{N-r}}{w^{M-r}} \\
= & \frac{M^2}{(2\pi i)^2} \oint_{\Gamma}dz\oint_{\Sigma}dw \int^{\infty}_0  e^{-M(y+t)z} \frac{z^{M-r}}{(z-1)^{N-r}} e^{M(x+t)w}\frac{(w-1)^{N-r}}{w^{M-r}} \\
= & M^2 \int^{\infty}_0\left( \frac{1}{2\pi i} \oint_{\Gamma} e^{-M(y+t)z} \frac{z^{M-r}}{(z-1)^{N-r}} dz \right) \left( \frac{1}{2\pi i} \oint_{\Sigma} e^{M(x+t)w}\frac{(w-1)^{N-r}}{w^{M-r}} dw \right) dt.
\end{split}
\end{equation}
Put \eqref{eq:raw_formula_for_involution}--\eqref{eq:last_formula_for_involution} together, we get the result.
\end{proof}

Now $K_{2a}(x,y)$ ($\varK_{2a}(\xi,\eta)$) is expressed by two functions and $K_{2b}(x,y)$ ($\varK_{2b}(\xi,\eta)$) is a finite rank operator. To undertake the asymptotic analysis, we need two propositions on the convergence in trace norm:

\begin{prop} \label{prop:first_trace_norm_convergence_prop}
Let $\{ f_{j,n} \}$ and $\{ g_{j,n} \}$ ($1 \leq j \leq m$, $1 \leq n < \infty$) be $2m$ series of functions on $[T, \infty)$ and $f_{j,n} \rightarrow f_j$, $g_{j,n} \rightarrow g_j$ in $L^2$ norm on $[T, \infty)$. We have the convergence in trace norm of operators ($K_n(x,y) = \sum^m_{j=1} f_{j,n}(x) g_{j,n}(x)$, $K(x,y) = \sum^m_{j=1} f_j(x) g_j(x)$):
\begin{equation}
\lim_{n \rightarrow \infty} \chi(x)K_n(x,y) \chi(y) = \chi(x)K(x,y) \chi(y).
\end{equation}
\end{prop}

\begin{prop} \label{prop:second_trace_norm_convergence_prop}
Let $f_n(x)$, $g_n(y)$ be two series of functions on $[T, \infty)$ and in $L^2$ norm we have
\begin{align}
\lim_{n \rightarrow \infty} \lVert (f_n(x) -f(x))(x-T) \rVert_{L^2([T, \infty))} & = 0, \label{eq:first_fubini} \\
\lim_{n \rightarrow \infty} \lVert (g_n(y) -g(y))(y-T) \rVert_{L^2([T, \infty))} & = 0, \label{eq:second_fubini}
\end{align}
then we have the convergence in trace norm of operators ($K_n(x,y) = \int^{\infty}_0 f_n(x+t)g_n(y+t) dt$, $K(x,y) = \int^{\infty}_0 f(x+t)g(y+t) dt$):
\begin{equation}
\lim_{n \rightarrow \infty} \chi(x)K_n(x,y) \chi(y) = \chi(x)K(x,y) \chi(y).
\end{equation}
\end{prop}

We can verify proposition \ref{prop:first_trace_norm_convergence_prop} by the definition of trace norm. Proposition \ref{prop:second_trace_norm_convergence_prop} is essentially a fact of trace norm convergence \cite{Lax02}: If $I_n$, $J_n$ are Hilbert-Schmidt operators and $I_n \rightarrow I$, $J_n \rightarrow J$ in the Hilbert-Schmidt norm, then the product $I_nJ_n$ and $IJ$ are trace class operators and $I_nJ_n \rightarrow IJ$ in trace norm. Her we take $I_n$ as the integral operator from $L^2(\mathbb{R})$ to $L^2(\mathbb{R})$ with the kernel $\chi(x) f_n(x+y) \chi_{[0, \infty)}(y)$, and $J_n$ also an integral operator from $L^2(\mathbb{R})$ to $L^2(\mathbb{R})$ with the kernel $\chi_{[0, \infty)}(x) g_n(x+y) \chi(y)$. For an integral operator, the Hilbert-Schmidt norm is equivalent to the $L^2$ norm of its kernel as a $2$ variable function. In our special case, the convergence results of $L^2(\mathbb{R}^2)$ is equivalent to the convergence of $L^2([T, \infty))$ in \eqref{eq:first_fubini} and \eqref{eq:second_fubini}, due to the Fubini's theorem.

To apply propositions \ref{prop:first_trace_norm_convergence_prop} and \ref{prop:second_trace_norm_convergence_prop}, we sometimes need to conjugate the integral operator by a weight function:
\begin{prop}  \label{prop:conjugation_formula_for_scalar_kernel}
We have
\begin{equation} \label{eq:conjugation_formula_for_scalar_kernel}
\det(1 - K(x,y)) = \det(1 - f(x)K(x,y)f^{-1}(y)),
\end{equation}
for any function $f$ which make the right hand side of \eqref{eq:conjugation_formula_for_scalar_kernel} well defined.
\end{prop}
\begin{proof}
A direct application of the definition of Fredholm determinant \eqref{eq:definition_of_Fredholm_determinant}.
\end{proof}

If we take $f(x) = x^{(M-N)/2}e^{\frac{1-\gamma}{2(\gamma+1)}Mx}$, then by \eqref{eq:definition_of_varphi}, \eqref{eq:definition_of_varvarphi} and \eqref{eq:raw_form_of_K_2b} \eqref{eq:integral_form_of_Airy_kernel} we have
\begin{align}
f(x)K_{2a}(x,y)f^{-1}(y) = & -M^2 \int^{\infty}_0 e^{-\frac{\gamma}{\gamma+1}M(x+t)}\psi(x+t) e^{\frac{\gamma}{\gamma+1}M(y+t)}\varpsi(y+t) dt, \label{eq:conjugated_K_2a} \\
f(x)K_{2b}(x,y)f^{-1}(y) = & M\sum^r_{r'=1} \frac{1}{1+a_{s'}} e^{-\frac{\gamma}{\gamma+1}M(x+t)}\psi_{r'}(x) e^{\frac{\gamma}{\gamma+1}M(y+t)}\varpsi_{r'}(y), \label{eq:conjugated_K_2b}
\end{align}
where
\begin{align}
\psi(x) = & \frac{1}{2\pi i} \oint_{\Sigma} e^{Mxz} \frac{(z-1)^{N-r}}{z^{M-r}} dz, \label{eq:final_integral_formula_of_psi} \\
\varpsi(y) = & \frac{1}{2\pi i}\oint_{\Gamma} e^{-Myz} \frac{z^{M-r}}{(z-1)^{N-r}} dz, \label{eq:final_integral_formula_of_varpsi} \\
\psi_{r'}(x) = & \frac{1}{2\pi i}\oint_{\Sigma} e^{Mxz} \frac{(z-1)^{N-r}}{z^{M-r+r'}} \left( \prod^{s'-1}_{j=1} \left( z - \frac{1}{1+a_j} \right)^{r_j} \right) \left( z - \frac{1}{1+a_{s'}} \right)^{t'-1} dz, \label{eq:definition_of_phi_rprime} \\
\varpsi_{r'}(y) = & \frac{1}{2\pi i} \oint_{\Gamma} e^{-Mxz} \frac{z^{M-r+r'-1}}{(z-1)^{N-r}} \left( \prod^{s'-1}_{j=1} \left( z - \frac{1}{1+a_j} \right)^{-r_j} \right) \left( z - \frac{1}{1+a_{s'}} \right)^{-t'} dz. \label{eq:definition_of_varphi_rprime}
\end{align}

\section{Proof of theorem \ref{theorem:spiked_model_theorem_for_complex}}

We give proofs of all the three parts separately. In this section, we let $x = p+q\xi$ and $y = q+q\eta$.

\subsection{The $-1 < a_s < \gamma^{-1}$ part} \label{the max_lambda_<_gamma^-1_part_complex}

In this case, we choose $p = (1+\gamma^{-1})^2$ and $q = \frac{(\gamma+1)^{4/3}}{\gamma M^{2/3}}$, and by \eqref{eq:relation_of_original_and_scaled_kernels}, write \eqref{eq:conjugated_K_2a} as
\begin{equation}
f(x)\varK_{2a}(\xi, \eta)f^{-1}(y) = -\int^{\infty}_0 \Psi(\xi + t)\tildePsi(\eta + t) dt,
\end{equation}
where 
\begin{align}
\Psi(\xi) = & \frac{(\gamma+1)^{4/3}}{\gamma}M^{1/3}(-1)^N\frac{\gamma^M}{(\gamma+1)^{M-N}} e^{-\frac{\gamma}{\gamma+1}Mx}\psi(p+q\xi), \\
\tildePsi(\eta) = & \frac{(\gamma+1)^{4/3}}{\gamma}M^{1/3}(-1)^N\frac{(\gamma+1)^{M-N}}{\gamma^M} e^{\frac{\gamma}{\gamma+1}My}\varpsi(p+q\eta).
\end{align}
Since \eqref{eq:asymptotics_of_psi_less} and \eqref{eq:asymptotics_of_varpsi_less} give the $L^2$ convergence results
\begin{align}
\lim_{M \rightarrow \infty} & \left\lVert (\Psi(\xi) - (-\gamma)^r\Ai(\xi))(\xi-T) \right\rVert_{L^2([T, \infty))} = 0, \label{eq:L2_convergence_for_Psi} \\
\lim_{M \rightarrow \infty} & \left\lVert (\tildePsi(\eta) - (-1)^{r-1}\gamma^{-r}\Ai(\eta))(\eta-T) \right\rVert_{L^2([T, \infty))} = 0, \label{eq:L2_convergence_for_tildePsi}
\end{align}
we get the convergence in trace norm by proposition \ref{prop:second_trace_norm_convergence_prop}
\begin{equation} \label{eq:trace_norm_convergence_of_K2a_less}
\begin{split}
& \lim_{M \rightarrow \infty} \chi(\xi)f(x)\varK_{2a}(\xi, \eta)f^{-1}(y)\chi(\eta) \\
= & -\chi(\xi) \int^{\infty}_0 (-\gamma)^r \Ai(\xi + t) (-1)^{r-1}\gamma^{-r} \Ai(\eta + t) dt \chi(\eta) \\
= & \chi(\xi) K_{\Airy}(\xi, \eta) \chi(\eta).
\end{split}
\end{equation}

Similarly, for any $r' = 1, \dots, r$, by \eqref{eq:conjugated_K_2b} and \eqref{eq:relation_of_original_and_scaled_kernels}, 
\begin{equation}
f(x)\varK_{2b}(\xi, \eta)f^{-1}(y) = \frac{\gamma}{(\gamma+1)^{4/3}M^{1/3}} \sum^r_{r'=1} \frac{1}{1+a_{s'}} \Psi_{r'}(\xi)\tildePsi_{r'}(\eta),
\end{equation}
where
\begin{align}
\Psi_{r'}(\xi) = & \frac{(\gamma+1)^{4/3}}{\gamma}M^{1/3}(-1)^N\frac{\gamma^M}{(\gamma+1)^{M-N}} e^{-\frac{\gamma}{\gamma+1}Mx}\psi_{r'}(p+q\xi), \\
\tildePsi_{r'}(\eta) = & \frac{(\gamma+1)^{4/3}}{\gamma}M^{1/3}(-1)^N\frac{(\gamma+1)^{M-N}}{\gamma^M} e^{\frac{\gamma}{\gamma+1}My}\psi_{r'}(p+q\eta).
\end{align}
Because \eqref{eq:asymptotics_of_psi_rprime_less} and \eqref{eq:asymptotics_of_varpsi_rprime_less} imply the $L^2$ convergence results ($\bar{C}_{r'}$ is defined in \eqref{eq:definition_of_Cbar_rprime})
\begin{align}
\lim_{M \rightarrow \infty} \Psi_{r'}(\xi)\chi(\xi) = & (-1)^r\gamma^{r-r'}(\gamma+1)^{r'}\bar{C}_{r'-1} \Ai(\xi)\chi(\xi), \\
\lim_{M \rightarrow \infty} \tildePsi_{r'}(\eta)\chi(\eta) = & \frac{(-1)^r(1+\gamma^{-1})}{\gamma^{r-r'}(\gamma+1)^{r'}\bar{C}_{r'}} \Ai(\eta)\chi(\eta), \\
\end{align}
we get by proposition \ref{prop:first_trace_norm_convergence_prop} the boundedness in trace norm ($C$ is a large enough positive constant)
\begin{equation}
\left\lVert \chi(\xi) \left(\sum^r_{r'=1} \frac{1}{1+a_{s'}} \Psi_{r'}(\xi)\tildePsi_{r'}(\eta) \right) \chi(\eta) \right\rVert_{\Tr} < C,
\end{equation}
so that in trace norm
\begin{equation} \label{eq:trace_norm_convergence_of_K2b_less}
\lim_{M \rightarrow \infty} \chi(\xi)f(x)\varK_{2b}(\xi, \eta)f^{-1}(y)\chi(\eta) = 0.
\end{equation}
Therefore,
\begin{equation}
\begin{split}
& \lim_{M \rightarrow \infty} \det \left( 1 - \chi(\xi) \varK(\xi, \eta) \chi(\eta) \right) \\
= & \lim_{M \rightarrow \infty} \det \left( 1 - \chi(\xi) f(x)\varK(\xi, \eta) \chi(\eta)f^{-1}(y) \right) \\
= & \lim_{M \rightarrow \infty} \det \left( 1 - \chi(\xi) f(x)\varK_{2a}(\xi, \eta)f^{-1}(y) \chi(\eta) \right) \\
= & \det \left( 1 - \chi(\xi) K_{\Airy}(\xi, \eta) \chi(\eta) \right).
\end{split}
\end{equation}

\subsection{The $a_s = \gamma^{-1}$ part} \label{the max_lambda_=_gamma^-1_part_complex}

In this case we still choose the same $p$ and $q$ as in the previous case, but we need to conjugate the kernel not only by $f(x)$, but also $e^{\xi/3}$. If we consider
\begin{equation}
f(x)e^{\xi/3} \varK_{2a}(\xi, \eta) f^{-1}(y)e^{-\eta/3} = - \int^{\infty}_0 e^{(\xi+t)/3}\Psi(\xi+t)e^{-(\eta+t)/3}\tildePsi(\eta+t) dt,
\end{equation}
and similar to \eqref{eq:L2_convergence_for_Psi} and \eqref{eq:L2_convergence_for_tildePsi}, \eqref{eq:asymptotics_of_psi_less} and \eqref{eq:asymptotics_of_varpsi_less} imply also the $L^2$ convergence results
\begin{align}
\lim_{M \rightarrow \infty} & \left\lVert (e^{\xi/3}\Psi(\xi) - (-\gamma)^re^{\xi/3}\Ai(\xi))(\xi-T) \right\rVert_{L^2([T, \infty))} = 0, \\
\lim_{M \rightarrow \infty} & \left\lVert (e^{-\eta/3}\tildePsi(\eta) - (-1)^{r-1}\gamma^{-r}e^{-\eta/3}\Ai(\eta))(\eta-T) \right\rVert_{L^2([T, \infty))} = 0, 
\end{align}
and we can get the result of trace norm convergence similar to \eqref{eq:trace_norm_convergence_of_K2a_less}
\begin{equation}
\lim_{M \rightarrow \infty} \chi(\xi)f(x)e^{\xi/3}\varK_{2a}(\xi, \eta)f^{-1}(y)e^{-\eta/3}\chi(\eta) 
= \chi(\xi)e^{\xi/3} K_{\Airy}(\xi, \eta) e^{-\eta/3}\chi(\eta).
\end{equation}
we can still get by \eqref{eq:asymptotics_of_psi_less}, \eqref{eq:asymptotics_of_varpsi_less} and proposition \ref{prop:second_trace_norm_convergence_prop} the trace norm convergence
\begin{equation}
\begin{split}
& \lim_{M \rightarrow \infty} \chi(\xi)f(x)e^{\xi/3} \varK_{2a}(\xi, \eta)f^{-1}(y) e^{-\eta/3}\chi(\eta) \\
= & \chi(\xi) \int^{\infty}_0 e^{(\xi+t)/3}\Ai(\xi + t)e^{-(\eta+t)/3}\Ai(\eta + t) dt \chi(\eta) \\
= & \chi(\xi)e^{\xi/3} K_{\Airy}(\xi, \eta) e^{-\eta/3}\chi(\eta).
\end{split}
\end{equation}

We can write the formula \eqref{eq:conjugated_K_2b} as
\begin{multline} \label{eq:conjugated_K_2b_equal}
f(x)e^{\xi/3}\varK_{2b}(\xi, \eta)f^{-1}(y)e^{-\eta/3} = \frac{\gamma}{(\gamma+1)^{4/3}M^{1/3}} \sum^{r-r_s}_{r'=1} \frac{1}{1+a_{s'}} e^{\xi/3}\Psi_{r'}(\xi)e^{-\eta/3}\tildePsi_{r'}(\eta) \\
+ \sum^{r}_{r'=r-r_s+1} \frac{1}{1+\gamma^{-1}} e^{\xi/3}\Psi_{t',r'}(\xi)e^{-\eta/3}\tildePsi_{t',r'}(\eta),
\end{multline}
where
\begin{align}
\Psi_{t',r'}(\xi) = & \left( \frac{(\gamma+1)^{4/3}}{\gamma}M^{1/3} \right)^{t'} (-1)^N\frac{\gamma^M}{(\gamma+1)^{M-N}} e^{-\frac{\gamma}{\gamma+1}Mx}\psi_{r'}(p+q\xi), \\
\tildePsi_{t',r'}(\eta) = & \left( \frac{(\gamma+1)^{4/3}}{\gamma}M^{1/3} \right)^{1-t'}(-1)^N\frac{(\gamma+1)^{M-N}}{\gamma^M} e^{\frac{\gamma}{\gamma+1}My}\psi_{r'}(p+q\eta).
\end{align}
In the same way of \eqref{eq:trace_norm_convergence_of_K2b_less}, we have the convergence result in trace norm
\begin{equation} \label{eq:insignificant_part_of_K_2b_vanish_equal}
\lim_{M \rightarrow \infty} \left\lVert \frac{\gamma}{(\gamma+1)^{4/3}M^{1/3}} \chi(\xi) \left( \sum^{r-r_s}_{r'=1} \frac{1}{1+a_{s'}} e^{\xi/3}\Psi_{r'}(\xi)e^{-\eta/3}\tildePsi_{r'}(\eta) \right) \chi(\eta) \right\rVert_{\Tr} = 0.
\end{equation}
However, if $r' = \left( \sum^{s-1}_j r_j \right) + t'$, $t' = 1, \dots, r_s$,  \eqref{eq:asymptotics_of_psi_rprime_equal} and \eqref{eq:asymptotics_of_varpsi_rprime_equal} yield the convergence in $L^2$ norm
\begin{align}
\lim_{M \rightarrow \infty} e^{\xi/3} \Psi(\xi)_{t',r'}\chi(\xi) = & (-1)^r \gamma^{r-r'}(\gamma+1)^{r'}\bar{C}_{r'-t'} (-1)^{t'-1}e^{\xi/3}t^{(t')}(\xi), \\
\lim_{M \rightarrow \infty} e^{-\eta/3} \tildePsi(\eta)_{t',r'}\chi(\eta) = & \frac{(-1)^r(1+\gamma^{-1})}{\gamma^{r-r'}(\gamma+1)^{r'}\bar{C}_{r'-t'}} (-1)^{t'-1}e^{-\eta/3}s^{(t')}(\xi),
\end{align}
so that by proposition \ref{prop:first_trace_norm_convergence_prop} we have the convergence in trace norm
\begin{multline}
\lim_{M \rightarrow \infty} \chi(\xi)e^{\xi/3} \left( \sum^r_{r'=r-r_s+1} \frac{1}{1+\gamma^{-1}} e^{\xi/3}\Psi_{t',r'}(\xi)e^{-\eta/3}\tildePsi_{t',r'}(\eta) \right) e^{-\eta/3}\chi(\eta) = \\
e^{\xi/3} \left( \sum^{r_s}_{j=1} t^{(j)}(\xi) s^{(j)}(\eta) \right) e^{-\eta/3},
\end{multline}
which is exactly the trace norm limit of $\chi(\xi) f(x)e^{\xi/3}\varK_{2b}(\xi, \eta)f^{-1}(y)e^{-\eta/3} \chi(\eta)$.

Therefore,
\begin{equation}
\begin{split}
& \lim_{M \rightarrow \infty} \det \left( 1 - \chi(\xi) \varK(\xi, \eta) \chi(\eta) \right) \\
= & \lim_{M \rightarrow \infty} \det \left( 1 - \chi(\xi) f(x)e^{\xi/3} (\varK_{2a}(\xi, \eta) + \varK_{2b}(\xi, \eta)) \chi(\eta)f^{-1}(y)e^{-\eta/3} \right) \\
= & \det \left( 1 - \chi(\xi)e^{\xi/3} \left( K_{\Airy}(\xi, \eta) + \sum^{r_s}_{j=1} t^{(j)}(\xi) s^{(j)}(\eta) \right) e^{-\eta/3}\chi(\eta) \right) \\
= & \det \left( 1 - \chi(\xi) \left( K_{\Airy}(\xi, \eta) + \sum^{r_s}_{j=1} t^{(j)}(\xi) s^{(j)}(\eta) \right) \chi(\eta) \right).
\end{split}
\end{equation}

\subsection{The $a_s = a > \gamma^{-1}$ part} 

In this we choose $p = (1+a)\left( 1 + \frac{1}{\gamma^2 a} \right)$ and $q = (1+a) \sqrt{1-\frac{1}{\gamma^2 a^2}}\frac{1}{\sqrt{M}}$, and conjugate the kernel by $f_a(x)e^{2\xi/3}$, where $f_a(x) = x^{(M-N)/2}e^{\frac{a-1}{2(1+a)}Mx}$, Then by \eqref{eq:integral_form_of_Airy_kernel}, \eqref{eq:final_integral_formula_of_psi} and \eqref{eq:final_integral_formula_of_varpsi} we have
\begin{equation}
f_a(x)e^{2\xi/3} \varK_{2a}(\xi, \eta) f^{-1}_a(y)e^{-2\eta/3} = -\int^{\infty}_0 e^{2(\xi+t)/3}\Psi^a(\xi+t)\tildePsi^a(\eta+t) e^{-2(\eta+t)/3} dt, 
\end{equation}
where
\begin{align}
\Psi^a(\xi) = & (1+a) \sqrt{\left( 1-\frac{1}{\gamma^2 a^2} \right)M} \frac{e^{-\frac{M}{1+a}x}}{(-a)^N(1+a)^{M-N}} \psi(p+q\xi), \\
\tildePsi^a(\eta) = & (1+a) \sqrt{\left( 1-\frac{1}{\gamma^2 a^2} \right)M} (-a)^N(1+a)^{M-N}e^{\frac{M}{1+a}y} \varpsi(p+q\eta),
\end{align}
\eqref{eq:insignificant_psi_more} and \eqref{eq:insignificant_part_of_varpsi_more} yield that ($C$ is a large enough positive number)
\begin{align}
\lim_{M \rightarrow \infty} & \lVert e^{2\xi/3}\Psi^a(\xi)(\xi-T) \rVert_{L^2([T,\infty))} < C, \\
\lim_{M \rightarrow \infty} & \lVert e^{-2\eta/3}\tildePsi^a(\eta)(\eta-T) \rVert_{L^2([T,\infty))} = 0,
\end{align}
then by proposition \ref{prop:second_trace_norm_convergence_prop}, we have in trace norm
\begin{equation}
\lim_{M \rightarrow \infty} \chi(\xi) f_a(x)e^{2\xi/3} \varK_{2a}(\xi, \eta) f^{-1}_a(y)e^{-2\eta/3} \chi(\eta) = 0.
\end{equation}

For the $\varK_{2b}$ part, we have the formula similar to \eqref{eq:conjugated_K_2b_equal}
\begin{equation}
\begin{split}
& f_a(x)e^{2\xi/3} \varK_{2b}(\xi, \eta) f^{-1}_a(y)e^{-2\eta/3} \\
= & \frac{1}{(1+a)\sqrt{\left( 1 - \frac{1}{\gamma^2 a^2} \right)M}}\sum^{r-r_s}_{r'=1} e^{2\xi/3}\Psi^a_{r'}(\xi)\tildePsi^a_{r'}(\eta)e^{-2\eta/3} \\
& + \sum^r_{j = r-r_s+1} \frac{1}{1+a} e^{2\xi/3}\Psi^a_{t',r'}(\xi)\tildePsi^a_{t',r'}(\eta)e^{-2\eta/3},
\end{split}
\end{equation}
where
\begin{align}
\Psi^a_{r'}(\xi) = & (1+a) \sqrt{\left( 1-\frac{1}{\gamma^2 a^2} \right)M} \frac{e^{-\frac{M}{1+a}x}}{(-a)^N(1+a)^{M-N}} \psi_{r'}(p+q\xi), \\
\tildePsi^a_{r'}(\eta) = & (1+a) \sqrt{\left( 1-\frac{1}{\gamma^2 a^2} \right)M} (-a)^N(1+a)^{M-N}e^{\frac{M}{1+a}y} \varpsi_{r'}(p+q\eta), \\
\Psi^a_{t',r'}(\xi) = & \left( (1+a) \sqrt{\left( 1-\frac{1}{\gamma^2 a^2} \right)M} \right)^{t'} \frac{e^{-\frac{M}{1+a}x}}{(-a)^N(1+a)^{M-N}} \psi_{r'}(p+q\xi), \\
\tildePsi^a_{t',r'}(\eta) = & \left( (1+a) \sqrt{\left( 1-\frac{1}{\gamma^2 a^2} \right)M} \right)^{1-t'} (-a)^N(1+a)^{M-N}e^{\frac{M}{1+a}y} \varpsi_{r'}(p+q\eta).
\end{align}
\eqref{eq:insignificant_psi_rprime_more} and \eqref{eq:insignificant_part_of_varpsi_more} yield that for $r' = 1, \dots, r-r_s$, ($C$ is a large enough positive number)
\begin{align}
\lim_{M \rightarrow \infty} & \lVert e^{2\xi/3} \Psi^a_{r'}(\xi) \chi(\xi) \rVert_{L^2([T, \infty))} < C, \\
\lim_{M \rightarrow \infty} & \lVert e^{-2\eta/3} \tildePsi^a_{r'}(\eta) \chi(\eta) \rVert_{L^2([T, \infty))} = 0,
\end{align}
so that similar to \eqref{eq:insignificant_part_of_K_2b_vanish_equal}, we have the trace norm convergence
\begin{equation} 
\lim_{M \rightarrow \infty} \left\lVert \frac{1}{(1+a)\sqrt{\left( 1 - \frac{1}{\gamma^2 a^2} \right)M}} \chi(\xi) \left( \sum^{r-r_s}_{r'=1} e^{2\xi/3}\Psi_{r'}(\xi)e^{-2\eta/3}\tildePsi_{r'}(\eta) \right) \chi(\eta) \right\rVert_{\Tr} = 0.
\end{equation}
However, \eqref{eq:significant_psi_rprime_more} and \eqref{eq:significant_varpsi_rprime_more} give the $L^2$ norm convergence results
\begin{align}
\lim_{M \rightarrow \infty} e^{2\xi/3}\Psi^a_{t',r'}(\xi)\chi(\xi) = & \frac{(1+a)^{r'}}{(-a)^r}\bar{C}_{a,r'-t'} (-1)^{t'-1} e^{2\xi/3} \frac{H_{t'-1}(x)}{\sqrt{2\pi}} e^{-\frac{\xi^2}{2}}, \\
\lim_{M \rightarrow \infty} e^{-2\eta/3}\tildePsi^a_{t',r'}(\eta)\chi(\eta) = & \frac{(-a)^r}{(1+a)^{r'-1}\bar{C}_{r'-t'}} \frac{(-1)^{t'-1}}{(t'-1)!} e^{-2\eta/3} H_{t'-1}(\eta).
\end{align}
Then by proposition \ref{prop:first_trace_norm_convergence_prop}, we have the trace norm convergence result
\begin{multline}
\lim_{M \rightarrow \infty} \chi(\xi) \left( \sum^r_{j = r-r_s+1} \frac{1}{1+a} e^{2\xi/3}\Psi^a_{t',r'}(\xi)\tildePsi^a_{t',r'}(\eta)e^{-2\eta/3} \right) \chi(\eta) = \\
\chi(\xi)e^{2\xi/3} \left( \sum^{r_s-1}_{j=0} \frac{1}{j!\sqrt{2\pi}}H_j(\xi)H_j(\eta)e^{-\frac{\xi^2}{2}} \right) e^{-2\eta/3}\chi(\eta).
\end{multline}
Therefore,
\begin{equation}
\begin{split}
& \lim_{M \rightarrow \infty} \det(1 - \chi(\xi)\varK(\xi, \eta)\chi(\eta)) \\
= & \lim_{M \rightarrow \infty} \det \left( 1 - \chi(\xi)f_a(x)e^{2\xi/3} (\varK_{2a}(\xi, \eta) + \varK_{2b}(\xi, \eta)) f^{-1}_a(y)e^{-2\eta/3} \right) \\
= & \det \left( 1 - \chi(\xi)e^{2\xi/3} \left( \sum^{r_s-1}_{j=0} \frac{1}{j!\sqrt{2\pi}}H_j(\xi)H_j(\eta)e^{-\frac{\xi^2}{2}} \right) e^{-2\eta/3}\chi(\eta) \right) \\
= & \det \left( 1 - \chi(\xi) \left( \sum^{r_s-1}_{j=0} \frac{1}{j!\sqrt{2\pi}}H_j(\xi)e^{-\frac{\xi^2}{4}} H_j(\eta)e^{-\frac{\eta^2}{4}} \right) \chi(\eta) \right),
\end{split}
\end{equation}
where in the last step we conjugate the kernel by $e^{-\frac{2\xi}{3} + \frac{\xi^2}{4}}$.

%% file: quaternionic_spiked_model.tex
\chapter{Rank $1$ quaternionic spiked model}

\label{quaternionic_spiked_model_rank_1}

In this chapter, we consider the rank $1$ quaternionic Wishart ensemble unless otherwise stated. Since we have only one parameter $\alpha_1$ in the \pdf\ of sample eigenvalues \eqref{eq:quaternionic_Zonal_expanssion}, we denote it as $a$.

The reader should be cautious that some notations, for instance, $\varphi_j$ and $\psi_j$, are defined differently from their definitions in chapter \ref{complex_spiked_model}, and $\mVandermonde$ is not the fourth power of $\varV(\lambda)$.

\input{joint_pdf_of_quater_rank_1}

\input{pfaffian_and_det_formulas_of_quater_rank_1}

\input{S_4_of_quater_rank_1}

\input{proof_of_theorem_quater_rank_1}

%% file: joint_pdf_of_quater_rank_1.tex
\section{The joint distribution function}

In this section, we prove
\begin{theorem}
The joint probability distribution function of $\lambda$ in the quaternionic spiked model is 
\begin{equation} \label{eq:final joint distribution formula}
P(\lambda) = \frac{1}{C} \mVandermonde \prod^{N}_{j=1} \left(
\lambda^{2(M-N)+1}_j e^{-2M\lambda_j} \right).
\end{equation}
\end{theorem}
Here
\begin{equation}
\mVandermonde =
\begin{vmatrix}
1 & 0 & \dots & 1 & 0 \\
\lambda_1 & 1 & \dots & \lambda_N & 1 \\
\lambda^2_1 & 2\lambda_1 & \dots & \lambda^2_N & 2\lambda_N \\
\vdots & \vdots & \dots & \vdots & \vdots \\
\lambda^{2N-2}_1 & (2N-2)\lambda^{2N-3}_1 & \dots & \lambda^{2N-2}_N
& (2N-2)\lambda^{2N-3}_N \\
e^{\frac{a}{1+a}2M\lambda_1} & \frac{a}{1+a}2M
e^{\frac{a}{1+a}2M\lambda_1} & \dots & e^{\frac{a}{1+a}2M\lambda_N}
& \frac{a}{1+a}2M e^{\frac{a}{1+a}2M\lambda_N}
\end{vmatrix},
\end{equation}
the determinant of a $2N \times 2N$ matrix whose $(2N, 2k-1)$ entry is $e^{\frac{a}{1+a}2M\lambda_k}$, $(j, 2k-1)$ entry is $\lambda^{j-1}_k$ for $j = 1, \dots, 2N-1$, and $2i$-th column is the derivative of the $(2i-1)$-st column. $\mVandermonde$ is a variation of the $V(\lambda)^4$ appearing in the LSE (see \cite{Mehta04}).

Since rank $r = 1$, we can simplify \eqref{eq:quaternionic_Zonal_expanssion} as
\begin{equation} \label{joint+pdf_for_quater_rank_1}
P(\lambda) = \frac{1}{C} V(\lambda)^4 \prod^N_{j=1} \lambda^{2(M-N)+1}_j e^{-2M\lambda_j} \sum^{\infty}_{j=0} \frac{(2M)^j}{j!} \frac{\Qj(\frac{a}{1+a}) \Qj(\lambda_1, \dots, \lambda_N)}{\Qj(I_N)},
\end{equation}
just like formula \eqref{eq:rank_1_pdf_in_complex_case} for the complex case.

We have \cite{Stanley89}
\begin{align}
\Qj (\frac{a}{1+a}) = & \left( \frac{a}{1+a} \right)^j \\
\intertext{and since the number of variables is $N$ \cite{Stanley89}} 
\Qj (1, \dots, 1) = & \frac{1}{(j+1)!} \prod^{j-1}_{i=0} (2N+i),
\end{align}
so we get
\begin{multline}
\sum^{\infty}_{j=0} \frac{(2M)^j}{j!} \frac{\Qj(\frac{a}{1+a}) \Qj(\lambda_1, \dots, \lambda_N)}{\Qj(I_N)} = \\
\sum^{\infty}_{j=0} \frac{j+1}{\prod^{j-1}_{i=0} (2N+i)} \left( \frac{a}{1+a} 2M \right)^j \Qj(\lambda_1, \dots, \lambda_N).
\end{multline}
In \cite{Stanley89} there is an identity
\begin{equation}
\sum^{\infty}_{j=0} (j+1)\Qj(\lambda_1, \dots, \lambda_N)t^j = \prod^{N}_{j=1} \frac{1}{(1+\lambda_{j}t)^2}.
\end{equation}
Comparing it with the well-known identity for Schur polynomials
\begin{equation}
\sum^{\infty}_{j=0} \sj(\lambda_1, \dots, \lambda_N)t^j = \prod^{N}_{j=1} \frac{1}{1+\lambda_{j}t},
\end{equation}
we get the identity
\begin{equation} \label{eq:defination of double Schur}
(j+1)\Qj(\lambda_1, \dots, \lambda_N) = \sj(\lambda_1, \lambda_1, \lambda_2, \lambda_2,\dots, \lambda_N, \lambda_N),
\end{equation}
with each $\lambda_i$ appearing twice as variables of the $\sj$. For notational simplicity, we denote the right hand side of (\ref{eq:defination of double Schur}) as $\ssj(\Lambda)$, which is a plethysm \cite{Macdonald95}
\begin{equation}
\ssj(\lambda_1, \dots, \lambda_N) = \sj \circ 2p_1(\lambda_1, \dots, \lambda_N).
\end{equation}
Now we get
\begin{multline} \label{eq:sum in double Schur}
\sum^{\infty}_{j=0} \frac{(2M)^j}{j!} \frac{\Qj(\frac{a}{1+a}) \Qj(\lambda_1, \dots, \lambda_N)}{\Qj(I_N)} = \\
\sum^{\infty}_{j=0} \frac{1}{\prod^{j-1}_{i=0} (2N+i)} \left( \frac{a}{1+a} 2M \right)^j \ssj(\lambda_1, \dots, \lambda_N).
\end{multline}

Then we need a lemma to simplify (\ref{eq:sum in double Schur}) further.

\begin{lemma} \label{lemma:matrix representation}
\begin{equation} \label{eq:matrix_representation}
\ssj(\Lambda) = \frac{
\begin{vmatrix}
1 & 0 & \dots & 1 & 0 \\
\lambda_1 & 1 & \dots & \lambda_N & 1 \\
\lambda^2_1 & 2\lambda_1 & \dots & \lambda^2_N & 2\lambda_N \\
\vdots & \vdots & \dots & \vdots & \vdots \\
\lambda^{2N-2}_1 & (2N-2)\lambda^{2N-3}_1 & \dots & \lambda^{2N-2}_N & (2N-2)\lambda^{2N-3}_N \\
\lambda^{2N+j-1}_1 & (2N+j-1)\lambda^{2N+j-2}_1 & \dots & \lambda^{2N+j-1}_N & (2N+j-1)\lambda^{2N+j-2}_N
\end{vmatrix}
} {V(\lambda)^4},
\end{equation}
with the $(k,2j-1)$ entry of the matrix being a power of $\lambda_j$ with the exponent $k-1$ if $k \neq 2N$ and $2N+j-1$ if $k = 2N$, and the $(k,2j)$ entry being the derivative of the $(k,2j-1)$ entry with respect to $\lambda_j$.
\end{lemma}

To prove this lemma, we need the well known fact (see \cite{Mehta04}), proven by L'H\^{o}pital's rule
\begin{equation} \label{eq:Vandemonde_to_fourth}
V(\lambda)^4 =
\begin{vmatrix}
1 & 0 & \dots & 1 & 0 \\
\lambda_1 & 1 & \dots & \lambda_N & 1 \\
\vdots & \vdots & \dots & \vdots & \vdots \\
\lambda^{2N-1}_1 & (2N-1)\lambda^{2N-2}_1 & \dots & \lambda^{2N-1}_N & (2N-1)\lambda^{2N-2}_N
\end{vmatrix},
\end{equation}
with the $(k,2j-1)$ entry being $\lambda^{k-1}_j$ and the $(k,2j)$ entry $(k-1)\lambda^{k-2}_j$.

\begin{proof}[Proof of the lemma]
Applying the L'H\^{o}pital's rule repeatedly with respect to $x_{2i}$, $i=1,\dots,N$, we get the identity
\begin{equation}
\begin{split}
& \frac{
\begin{vmatrix}
1 & 0 & \dots & 1 & 0 \\
\lambda_1 & 1 & \dots & \lambda_N & 1 \\
\lambda^2_1 & 2\lambda_1 & \dots & \lambda^2_N & 2\lambda_N \\
\vdots & \vdots & \dots & \vdots & \vdots \\
\lambda^{2N-2}_1 & (2N-2)\lambda^{2N-3}_1 & \dots & \lambda^{2N-2}_N & (2N-2)\lambda^{2N-3}_N \\
\lambda^{2N+j-1}_1 & (2N+j-1)\lambda^{2N+j-2}_1 & \dots &
\lambda^{2N+j-1}_N & (2N+j-1)\lambda^{2N+j-2}_N
\end{vmatrix}}
{\begin{vmatrix}
1 & 0 & \dots & 1 & 0 \\
\lambda_1 & 1 & \dots & \lambda_N & 1 \\
\vdots & \vdots & \dots & \vdots & \vdots \\
\lambda^{2N-1}_1 & (2N-1)\lambda^{2N-2} & \dots & \lambda^{2N-1}_N & (2N-1)\lambda^{2N-2}_N
\end{vmatrix}} \\
= & \left. \frac{\frac{\partial^N}{\partial x_2 \partial x_4 \dots \partial x_{2N}}
\begin{vmatrix}
1 & 1 & \dots & 1 & 1 \\
x_1 & x_2 & \dots & x_{2N-1} & x_{2N} \\
\vdots & \vdots & \dots & \vdots & \vdots \\
x^{2N-2}_1 & x^{2N-2}_2 & \dots & x^{2N-2}_{2N-1} & x^{2N-2}_{2N} \\
x^{2N+j-1}_1 & x^{2N+j-1}_2 & \dots & x^{2N+j-1}_{2N-1} & x^{2N+j-1}_{2N}
\end{vmatrix}} {\frac{\partial^N}{\partial x_2 \partial x_4 \dots \partial x_{2N}}
\begin{vmatrix}
1 & 1 & \dots & 1 & 1 \\
x_1 & x_2 & \dots & x_{2N-1} & x_{2N} \\
\vdots & \vdots & \dots & \vdots & \vdots \\
x^{2N-2}_1 & x^{2N-2}_2 & \dots & x^{2N-2}_{2N-1} & x^{2N-2}_{2N} \\
x^{2N+j-1}_1 & x^{2N+j-1}_2 & \dots & x^{2N+j-1}_{2N-1} & x^{2N+j-1}_{2N}
\end{vmatrix}} \right|_{\substack{x_{2i-1}=x_{2i}=\lambda_i \\
i=1.\dots,N}} \\
= & \sj(\lambda_1, \lambda_1, \lambda_2, \lambda_2, \dots, \lambda_N, \lambda_N) = \ssj(\lambda_1, \dots, \lambda_N),
\end{split}
\end{equation}
from the matrix representation of Schur polynomials, and now use (\ref{eq:Vandemonde_to_fourth}) to get the compact formula (\ref{eq:matrix_representation}).
\end{proof}

Substituting (\ref{eq:matrix_representation}) into (\ref{eq:sum in double Schur}), we get
\begin{equation} \label{eq:modified 2Nx2N matrix}
\begin{split}
& V(\lambda)^4 \sum^{\infty}_{j=0} \frac{(2M)^j}{j!} \frac{\Qj(\frac{a}{1+a}) \Qj(\lambda_1, \dots, \lambda_N)}{\Qj(I_N)} \\
= &
\begin{vmatrix}
1 & 0 & \dots & 1 & 0 \\
\lambda_1 & 1 & \dots & \lambda_N & 1 \\
\lambda^2_1 & 2\lambda_1 & \dots & \lambda^2_N & 2\lambda_N \\
\vdots & \vdots & \dots & \vdots & \vdots \\
\lambda^{2N-2}_1 & (2N-2)\lambda^{2N-3}_1 & \dots & \lambda^{2N-2}_N & (2N-2)\lambda^{2N-3}_N \\
p(\lambda_1) & p'(\lambda_1) & \dots & p(\lambda_N) & p'(\lambda_N) \end{vmatrix} \\
= & \frac{1}{C}
\begin{vmatrix}
1 & 0 & \dots & 1 & 0 \\
\lambda_1 & 1 & \dots & \lambda_N & 1 \\
\lambda^2_1 & 2\lambda_1 & \dots & \lambda^2_N & 2\lambda_N \\
\vdots & \vdots & \dots & \vdots & \vdots \\
\lambda^{2N-2}_1 & (2N-2)\lambda^{2N-3}_1 & \dots & \lambda^{2N-2}_N & (2N-2)\lambda^{2N-3}_N \\
e^{\frac{a}{1+a}2M\lambda_1} & \frac{a}{1+a}2M e^{\frac{a}{1+a}2M\lambda_1} & \dots & e^{\frac{a}{1+a}2M\lambda_N} & \frac{a}{1+a}2M e^{\frac{a}{1+a}2M\lambda_N}
\end{vmatrix} \\
= & \frac{1}{C} \mVandermonde,
\end{split}
\end{equation}
where
\begin{equation}
\begin{split}
p(x) = & \sum^{\infty}_{j=0} \frac{1}{\prod^{j-1}_{i=0} (2N+i)} \left( \frac{a}{1+a} 2M \right)^j x^{2N+j-1} \\
= & \frac{(2N-1)!}{\left( \frac{a}{1+a} 2M \right)^{2N-1}} \left( e^{\frac{a}{1+a}2Mx} - \sum^{2N-2}_{j=0} \frac{1}{j!} \left( \frac{a}{1+a} 2Mx \right)^j \right),
\end{split}
\end{equation}
and if $k \neq 2N$, the $(k,2j-1)$ entries in both matrices are $\lambda^{k-1}_{j}$, and the $(k,2j)$ entries are $(k-1)\lambda^{k-2}_{j}$, and the $2N,2i-1$ entry in the former (latter) matrix is $p(\lambda_i)$ (resp. $e^{\frac{a}{1+a}2M\lambda_i}$) and the $2N,2i$ entry $p'(\lambda_i)$ (resp. $\frac{a}{1+a}2M e^{\frac{a}{1+a}2M\lambda_i}$).

\begin{proof}[Proof of the theorem]
Formulas (\ref{joint+pdf_for_quater_rank_1}) and (\ref{eq:modified 2Nx2N matrix}) together give the result (\ref{eq:final joint distribution formula}).
\end{proof}

%% file: pfaffian_and_det_formulas_of_quater_rank_1.tex
\section{The determinantal formula}

With the formula (\ref{eq:final joint distribution formula}) ready to use, we get the limiting  distribution formula for the largest sample eigenvalue, in the same spirit as the solution of the quaternionic white Wishart ensemble. Our process below is closely parallel to that in \cite{Tracy-Widom98}.

First, we find a skew orthogonal basis $\{ \varphi_0(x), \varphi_1(x), \dots, \varphi_{2N-1}(x) \}$ of the linear space spanned by $\{ 1, x, x^2, \dots, x^{2N-2}, e^{\frac{a}{1+a}2Mx} \}$.We require that the $\varphi_j(x)$ is a linear combination of $\{ 1,x, x^2, \dots, x^j \}$ if $j < 2N-1$, while $\varphi_{2N-1}(x)$ can be arbitrary, with the skew inner products among them
\begin{equation}
\begin{split}
\langle \varphi_j(x), \varphi_k(x) \rangle_4 = & \int^{\infty}_0
(\varphi_j(x)\varphi'_k(x) - \varphi'_j(x)\varphi_k(x))
x^{2(M-N)+1}e^{-2Mx} dx \\
= &
\begin{cases}
r_{j/2} & \textrm{if $j$ is even and $k=j+1$,} \\
-r_{k/2} & \textrm{if $k$ is even and $j=k+1$,} \\
0 & \textrm{otherwise.}
\end{cases}
\end{split}
\end{equation}

\begin{remark}
Due to the shortage of notations, we abuse the language so that we use $\varphi$ and $\psi$ in this chapter to mean functions different from those in chapter \ref{complex_spiked_model}.
\end{remark}

Then we can reformulate the distribution function of $\lambda$ as
\begin{equation} \label{eq:skew polynomial form distribution}
\begin{split}
P(\lambda) = & \frac{1}{C}
\begin{vmatrix}
\varphi_0(\lambda_1) & \varphi'_0(\lambda_1) & \dots & \varphi_0(\lambda_N) & \varphi'_0(\lambda_N) \\
\varphi_1(\lambda_1) & \varphi'_1(\lambda_1) & \dots & \varphi_1(\lambda_N) & \varphi'_1(\lambda_N) \\
\vdots & \vdots & \dots & \vdots & \vdots \\
\varphi_{2N-1}(\lambda_1) & \varphi'_{2N-1}(\lambda_1) & \dots & \varphi_{2N-1}(\lambda_N) & \varphi'_{2N-1}(\lambda_N)
\end{vmatrix} \\
& \phantom{\frac{1}{C}} \prod^{N}_{j=1} \left( \lambda^{2(M-N)+1}_j e^{-2M\lambda_j} \right) \\
= & \frac{1}{C}
\begin{vmatrix}
\psi_0(\lambda_1) & \psi'_0(\lambda_1) & \dots & \psi_0(\lambda_N) & \psi'_0(\lambda_N) \\
\psi_1(\lambda_1) & \psi'_1(\lambda_1) & \dots & \psi_1(\lambda_N) & \psi'_1(\lambda_N) \\
\vdots & \vdots & \dots & \vdots & \vdots \\
\psi_{2N-1}(\lambda_1) & \psi'_{2N-1}(\lambda_1) & \dots & \psi_{2N-1}(\lambda_N) & \psi'_{2N-1}(\lambda_N)
\end{vmatrix},
\end{split}
\end{equation}
where
\begin{equation} \label{eq:definition_of_psi}
\psi_i(x) = \varphi_i(x) x^{M-N+1/2}e^{-Mx}.
\end{equation}
For an arbitrary function $f(x)$ on $[0,\infty)$, by the formula of de Bruijn \cite{de_Bruijn55},
\begin{multline} \label{eq:Pfaffian formula}
\int^{\infty}_0 \dots \int^{\infty}_0
\begin{vmatrix}
\psi_0(\lambda_1) & \psi'_0(\lambda_1) & \dots & \psi_0(\lambda_N) & \psi'_0(\lambda_N) \\
\psi_1(\lambda_1) & \psi'_1(\lambda_1) & \dots & \psi_1(\lambda_N) & \psi'_1(\lambda_N) \\
\vdots & \vdots & \dots & \vdots & \vdots \\
\psi_{2N-1}(\lambda_1) & \psi'_{2N-1}(\lambda_1) & \dots &  \psi_{2N-1}(\lambda_N) & \psi'_{2N-1} (\lambda_N)
\end{vmatrix} \\
\prod^N_{i=1}(1+f(\lambda_i))d\lambda_i = C \Pf(P(1+f)),
\end{multline}
where $P(1+f)$ is a $2N \times 2N$ matrix, whose entries depend on $1+f$ in the following way
\begin{equation}
(P(1+f))_{j,k} = \int^{\infty}_0 (\psi_{j-1}(x)\psi'_{k-1}(x) - \psi'_{j-1}(x)\psi_{k-1}(x)) (1+f(x))dx.
\end{equation}
Now we define a matrix $Z$ as
\begin{equation}
Z=
\begin{pmatrix}
0 & r_0 & & & & & \\
-r_0 & 0 & & & & & \\
& & 0 & r_1 & & & \\
& & -r_1 & 0 & & & \\
& & & & \ddots & & \\
& & & & & 0 & r_{N-1} \\
& & & & & -r_{N-1} & 0
\end{pmatrix},
\end{equation}
with
\begin{equation}
Z_{j,k} =
\begin{cases}
r_{k/2-1} & \textrm{if $k$ is even and $j=k-1$,} \\
-r_{j/2-1} & \textrm{if $j$ is even and $k=j-1$,} \\
0 & \textrm{otherwise,}
\end{cases}
\end{equation}
and define for $j = 0, \dots, N-1$, $\eta = Z^{-1}\psi$, i.e.,
\begin{equation}
\eta_{2j}(x) = -\frac{\psi_{2j+1}(x)}{r_j} \quad \textrm{and} \quad
\eta_{2j+1}(x) = \frac{\psi_{2j}(x)}{r_j}.
\end{equation}
So we have
\begin{equation}
\begin{split}
(P(1+f))_{j,k} = & \int^{\infty}_0 (\psi_{j-1}(x)\psi'_{k-1}(x) - \psi'_{j-1}(x)\psi_{k-1}(x)) dx \\
& + \int^{\infty}_0 (\psi_{j-1}(x)\psi'_{k-1}(x) - \psi'_{j-1}(x)\psi_{k-1}(x)) f(x)dx \\
= & Z_{j,k} + \int^{\infty}_0 (\psi_{j-1}(x)\psi'_{k-1}(x) - \psi'_{j-1}(x)\psi_{k-1}(x)) f(x)dx.
\end{split}
\end{equation}
And if we denote $Q(1+f) = Z^{-1}P(1+f)$, then
\begin{equation}
Q(1+f)_{j,k} = \delta_{j,k} + \int^{\infty}_0 (\eta_{j-1}(x)\psi'_{k-1}(x) - \eta'_{j-1}(x)\psi_{k-1}(x)) f(x)dx.
\end{equation}

If we choose $f$ to be $-\chi_{(T,\infty)}$, then the integral on the left hand side of (\ref{eq:Pfaffian formula}), after multiplying a constant, is the probability of all $\lambda_i$'s smaller than $T$. In latter part of the paper, we abbreviate $\chi_{(T,\infty)}$ to $\chi$ as before. So we get for a $T$-independent constant
\begin{equation}
\mathbb{P}(\max(\lambda_i) \leq T) = C\Pf(P(1-\chi)),
\end{equation}
and
\begin{equation}
(\mathbb{P}(\max(\lambda_i) \leq T))^2 = C^2\det(P(1-\chi)) =
C^2\det(Q(1-\chi)).
\end{equation}

Now we apply a matrix version of \eqref{eq:changed_determinantal_formula_for_complex}. In linear algebra, we have the determinant identity
\begin{equation} \label{eq:det(I-AB)=det(I-BA)}
\det(I-AB) = \det(I-BA),
\end{equation}
for $A$ an linear map from $\mathbb{R}^n$ to $\mathbb{R}^m$ and $B$ an linear map from $\mathbb{R}^n$ to $\mathbb{R}^m$, and the identity still holds in infinite dimensional settings \cite{Gohberg-Krein69}. Letting $\det$ mean a Fredholm determinant for a matrix integral operator defined in \eqref{eq:definition_of_matrix_integral_Fredholm_det}, we describe a setting due to Tracy-Widom \cite{Tracy-Widom98}.

If $A$ is an operator from $L^2(\mathbb{R}) \times L^2(\mathbb{R})$ to the vector space $\mathbb{R}^{2N}$ with
\begin{equation}
A
\begin{pmatrix}
g(x) \\
h(x)
\end{pmatrix}_j
= \int^{\infty}_0 \chi(x)\eta_{j-1}(x)g(x)dx - \int^{\infty}_0 \chi(x)\eta'_{j-1}(x)h(x)dx,
\end{equation}
and $B$ is an operator from $\mathbb{R}^{2N}$ to $L^2(\mathbb{R})
\times L^2(\mathbb{R})$ with
\begin{equation}
B
\begin{pmatrix}
c_1 \\
\vdots \\
c_{2N}
\end{pmatrix}
=
\begin{pmatrix}
\sum^{2N}_{k=1} c_k\psi'_{k-1}(x)\chi(x) \\
\sum^{2N}_{k=1} c_k\psi_{k-1}(x)\chi(x)
\end{pmatrix},
\end{equation}
then
\begin{equation}
I-AB = Q(1-\chi),
\end{equation}
and
\begin{equation}
I-BA = I - \chi(x)
\begin{pmatrix}
S_4(x,y) & SD_4(x,y) \\
IS_4(x,y) & S_4(y,x)
\end{pmatrix}\chi(y),
\end{equation}
where $S_4(x,y)$, $IS_4(x,y)$ and $SD_4(x,y)$ are integral operators whose kernels are
\begin{align}
S_4(x,y) = & \sum^{2N-1}_{j=0} \psi'_j(x)\eta_j(y) = \sum^{N-1}_{j=0} \frac{1}{r_j} (-\psi'_{2j}(x)\psi_{2j+1}(y) + \psi'_{2j+1}(x)\psi_{2j}(y)), \label{eq:definition_of_S_4(xi,eta)}\\
SD_4(x,y) = & \sum^{2N-1}_{j=0} -\psi'_j(x)\eta'_j(y) = \sum^{N-1}_{j=0} \frac{1}{r_j} (\psi'_{2j}(x)\psi'_{2j+1}(y) - \psi'_{2j+1}(x)\psi'_{2j}(y)), \\
IS_4(x,y) = & \sum^{2N-1}_{j=0} \psi_j(x)\eta_j(y) = \sum^{N-1}_{j=0} \frac{1}{r_j} (-\psi_{2j}(x)\psi_{2j+1}(y) + \psi_{2j+1}(x)\psi_{2j}(y)), \\
S_4(y,x) = & \sum^{2N-1}_{j=0} -\psi_j(x)\eta'_j(y) = \sum^{N-1}_{j=0} \frac{1}{r_j} (\psi_{2j}(x)\psi'_{2j+1}(y) - \psi_{2j+1}(x)\psi'_{2j}(y)). \label{eq:definition_of_S_4(eta,xi)}
\end{align}

\begin{remark}
It is clear that the nomenclature of $SD_4(x,y)$ is due to the fact that $SD_4(x,y)$ is the negative of the derivative of $S_4(x,y)$. But $IS_4(x,y)$, which gets its name in the same way in earlier literature in GSE (e.g., \cite{Tracy-Widom96}), in our problem may not satisfy the equation
\begin{equation}
IS_4(x,y) = -\int^{\infty}_x S_4(t,y)dt,
\end{equation}
since the integral on the right hand side may diverge.
\end{remark}

In conclusion,
\begin{equation}
(\mathbb{P}(\max(\lambda_i) \leq T))^2 = C^2\det \left( I - \chi(x)
\begin{pmatrix}
S_4(x,y) & SD_4(x,y) \\
IS_4(x,y) & S_4(y,x)
\end{pmatrix} \chi(y) \right),
\end{equation}
and we can find that $C^2=1$ by taking the limit $T \rightarrow
\infty$. We define a $2 \times 2$ matrix kernel as
\begin{equation}
\begin{split}
\Ptt(x,y) = & \chi(x) \begin{pmatrix}
S_4(x,y) & SD_4(x,y) \\
IS_4(x,y) & S_4(y,x)
\end{pmatrix} \chi(y) \\
= & \begin{pmatrix}
\chi(x) S_4(x,y) \chi(y) & \chi(x) DS_4(x,y) \chi(y) \\
\chi(x) IS_4(x,y) \chi(y) & \chi(x) S_4(y,x) \chi(y)
\end{pmatrix},
\end{split}
\end{equation}
then we have
\begin{equation}
(\mathbb{P}(\max(\lambda_i) \leq T))^2 = \det(I-\Ptt(x,y)).
\end{equation}

%% file: S_4_of_quater_rank_1.tex
\section{$S_4(x,y)$ in terms of Laguerre polynomials}
\label{S_4(xy)_in_terms_of_Laguerre_polynomials}

In manipulation of skew orthogonal polynomials, we take the approach of \cite{Adler-Forrester-Nagao-van_Moerbeke00}, and all properties of Laguerre polynomials are from \cite{Szego75}.

Since Laguerre polynomials by definition satisfy the orthogonal property
\begin{equation}
\int^{\infty}_0 \qLag{j}\qLag{k} x^{2(M-N)}e^{-x}dx = \frac{(j+2(M-N))!}{j!}\delta_{j,k},
\end{equation}
and they have the differential identity \footnote{We assume $\qLag{n}(x) = 0$ if $n < 0$.} \begin{equation} \label{eq:differential identity}
x\frac{d}{dx}\qLag{n}(x) = n\qLag{n}(x) - (n+2(M-N))\qLag{n-1}(x),
\end{equation}
it is easy to get that
\begin{equation}
\begin{split}
& \left\langle \qLag{j}(2Mx), \qLag{k}(2Mx) \right\rangle_4 \\
= & \int^{\infty}_0 \left( \qLag{j}(2Mx)\frac{d}{dx}\qLag{k}(2Mx) \right. \\
& \phantom{\int^{\infty}_0} - \left. \qLag{k}(2Mx)\frac{d}{dx}\qLag{j}(2Mx) \right) x^{2(M-N)+1}e^{-2Mx}dx \\
= &
\begin{cases}
\left( \frac{1}{2M} \right)^{2(M-N)+1} \frac{(j+2(M-N))!}{(j-1)!} & \textrm{if $j=k+1$,} \\
-\left( \frac{1}{2M} \right)^{2(M-N)+1} \frac{(k+2(M-N))!}{(k-1)!} & \textrm{if $k=j+1$,} \\
0 & \textrm{otherwise.}
\end{cases}
\end{split}
\end{equation}
So we can choose for $j = 0, \dots, N-2$,
\begin{align}
\varphi_{2j}(x) = & \sum^j_{k=0} \left( \prod^k_{i=1} \frac{2i-1}{2i+2(M-N)} \right) \qLag{2k}(2Mx), \label{eq:defination of even varphi} \\
\varphi_{2j+1}(x) = & -\qLag{2j+1}(2Mx), \label{eq:definition of odd varphi} \\
\intertext{and} r_j = & \left( \frac{1}{2M} \right)^{2(M-N)+1} \frac{(2j+2(M-N)+1)!}{(2j)!} \prod^j_{k=1}\frac{2k-1}{2k+2(M-N)}.
\label{eq:definition_of_r_i}
\end{align}
We can also choose
\begin{equation}
\varphi_{2N-2}(x) = \sum^{N-1}_{k=0} \left( \prod^k_{i=1} \frac{2i-1}{2i+2(M-N)} \right) \qLag{2k}(2Mx),
\end{equation}
but $\varphi_{2N-1}(x)$ is not a polynomial and needs to be treated separately.

By the Rodrigues' representation
\begin{equation}
x^{2(M-N)}e^{-x}\qLag{n}(x) = \frac{1}{n!} \frac{d^n}{dx^n} (e^{-x}x^{n+2(M-N)}),
\end{equation}
and repeated integration by parts, we get for $n > 0$
\begin{multline}
\left\langle e^{\frac{a}{1+a}2Mx}, \qLag{n}(2Mx) \right\rangle_4 = \\
\left( \frac{1+a}{2M} \right)^{2(M-N)+1} \left( (-a)^{n+1}\frac{(n+2(M-N)+1)!}{n!} - (-a)^{n-1}\frac{(n+2(M-N))!}{(n-1)!} \right)
\end{multline}
and
\begin{equation}
\left\langle e^{\frac{a}{1+a}2Mx}, \qLag{0}(2Mx) \right\rangle_4 = - \left( \frac{1+a}{2M} \right)^{2(M-N)+1} a(2(M-N)+1)!,
\end{equation}
so that
\begin{multline}
\left\langle e^{\frac{a}{1+a}2Mx}, \varphi_{2j}(x) \right\rangle_4 = \\
-\left( \frac{1+a}{2M} \right)^{2(M-N)+1} a^{2j+1}\frac{(2j+2(M-N)+1)!}{(2j)!} \prod^j_{k=1}\frac{2k-1}{2k+2(M-N)}
\end{multline}
and
\begin{multline}
\left\langle e^{\frac{a}{1+a}2Mx}, \varphi_{2j+1}(x) \right\rangle_4 = \\
-\left( \frac{1+a}{2M} \right)^{2(M-N)+1} \left( a^{2j+2}\frac{(2j+2(M-N)+2)!}{(2j+1)!} - a^{2j}\frac{(2j+2(M-N)+1)!}{(2j)!} \right).
\end{multline}
Now by the skew orthogonality, we can choose
\begin{equation} \label{eq:psi_2N-1_raw_formula}
\begin{split}
\varphi_{2N-1}(x) = & - \sum^{N-2}_{j=0} \frac{1}{r_j}\left( \left\langle e^{\frac{a}{1+a}2Mx}, \varphi_{2j+1}(x) \right\rangle_4 \varphi_{2j}(x) - \left\langle e^{\frac{a}{1+a}2Mx}, \varphi_{2j}(x) \right\rangle_4 \varphi_{2j+1}(x) \right) \\
& - (1+a)^{2(M-N)+1} a^{2N-2}\prod^{N-1}_{j=1}\frac{2j+2(M-N)}{2j-1} \varphi_{2N-2}(x) + e^{\frac{a}{1+a}2Mx} \\
= & e^{\frac{a}{1+a}2Mx} - (1+a)^{2(M-N)+1} \sum^{2N-2}_{j=0} (-a)^j \qLag{j}(2Mx)
\end{split}
\end{equation}
and
\begin{equation}
r_{N-1} = \left( \frac{1+a}{2M} \right)^{2(M-N)+1} a^{2N-1}\frac{(2M-1)!}{(2N-2)!} \prod^{N-1}_{k=1}\frac{2k-1}{2k+2(M-N)}.
\end{equation}

Now, we write $S_4(x,y)$ as $S_{4a}(x,y) + S_{4b}(x,y)$, where
\begin{align}
S_{4a}(x,y) = & \sum^{N-2}_{j=0} \frac{1}{r_j} (-\psi'_{2j}(x)\psi_{2j+1}(y) + \psi'_{2j+1}(x)\psi_{2j}(y)) \label{eq:S4a}\\
\intertext{and} 
S_{4b}(x,y) = & \frac{1}{r_{N-1}} (-\psi'_{2N-2}(x)\psi_{2N-1}(y) + \psi'_{2N-1}(x)\psi_{2N-2}(y)),
\end{align}
and simplify them separately.

The formula (\ref{eq:S4a}) of our $S_{4a}(x,y)$ is also the formula for $S_4(x,y)$ in the LSE problem, with parameters $M$ and $N-2$, and has been well studied. For completeness we derive its Laguerre polynomial expression here, following \cite{Adler-Forrester-Nagao-van_Moerbeke00}.

By the differential identity (\ref{eq:differential identity}) and the identity
\begin{multline}
n\qLag{n}(x) = \\
(-x+2n+2(M-N)-1)\qLag{n-1}(x) - (n+2(M-N)-1)\qLag{n-2}(x),
\end{multline}
we get, remembering the definition (\ref{eq:definition_of_psi}), the telescoping sequence
\begin{equation}
\begin{split}
\psi'_{2j}(x) = & \sum^j_{k=0} \left( \prod^k_{i=1}\frac{2i-1}{2i+2(M-N)} \right. \\
& \left. \phantom{\sum^j_{k=0}} \left( M-N+1/2-Mx+x\frac{d}{dx} \right) \qLag{2k}(2Mx) \right) x^{M-N-1/2}e^{-Mx} \\
= & \frac{1}{2}\sum^j_{k=0} \prod^k_{i=1} \frac{2i-1}{2i+2(M-N)} \left(
(2k+1)\qLag{2k+1}(2Mx) \right. \\
& \phantom{\frac{1}{2}\sum^j_{k=0}} \left. - (2k+2(M-N))\qLag{2k-1}(2Mx) \right) x^{M-N-1/2}e^{-Mx} \\
= & \frac{1}{2} \left( \prod^j_{k=1}\frac{2k-1}{2k+2(M-N)} \right) (2j+1)\qLag{2j+1}(2Mx) x^{M-N-1/2}e^{-Mx} \label{eq:even derivative}
\end{split}
\end{equation}
and
\begin{equation}
\begin{split}
\psi'_{2j+1}(x) = & -\left( M-N+1/2-Mx+x\frac{d}{dx} \right) \qLag{2j+1}(2Mx) x^{M-N-1/2}e^{-Mx} \\
= & -\frac{1}{2} \left( (2j+2)\qLag{2j+2}(2Mx) \right. \\
& \phantom{-\frac{1}{2}} \left. - (2j+2(M-N)+1)\qLag{2j}(2Mx) \right) x^{M-N-1/2}e^{-Mx}. \label{eq:odd derivative}
\end{split}
\end{equation}
Therefore, if we plug in (\ref{eq:definition_of_r_i}), (\ref{eq:even derivative}) and (\ref{eq:odd derivative}) into (\ref{eq:S4a}), we get after some trick,
\begin{equation}
\begin{split}
S_{4a}(x,y) = & \frac{1}{2}(2M)^{2(M-N)+1} x^{M-N-1/2}e^{-Mx}y^{M-N+1/2}e^{-My} \\
& \left\{ \sum^{2N-2}_{j=0} \frac{j!}{(j+2(M-N))!} \qLag{j}(2Mx)\qLag{j}(2My) \right. \\
& \left. -\frac{(2N-2)!}{(2M-2)!} \left( \prod^{N-1}_{j=1}\frac{2j+2(M-N)}{2j-1} \right)
\qLag{2N-2}(2Mx)\varphi_{2N-2}(y) \right\}.
\end{split}
\end{equation}

Furthermore, we can simplify $\psi_{2N-2}(x)$. Since for $j \neq 2N-1$, (if we define $\varphi_j(x)$ and then $\psi_j(x)$ for $j>2N-1$ by the formula (\ref{eq:defination of even varphi}) and (\ref{eq:definition of odd varphi}),)
\begin{equation}
\int^{\infty}_0 \left( \psi_{2N-2}(x)\psi'_j(x) - \psi'_{2N-2}(x)\psi_j(x) \right) dx = 0,
\end{equation}
we get for $j \neq 2N-1$, using integration by parts,
\begin{equation}
\int^{\infty}_0 \psi'_{2N-2}(x)\qLag{j}(2Mx) x^{M-N+1/2}e^{-Mx}dx = 0.
\end{equation}
So by the orthogonal property of Laguerre polynomials, we get
\begin{equation}
\psi'_{2N-2}(x) = C\qLag{2N-1}(2Mx) x^{M-N-1/2}e^{-Mx},
\end{equation}
and we can determine that
\begin{equation}
C = \frac{2N-1}{2}\prod^{N-1}_{j=1}\frac{2j-1}{2j+2(M-N)}
\end{equation}
without much difficulty. Together with the fact $\lim_{x\rightarrow\infty}\psi_{2N-2}(x)=0$, we get
\begin{equation}
\psi_{2N-2}(x) = -\frac{2N-1}{2}\prod^{N-1}_{j=1}\frac{2j-1}{2j+2(M-N)} \int^{\infty}_x t^{M-N-1/2}e^{-Mt}\qLag{2N-1}(2Mt)dt.
\end{equation}

Now, we can write $S_{4a}(x,y)$ as $S_{4a1}(x,y) + S_{4a2}(x,y)$, where
\begin{multline}
S_{4a1}(x,y) = \frac{1}{2}(2M)^{2(M-N)+1} x^{M-N-1/2}e^{-Mx} y^{M-N+1/2}e^{-My} \\
\sum^{2N-2}_{j=0} \frac{j!}{(j+2(M-N))!} \qLag{j}(2Mx) \qLag{j}(2My) 
\end{multline}
and
\begin{multline} \label{eq:definition_of_S_4a2}
S_{4a2}(x,y) = \frac{1}{4} (2M)^{2(M-N)+1} \frac{(2N-1)!}{(2M-2)!} \\
\qLag{2N-2}(2Mx)x^{M-N-1/2}e^{-Mx} \int^{\infty}_y t^{M-N-1/2}e^{-Mt}\qLag{2N-1}(2Mt)dt.
\end{multline}

Finally,
\begin{multline} \label{eq:formula_for_S_4b} S_{4b}(x,y) =  -\frac{1}{2} \left( \frac{2M}{1+a}
\right)^{2(M-N)+1} a^{-(2N-1)} \frac{(2N-1)!}{(2M-1)!} \\
\left\{ \vphantom{\int^{\infty}_x} \qLag{2N-1}(2Mx)x^{M-N-1/2}e^{-Mx}\psi_{2N-1}(y) \right. \\
+ \left. \psi'_{2N-1}(x) \int^{\infty}_y \qLag{2N-1}(2Mt)t^{M-N-1/2}e^{-Mt}dt \right\},
\end{multline}
and we can take the asymptotic analyses of $S_{4a1}(x,y)$,
$S_{4a2}(x,y)$ and $S_{4b}(x,y)$ separately.

%% file: proof_of_theorem_quater_rank_1.tex
\section{Proof of theorem \ref{theorem:main_theorem_for_quater_rank_1}} \label{Asymptotic analysis}

The same as in the complex case, we consider the rescaled distribution problem, and wish to find the probability of the largest sample eigenvalue being in the domain $(0,p+qT]$. We can put the kernel in the new coordinate system (after a conjugation by $\left( \begin{smallmatrix} q^{1/2} & 0 \\ 0 & q^{-1/2} \end{smallmatrix} \right)$), and get
\begin{equation}
\begin{split}
(\mathbb{P}(\max(\lambda_i) \leq p+qT))^2 = & \det \left( I -
\begin{pmatrix}
\widetilde{S}_4(\xi,\eta) & \widetilde{SD}_4(\xi,\eta) \\
\widetilde{IS}_4(\xi,\eta) & \widetilde{S}_4(\eta,\xi)
\end{pmatrix} \chi(\eta) \right) \\
= & \det(I-\mPtt(\xi, \eta)),
\end{split}
\end{equation}
where as $L^2$ functions,
\begin{align}
\widetilde{SD}_4(\xi,\eta) = & q^2SD_4(x,y) |_{\substack{x=p+q\xi \\ y=p+q\eta}}, \label{eq:transformation_of_SD_4} \\
\widetilde{S}_4(\xi,\eta) = & qS_4(x,y) |_{\substack{x=p+q\xi \\ y=p+q\eta}}, \label{eq:transformation_of_S_4}\\
\widetilde{IS}_4(\xi,\eta) = & IS_4(x,y) |_{\substack{x=p+q\xi \\ y=p+q\eta}}, \label{eq:transformation_of_IS_4}
\end{align}
and
\begin{equation}
\mPtt(\xi, \eta) = \chi(\xi) \begin{pmatrix}
\widetilde{S}_4(\xi,\eta) & \widetilde{SD}_4(\xi,\eta) \\
\widetilde{IS}_4(\xi,\eta) & \widetilde{S}_4(\eta,\xi)
\end{pmatrix} \chi(\eta).
\end{equation}

In the proof of theorem \ref{theorem:main_theorem_for_quater_rank_1}, we need the matrix version of propositions \ref{prop:trace_norm_convergence_implies_det_convergence}---\ref{prop:conjugation_formula_for_scalar_kernel}, and the fact that the convergence in trace norm of a matrix integral operator is equivalent to the convergence in trace norm of all its entries.

Since the $IS_4(x,y)$ and $DS_4(x,y)$ are of the same form as $S_4(x,y)$, we only show the asymptotic analysis of $S_4(x,y)$, and state the result for the other two, for which the arguments are the same.
We give proofs of all the three parts below.

\subsection{The $-1 < a < \gamma^{-1}$ part} \label{a<gamma^-1_part}

In case $-1 < a \leq \gamma^{-1}$, we choose $p = (1+\gamma^{-1})^2$ and $q = \frac{(1+\gamma)^{4/3}}{\gamma(2M)^{2/3}}$, and denote \footnote{Here $*$ stands for $4$, $4a$, $4a1$, $4a2$ and $4b$. The definition of $\widetilde{S}_*(\xi,\eta)$ in (\ref{eq:transform_formula_for_a_leq_gamma^-1}) is only used in subsection \ref{a<gamma^-1_part} and \ref{a=gamma^-1_part}.} 
\begin{equation} \label{eq:transform_formula_for_a_leq_gamma^-1} 
\left. \widetilde{S}_*(\xi,\eta) = \frac{(1+\gamma)^{4/3}}{\gamma(2M)^{2/3}} S_*(x,y) \right|_{\substack{x=(1+\gamma^{-1})^2 + \frac{(1+\gamma)^{4/3}}{\gamma(2M)^{2/3}} \xi \\ y=(1+\gamma^{-1})^2 + \frac{(1+\gamma)^{4/3}}{\gamma(2M)^{2/3}}\eta}}.
\end{equation}

$S_{4a}(x,y)$ is the formula for the upper-left entry of the $2 \times 2$ matrix kernel of the quaternionic white Wishart ensemble with parameters $M$ and $N-1$, and its asymptotic behavior is well studied \cite{Forrester-Nagao-Honner99}. We want to prove that as $M \rightarrow \infty$, $S_{4a}(x,y)$ dominates $S_4(x,y)$ in the domain that we are interested in, and so naturally the distribution of the largest sample eigenvalue in the perturbed problem is the same as that in the quaternionic white Wishart ensemble. (The difference between $N$ and $N-1$ is negligible.)

$S_{4a1}(x,y)$ is almost the kernel for the complex white Wishart ensemble with parameters $2M-2$ and $2N-2$, besides a factor $\sqrt{y/x}/2$. By arguments in subsection \ref{the max_lambda_<_gamma^-1_part_complex} we have
\begin{equation} \label{eq:s_a1 convergence} 
\lim_{M \rightarrow \infty} \chi(\xi) \widetilde{S}_{4a1}(\xi,\eta) \chi(\eta) = \frac{1}{2} \chi(\xi) K_{\Airy}(\xi, \eta) \chi(\eta).
\end{equation}

For the $S_{4a2}(x,y)$ part, we also have in trace norm \cite{Forrester-Nagao-Honner99},
\begin{equation} \label{eq:s_a2 convergence}
\lim_{M \rightarrow \infty} \chi(\xi) \widetilde{S}_{4a2}(\xi,\eta) \chi(\eta) = -\frac{1}{4} \chi(\xi) \Ai(\xi)\int^{\infty}_{\eta}\Ai(t)dt \chi(\eta).
\end{equation}
We get the proof of \eqref{eq:s_a2 convergence} by asymptotics analysis. Formula \eqref{eq:asymototics_Laguerre_subcritical} and a similar result for $\qLag{2N-2}$ imply the convergence in $L^2$ norm of functions in $\xi$ and respectively
$\eta$,
\begin{equation} \label{eq:first_fomula_for_S_4a2}
\lim_{M \rightarrow \infty} \gamma^{-2N}(1+\gamma)^{4/3}(2M)^{1/3}e^{M-N} \qLag{2N-2} (2Mx) x^{M-N-1/2}e^{-Mx} \chi(\xi) = \Ai(\xi) \chi(\xi),
\end{equation}
\begin{equation} \label{eq:second_formula_for_S_4a2}
\lim_{M \rightarrow \infty} \gamma^{-2N}2M e^{M-N} \int^{\infty}_y \qLag{2N-1} (2Mt) t^{M-N-1/2}e^{-Mt}dt \chi(\eta) = -\int^{\infty}_{\eta} \Ai(t)dt \chi(\eta),
\end{equation}
and by the Stirling's formula,
\begin{equation}
\lim_{M \rightarrow \infty} (2M)^{2(M-N)-1} \frac{(2N-1)!}{(2M-2)!} e^{2(N-M)} \gamma^{4N-1} = 1.
\end{equation}
By (\ref{eq:definition_of_S_4a2}) and (\ref{eq:transform_formula_for_a_leq_gamma^-1}), we get
\begin{multline}
\chi(\xi) \widetilde{S}_{4a2}(\xi,\eta) \chi(\eta) = \frac{1}{4} (2M)^{2(M-N)-1} \frac{(2N-1)!}{(2M-2)!} e^{2(N-M)} \gamma^{4N-1} \\
\gamma^{-2N}(1+\gamma)^{4/3}(2M)^{1/3}e^{M-N} \qLag{2N-2} (2Mx) x^{M-N-1/2}e^{-Mx} \chi(\xi) \\
\gamma^{-2N}2M e^{M-N} \int^{\infty}_y \qLag{2N-1} (2Mt) t^{M-N-1/2}e^{-Mt}dt \chi(\eta).
\end{multline}
Therefore we get the trace norm convergence \eqref{eq:s_a2 convergence} from the $L^2$ convergence (\ref{eq:first_fomula_for_S_4a2}) and (\ref{eq:second_formula_for_S_4a2}) by proposition \ref{prop:first_trace_norm_convergence_prop}.

Now we need to analyze the term $S_{4b}(\xi,\eta)$, new to the perturbed problem. We need the following results of $L^2$ convergence, which are direct consequences of \eqref{eq:asymototics_Laguerre_subcritical}, \eqref{eq:asymptotics_psi_2N-1_critical} and a similar result of $\psi'_{2N-1}$:
\begin{multline}
\lim_{M \rightarrow \infty} \gamma^{-2N-1}(1+\gamma)^{4/3}(2M)^{1/3}e^{M-N} \qLag{2N-1}(2Mx)x^{M-N-1/2}e^{-Mx} \chi(\xi) = \\
 -\Ai(\xi) \chi(\xi),
\end{multline}
\begin{align}
& \lim_{M \rightarrow \infty} \gamma^{-2N}2M e^{M-N} \int^{\infty}_y \qLag{2N-1} (2Mt) t^{M-N-1/2}e^{-Mt}dt \chi(\eta) = -\int^{\infty}_{\eta} \Ai(t)dt \chi(\eta), \\
& \lim_{M \rightarrow \infty} (1+a)^{2(N-M)-1}a^{-2N+1} \frac{(1-a\gamma)(2M)^{1/3}}{(\gamma+1)^{2/3}\gamma^{2N-1}} e^{M-N} \psi_{2N-1}(y)\chi(\eta) = \Ai(\eta)\chi(\eta), \\
& \lim_{M \rightarrow \infty} (1+a)^{2(N-M)-1}a^{-2N+1} \frac{(1-a\gamma)(\gamma+1)^{2/3}}{\gamma^{2N}(2M)^{1/3}} e^{M-N} \psi'_{2N-1}(y)\chi(\xi) = \Ai'(\xi)\chi(\xi).
\end{align}
By the Stirling's formula, we get
\begin{equation} \label{eq:stirling's}
\lim_{M\rightarrow \infty} (2M)^{2(M-N)}\frac{(2N-1)!}{(2M-1)!}e^{2(N-M)}\gamma^{4N-1} = 1,
\end{equation}
and then by (\ref{eq:formula_for_S_4b}), (\ref{eq:transform_formula_for_a_leq_gamma^-1}) and proposition \ref{prop:first_trace_norm_convergence_prop}, we have the convergence in trace norm
\begin{multline} \label{eq:agamma-1}
\lim_{M\rightarrow \infty} \frac{(1-a\gamma)(2M)^{1/3}}{(1+\gamma)^{2/3}} \chi(\xi) \widetilde{S}_{4b}(\xi, \eta) \chi(\eta) = \\
\frac{1}{2} \chi(\xi) \left( \Ai(\xi)\Ai(\eta) + \Ai'(\xi) \int^{\infty}_{\eta}\Ai(t)dt \right) \chi(\eta),
\end{multline}
which implies that in trace norm,
\begin{equation}
\lim_{M\rightarrow \infty} \chi(\xi) \widetilde{S}_{4b}(\xi, \eta)
\chi(\eta) = 0.
\end{equation}

Now we get the desired result
\begin{equation}
\lim_{M\rightarrow \infty} \chi(\xi) \widetilde{S}_4(\xi, \eta)
\chi(\eta) = \lim_{M\rightarrow \infty} \chi(\xi)
\widetilde{S}_{4a}(\xi, \eta) \chi(\eta) = \chi(\xi)
\widehat{S}_4(\xi,\eta) \chi(\eta),
\end{equation}
and in the same way
\begin{align}
\lim_{M\rightarrow \infty} \chi(\xi) \widetilde{SD}_4(\xi, \eta)
\chi(\eta) = & \chi(\xi) \widehat{SD}_4(\xi,\eta) \chi(\eta), \\
\lim_{M\rightarrow \infty} \chi(\xi) \widetilde{IS}_4(\xi, \eta)
\chi(\eta) = & \chi(\xi) \widehat{IS}_4(\xi,\eta) \chi(\eta).
\end{align}
 
Therefore, in trace norm
\begin{equation}
\lim_{M\rightarrow \infty} \mPtt(\xi, \eta) =
\chi(\xi)
\begin{pmatrix}
\widehat{S}_4(\xi,\eta) & \widehat{SD}_4(\xi,\eta) \\
\widehat{IS}_4(\xi,\eta) & \widehat{S}_4(\eta. \xi)
\end{pmatrix} \chi(\eta),
\end{equation}
and the convergence of Fredholm determinant follows.

\subsection{The $a = \gamma^{-1}$ part} \label{a=gamma^-1_part}

When $a=\gamma^{-1}$, the $1-a\gamma^{-1}$ in (\ref{eq:agamma-1}) vanishes, so we need other asymptotic formulas for $\psi_{2N-1}(\eta)$ and $\psi'_{2N-1}(\eta)$. The approach is similar to that in the $a<\gamma^{-1}$ case, with the same choice of $p$ and $q$. We need the $L^2$ convergence results given by \eqref{eq:asymptotics_psi_2N-1_critical} and similar results:
\begin{multline}
\lim_{M \rightarrow \infty} \gamma^{-2N-1}(1+\gamma)^{4/3}(2M)^{1/3}e^{M-N} e^{\xi/3} \qLag{2N-1}(2Mx)x^{M-N-1/2}e^{-Mx} \chi(\xi) = \\
-e^{\xi/3}\Ai(\xi) \chi(\xi),
\end{multline}
\begin{multline}
\lim_{M \rightarrow \infty} \gamma^{-2N}2M e^{M-N} e^{-\eta/3}\int^{\infty}_y \qLag{2N-1} (2Mt) t^{M-N-1/2}e^{-Mt}dt \chi(\eta) = \\
-e^{-\eta/3}\int^{\infty}_{\eta} \Ai(t)dt \chi(\eta),
\end{multline}
\begin{align} 
& \lim_{M \rightarrow \infty} (1+a)^{2(N-M)-1}a^{-2N+1} e^{M-N}\gamma^{-2N+1} e^{-\eta/3}\psi_{2N-1}(y)\chi(\eta) = e^{-\eta/3} s^{(1)}(\eta)\chi(\eta), \label{eq:first identity in second case} \\
& \lim_{M \rightarrow \infty} (1+a)^{2(N-M)-1}a^{-2N+1} e^{M-N}\frac{(\gamma+1)^{4/3}}{\gamma^{2N}(2M)^{-2/3}} e^{\xi/3}\psi'_{2N-1}(x)\chi(\xi) = e^{\xi/3} \Ai(\xi)\chi(\xi).
\end{align}

Now we conclude the proof of the $a = \gamma^{-1}$ part of theorem \ref{theorem:main_theorem_for_quater_rank_1}. Using (\ref{eq:formula_for_S_4b}), (\ref{eq:stirling's}) and proposition \ref{prop:first_trace_norm_convergence_prop} we have the convergence in trace norm
\begin{equation}
\begin{split}
\lim_{M\rightarrow \infty} \chi(\xi)e^{\xi/3} \widetilde{S}_{4b}(\xi, \eta) e^{-\eta/3}\chi(\eta) = & \frac{1}{2} \chi(\xi)e^{\xi/3} \left( \Ai(\xi)s^{(1)}(\eta) + \Ai(\xi)\int^{\infty}_{\eta}\Ai(t)dt \right) e^{-\eta/3}\chi(\eta) \\
= & \frac{1}{2} \chi(\xi)e^{\xi/3} \Ai(\xi) e^{-\eta/3}\chi(\eta),
\end{split}
\end{equation}
and this together with the conjugated convergence result of $\widetilde{S}_{4a}(\xi, \eta)$ in formulas (\ref{eq:s_a1 convergence}) and (\ref{eq:s_a2 convergence}) of subsection \ref{a<gamma^-1_part}, which can be proved by arguments in subsection \ref{the max_lambda_=_gamma^-1_part_complex}, conclude
\begin{equation}
\lim_{M\rightarrow \infty} \chi(\xi) e^{\xi/3} \widetilde{S}_4(\xi, \eta) e^{-\eta/3} \chi(\eta) = \chi(\xi) e^{\xi/3} \overline{\overline{S}}_4(\xi,\eta) e^{-\eta/3} \chi(\eta).
\end{equation}
In the same way we get
\begin{align}
\lim_{M\rightarrow \infty} \chi(\xi) e^{\xi/3} \widetilde{SD}_4(\xi, \eta) e^{\eta/3} \chi(\eta) = & \chi(\xi) e^{\xi/3} \overline{\overline{SD}}_4(\xi,\eta) e^{\eta/3} \chi(\eta), \\
\lim_{M\rightarrow \infty} \chi(\xi) e^{-\xi/3} \widetilde{IS}_4(\xi, \eta) e^{-\eta/3} \chi(\eta) = & \chi(\xi) e^{-\xi/3} \overline{\overline{IS}}_4(\xi,\eta) e^{-\eta/3} \chi (\eta).
\end{align}

Then we get the convergence in trace norm of a conjugate of $\chi(\xi) \mPtt(\xi, \eta) \chi(\eta)$
\begin{equation}
\begin{split}
& \lim_{M\rightarrow \infty} \chi(\xi)
\begin{pmatrix}
e^{\xi/3}\widetilde{S}_4(\xi,\eta)e^{-\eta/3} & e^{\xi/3}\widetilde{SD}_4(\xi,\eta)e^{\eta/3} \\
e^{-\xi/3}\widetilde{IS}_4(\xi,\eta)e^{-\eta/3} & e^{-\xi/3}\widetilde{S}_4(\eta,\xi)e^{\eta/3}
\end{pmatrix} \chi(\eta) \\
= & \chi(\xi)
\begin{pmatrix}
e^{\xi/3} \overline{\overline{S}}_4(\xi,\eta) e^{-\eta/3} & e^{\xi/3} \overline{\overline{SD}}_4(\xi,\eta) e^{\eta/3} \\
e^{-\xi/3} \overline{\overline{IS}}_4(\xi,\eta) e^{-\eta/3} & e^{-\xi/3} \overline{\overline{S}}_4(\eta,\xi) e^{\eta/3}
\end{pmatrix} \chi(\eta) \\
= & \chi(\xi)
\begin{pmatrix}
e^{\xi/3} & 0 \\
0 & e^{-\xi/3}
\end{pmatrix}
\begin{pmatrix}
\overline{\overline{S}}_4(\xi,\eta) & \overline{\overline{SD}}_4(\xi,\eta) \\
\overline{\overline{IS}}_4(\xi,\eta) & \overline{\overline{S}}_4(\eta,\xi)
\end{pmatrix} 
\begin{pmatrix}
e^{-\eta/3} & 0 \\
0 & e^{\eta/3}
\end{pmatrix}
\chi(\eta).
\end{split}
\end{equation}
Therefore,
\begin{equation}
\begin{split}
& \lim_{M\rightarrow \infty} \det(I - \mPtt(\xi, \eta)) \\
= & \det \left( I - \chi(\xi)
\begin{pmatrix}
e^{\xi/3} \overline{\overline{S}}_4(\xi,\eta) e^{-\eta/3} & e^{\xi/3} \overline{\overline{SD}}_4(\xi,\eta) e^{\eta/3} \\
e^{-\xi/3} \overline{\overline{IS}}_4(\xi,\eta) e^{-\eta/3} & e^{-\xi/3} \overline{\overline{S}}_4(\eta,\xi) e^{\eta/3}
\end{pmatrix} \chi(\eta) \right) \\
= & \det \left( I - \chi(\xi) 
\begin{pmatrix}
\overline{\overline{S}}_4(\xi,\eta) & \overline{\overline{SD}}_4(\xi,\eta) \\
\overline{\overline{IS}}_4(\xi,\eta) & \overline{\overline{S}}_4(\eta,\xi)
\end{pmatrix}  \chi(\eta) \right),
\end{split}
\end{equation}
we we use the matrix version of proposition \ref{prop:conjugation_formula_for_scalar_kernel}.

\subsection{The $a > \gamma^{-1}$ part}

If $a>\gamma^{-1}$, the location as well as the fluctuation scale of the largest sample eigenvalue is changed. We change variables as $p = (a+1)\left( 1+\frac{1}{\gamma^2a} \right)$ and $q = (a+1)
 \sqrt{1 - \frac{1}{\gamma^2a^2}} \frac{1}{\sqrt{2M}}$, and then by (\ref{eq:transformation_of_S_4}) the kernel $S_*(x,y)$ after substitution is \footnote{Here $*$ stands for $4$, $4a$ or $4b$, and the $\widetilde{S}_*(\xi,\eta)$ in this subsection is not identical to that in subsection \ref{a<gamma^-1_part} and \ref{a=gamma^-1_part}.}
\begin{equation} \label{eq:another_transform_formula}
\widetilde{S}_*(\xi,\eta) = \left. (a+1) \sqrt{1 - \frac{1}{\gamma^2a^2}} \frac{1}{\sqrt{2M}} S_*(x,y) \right|_{\substack{x=(a+1)\left( 1+\frac{1}{\gamma^2a} \right) + (a+1) \sqrt{1 - \frac{1}{\gamma^2a^2}} \frac{1}{\sqrt{2M}} \xi \\ y=(a+1)\left( 1+\frac{1}{\gamma^2a} \right) + (a+1) \sqrt{1 - \frac{1}{\gamma^2a^2}} \frac{1}{\sqrt{2M}}\eta}}.
\end{equation}
We analyze $\widetilde{S}_{4b}(\xi,\eta)$ first. As before, we have $L^2$ convergence results by \eqref{eq:asymototics_Laguerre_supercritical} and similar results
\begin{multline} \label{eq:first_in_third_case}
\lim_{M \rightarrow \infty} \frac{(\gamma^2a+1)^{M-N+1/2}}
{(\gamma^2a)^{M+N+1/2}(a+1)^{M-N-1/2}} \sqrt{(\gamma^2a^2-1)2M}
e^{M-N} e^{\frac{\gamma^2a^2-1} {(\gamma^2a+1)(a+1)}Mx} \\
e^{2\xi/3} \qLag{2N-1}(2Mx)x^{M-N-1/2}e^{-Mx} \chi(\xi) =
-\frac{1}{\sqrt{2\pi}} e^{-\frac{1}{4}
\frac{\gamma^4a^2+\gamma^2a^2+4\gamma^2a+\gamma^2+1}
{(\gamma^2a+1)^2}\xi^2 + 2\xi/3} \chi(\xi),
\end{multline}
\begin{multline} \label{eq:second_in_third_case}
\lim_{M \rightarrow \infty} \frac{1}{2}
\frac{(\gamma^2a+1)^{M-N-1/2}(\gamma^2a^2-1)}
{(\gamma^2a)^{M+N+1/2}(a+1)^{M-N+1/2}} \sqrt{\gamma^2a^2-1}
(2M)^{3/2} e^{M-N} e^{\frac{\gamma^2a^2-1}
{(\gamma^2a+1)(a+1)}My} \\
e^{2\eta/3} \int^{\infty}_y \qLag{2N-1}(2Mt)
t^{M-N-1/2}e^{-Mt}dt \chi(\eta) = \\
-\frac{1}{\sqrt{2\pi}}
e^{-\frac{1}{4} \frac{\gamma^4a^2+\gamma^2a^2+4\gamma^2a+\gamma^2+1}
{(\gamma^2a+1)^2}\eta^2 + 2\eta/3} \chi(\eta),
\end{multline}
\begin{multline} \label{eq:third in third case}
\lim_{M \rightarrow \infty} \left(
\frac{\gamma^2a}{(\gamma^2a+1)(a+1)} \right)^{M-N+1/2}e^{M-N}
e^{-\frac{\gamma^2a^2-1}{(\gamma^2a+1)(a+1)}My} \\
e^{-2\eta/3} \psi_{2N-1}(y) \chi(\eta) =
e^{-\frac{1}{4}\frac{(\gamma^2a^2-1)(\gamma^2-1)}
{(\gamma^2a+1)^2}\eta^2 - 2\eta/3} \chi(\eta),
\end{multline}
\begin{multline} \label{eq:forth_in_third_case}
\lim_{M \rightarrow \infty} \left(
\frac{\gamma^2a}{(\gamma^2a+1)(a+1)} \right)^{M-N-1/2}e^{M-N}
\frac{\gamma^2a}{(\gamma^2a^2-1)M} e^{-\frac{\gamma^2a^2-1}
{(\gamma^2a+1)(a+1)}Mx} \\
e^{-2\xi/3} \psi'_{2N-1}(x) \chi(\xi) =
e^{-\frac{1}{4}\frac{(\gamma^2a^2-1)(\gamma^2-1)}
{(\gamma^2a+1)^2}\xi^2 - 2\xi/3} \chi(\xi).
\end{multline}

For notational simplicity, we denote functions on the left-hand
sides of (\ref{eq:first_in_third_case}) --
(\ref{eq:forth_in_third_case}) by $F_1(\xi)$, $F_2(\eta)$,
$F_3(\eta)$ and $F_4(\xi)$, and denote
\begin{equation}
c_M = (2M)^{2(M-N)}\frac{(2N-1)!}{(2M-1)!}e^{2(N-M)}\gamma^{4N-1}
\end{equation}
By (\ref{eq:stirling's}), we have
\begin{equation}
\lim_{M \rightarrow \infty} c_M = 1.
\end{equation}
Then we get from (\ref{eq:formula_for_S_4b}),
(\ref{eq:another_transform_formula}) and
(\ref{eq:first_in_third_case})---(\ref{eq:forth_in_third_case})
\begin{multline} \label{eq:F_expression_of_S4b}
\widetilde{S}_{4b}(\xi, \eta) = -\frac{c_M}{2} \left(
e^{-\frac{\gamma^2a^2-1}{(\gamma^2a+1)(a+1)}M(x-y) -
2(\xi-\eta)/3} F_1(\xi)F_3(\eta) \right. \\
\left. + e^{\frac{\gamma^2a^2-1}{(\gamma^2a+1)(a+1)}M(x-y) +
2(\xi-\eta)/3} F_4(\xi)F_2(\eta) \right).
\end{multline}

If we define
\begin{align}
SD_{4a}(x,y) = & \sum^{N-2}_{j=0} \frac{1}{r_j}
(\psi'_{2j}(x)\psi'_{2j+1}(y) - \psi'_{2j+1}(x)\psi'_{2j}(y)), \\
IS_{4a}(x,y) = & \sum^{N-2}_{j=0} \frac{1}{r_j}
(-\psi_{2j}(x)\psi_{2j+1}(y) + \psi_{2j+1}(x)\psi_{2j}(y)),
\end{align}
and
\begin{align}
SD_{4b}(x,y) = & \frac{1}{r_{N-1}}
(\psi'_{2N-2}(x)\psi'_{2N-1}(y) - \psi'_{2N-1}(x)\psi'_{2N-2}(y)), \\
IS_{4b}(x,y) = & \frac{1}{r_{N-1}} (-\psi_{2N-2}(x)\psi_{2N-1}(y) +
\psi_{2N-1}(x)\psi_{2N-2}(y)),
\end{align}
and by (\ref{eq:transformation_of_SD_4}) and
(\ref{eq:transformation_of_IS_4}) like
(\ref{eq:another_transform_formula}) \footnote{$*$ stands for $4$,
$4a$ or $4b$.}
\begin{align}
\widetilde{SD}_*(\xi,\eta) = & \left. (a+1)^2 \left( 1 -
\frac{1}{\gamma^2a^2} \right) \frac{1}{2M} SD_*(x,y)
\right|_{\substack{x=(a+1)\left( 1+\frac{1}{\gamma^2a} \right) +
(a+1) \sqrt{1 - \frac{1}{\gamma^2a^2}} \frac{1}{\sqrt{2M}} \xi
\\ y=(a+1)\left( 1+\frac{1}{\gamma^2a} \right) + (a+1)
\sqrt{1 - \frac{1}{\gamma^2a^2}} \frac{1}{\sqrt{2M}}\eta}}, \\
\widetilde{IS}_*(\xi,\eta) = & \left. IS_*(x,y)
\right|_{\substack{x=(a+1)\left( 1+\frac{1}{\gamma^2a} \right) +
(a+1) \sqrt{1 - \frac{1}{\gamma^2a^2}} \frac{1}{\sqrt{2M}} \xi
\\ y=(a+1)\left( 1+\frac{1}{\gamma^2a} \right) + (a+1)
\sqrt{1 - \frac{1}{\gamma^2a^2}} \frac{1}{\sqrt{2M}}\eta}},
\end{align}
then in the same way of (\ref{eq:F_expression_of_S4b}), we have
\begin{multline}
\widetilde{SD}_{4b}(\xi,\eta) = \frac{c_M}{4} C_M \left(
e^{-\frac{\gamma^2a^2-1}{(\gamma^2a+1)(a+1)}M(x-y) -
2(\xi-\eta)/3} F_1(\xi)F_4(\eta) \right. \\
\left. - e^{\frac{\gamma^2a^2-1}{(\gamma^2a+1)(a+1)}M(x-y) +
2(\xi-\eta)/3} F_4(\xi)F_1(\eta) \right),
\end{multline}
\begin{multline}
\widetilde{IS}_{4b}(\xi,\eta) = \frac{c_M}{C_M} \left(
e^{-\frac{\gamma^2a^2-1}{(\gamma^2a+1)(a+1)}M(x-y) -
2(\xi-\eta)/3} F_2(\xi)F_3(\eta) \right. \\
\left. - e^{\frac{\gamma^2a^2-1}{(\gamma^2a+1)(a+1)}M(x-y) +
2(\xi-\eta)/3} F_3(\xi)F_2(\eta) \right),
\end{multline}
with
\begin{equation}
C_M = \frac{(\gamma^2a^2-1)^{3/2} \sqrt{2M}}{a\gamma(\gamma^2a+1)}.
\end{equation}

Now we write $\widetilde{P}_T(\xi,\eta)$ as the sum
\begin{equation}
\widetilde{P}_T(\xi,\eta) = \widetilde{P}_{Ta}(\xi,\eta) +
\widetilde{P}_{Tb}(\xi,\eta),
\end{equation}
with
\begin{align}
\widetilde{P}_{Ta}(\xi,\eta) = & \chi(\eta)
\begin{pmatrix}
\widetilde{S}_{4a}(\xi,\eta) & \widetilde{SD}_{4a}(\xi,\eta) \\
\widetilde{IS}_{4a}(\xi,\eta) & \widetilde{S}_{4a}(\eta,\xi)
\end{pmatrix} \chi(\eta), \\
\widetilde{P}_{Tb}(\xi,\eta) = & \chi(\eta)
\begin{pmatrix}
\widetilde{S}_{4b}(\xi,\eta) & \widetilde{SD}_{4b}(\xi,\eta) \\
\widetilde{IS}_{4b}(\xi,\eta) & \widetilde{S}_{4b}(\eta,\xi)
\end{pmatrix} \chi(\eta).
\end{align}
If we denote
\begin{align}
U(\xi) = &
\begin{pmatrix}
e^{\frac{\gamma^2a^2-1}{(\gamma^2a+1)(a+1)}Mx + 2\xi/3} &
-\frac{C_M}{2}\frac{F_4(\xi)}{F_3(\xi)}
e^{\frac{\gamma^2a^2-1}{(\gamma^2a+1)(a+1)}Mx + 2\xi/3} \\
0 & e^{-\frac{\gamma^2a^2-1}{(\gamma^2a+1)(a+1)}Mx - 2\xi/3}
\end{pmatrix}, \\
U^{-1}(\eta) = &
\begin{pmatrix}
e^{-\frac{\gamma^2a^2-1}{(\gamma^2a+1)(a+1)}My - 2\eta/3} &
\frac{C_M}{2}\frac{F_4(\xi)}{F_3(\xi)}
e^{-\frac{\gamma^2a^2-1}{(\gamma^2a+1)(a+1)}Mx - 2\xi/3} \\
0 & e^{\frac{\gamma^2a^2-1}{(\gamma^2a+1)(a+1)}Mx + 2\xi/3}
\end{pmatrix},
\end{align}
then we have the result of kernel conjugation
\begin{multline}
U(\xi)\widetilde{P}_{Tb}(\xi,\eta)U^{-1}(\eta) = \\
\chi(\xi)
\begin{pmatrix}
-\frac{c_M}{2} \left( F_1(\xi) +
\frac{F_2(\xi)F_4(\xi)}{F_3(\xi)}
\right) F_3(\eta) & 0 \\
U(\xi)\widetilde{P}_{Tb}(\xi,\eta)U^{-1}(\eta)_{21} & -\frac{c_M}{2}
F_3(\xi) \left( F_1(\eta) + \frac{F_2(\eta)F_4(\eta)}{F_3(\eta)}
\right)
\end{pmatrix} \chi(\eta),
\end{multline}
with the entry
\begin{multline}
U(\xi)\widetilde{P}_{Tb}(\xi,\eta)U^{-1}(\eta)_{21} = \\
\frac{c_M}{C_M} \left(
e^{-\frac{2(\gamma^2a^2-1)}{(\gamma^2a+1)(a+1)}Mx - 4\xi/3}
F_2(\xi)F_3(\eta) - F_3(\xi)F_2(\eta)
e^{-\frac{2(\gamma^2a^2-1)}{(\gamma^2a+1)(a+1)}My - 4\eta/3}
\right).
\end{multline}

We want $U(\xi)\widetilde{P}_{Tb}(\xi,\eta)U^{-1}(\eta)$ to converge
in trace norm as $M \rightarrow \infty$, and need the result
\begin{lemma} \label{lemma:F2F4_F3}
In $L^2$ norm,
\begin{equation}
\lim_{M \rightarrow \infty}
\frac{F_2(\xi)F_4(\xi)}{F_3(\xi)}\chi(\xi) = -\frac{1}{\sqrt{2\pi}}
e^{-\frac{1}{4} \frac{\gamma^4a^2+\gamma^2a^2+4\gamma^2a+\gamma^2+1}
{(\gamma^2a+1)^2}\xi^2 + 2\xi/3} \chi(\xi).
\end{equation}
\end{lemma}
The proof is left to the reader. The main ingredient is
(\ref{eq:second_in_third_case}) and the fact that
$F_4(\xi)/F_3(\xi)$ approaches to $1$ uniformly on $[T, \infty)$.

We need another convergence result on
$U(\xi)\widetilde{P}_{Ta}(\xi,\eta)U^{-1}(\eta)$:
\begin{prop} \label{prop:convergence_of_P_Ta}
In trace norm,
\begin{equation} \label{eq:convergence_of_P_Ta}
\lim_{M \rightarrow \infty}
U(\xi)\widetilde{P}_{Ta}(\xi,\eta)U^{-1}(\eta) = 0.
\end{equation}
\end{prop}
The proof is left to the reader. Since all the four entries in
$\widetilde{P}_{Ta}(\xi,\eta)$ can be expressed by Laguerre
polynomials, the asymptotic results like
(\ref{eq:first_in_third_case}) and (\ref{eq:second_in_third_case})
give the convergence (\ref{eq:convergence_of_P_Ta}).

By lemma \ref{lemma:F2F4_F3} and proposition
\ref{prop:convergence_of_P_Ta}, we get in trace norm
\begin{equation}
\begin{split}
& \lim_{M \rightarrow \infty} \det(I - \widetilde{P}_T(\xi,\eta)) \\
= & \lim_{M \rightarrow \infty} \det(I -
U(\xi)\widetilde{P}_T(\xi,\eta)U^{-1}(\eta)) \\
= & \lim_{M \rightarrow \infty} \det(I -
U(\xi)\widetilde{P}_{Tb}(\xi,\eta)U^{-1}(\eta)) \\
= & \left( \int^T_{\infty} \frac{1}{\sqrt{2\pi}} e^{-\frac{t^2}{2}}
dt \right)^2,
\end{split}
\end{equation}
and we get the proof of the $a > \gamma^{-1}$ part of theorem
\ref{theorem:main_theorem_for_quater_rank_1}.

%% file: interpolations.tex
\chapter{Phase transition phenomenon}

\label{interpolations}

\section{Rank $1$ complex spiked model}

Here we assume that the single spiked population eigenvalue is $1+a = 1 + \gamma^{-1}$, and by the part 2 of theorem \ref{theorem:spiked_model_theorem_for_complex}, we know the distribution function of the largest sample eigenvalue is $F_{\GUE1}$, which is, according to \eqref{eq:expression_of_t_1}, defined as
\begin{equation} \label{eq:simplified_F_GUE1}
F_{\GUE1}(T) = \det\left( 1 - \chi(\xi) \left( K_{\Airy}(\xi, \eta) + \Ai(\xi)s^{(1)}(\eta) \right) \chi(\eta) \right).
\end{equation}
Forrester recognized that \cite{Forrester00}
\begin{equation} \label{eq:F_GUE1=F^2_GOE_by_Forrester}
F_{\GUE1}(T) = F^2_{\GOE}(T).
\end{equation}
Therefore as the perturbative parameter $a$ increases, by \eqref{eq:G_1_is_Gaussian} and \eqref{eq:F_GUE1=F^2_GOE_by_Forrester} we have the $F_{\GUE}$---$F^2_{\GOE}$---Gaussian phase transition phenomenon around $a = \gamma^{-1}$.

\section{Rank $1$ quaternionic spiked model}

Again we assume that the single spiked population eigenvalue is $1+a = 1 + \gamma^{-1}$. As the perturbative parameter $a$ increases, we have the  $F_{\GSE}$---$F_{\GOE}$---Gaussian phase transition phenomenon around $a = \gamma^{-1}$, by results of theorem \ref{theorem:main_theorem_for_quater_rank_1} and \ref{theorem:F_GSE1=F_GOE}. In this section we prove theorem \ref{theorem:F_GSE1=F_GOE}.

In manipulation of kernels, we follow the method of
\cite{Tracy-Widom96}. The procedure seems informal and cursory, but
is carefully justified in \cite{Tracy-Widom96}.

For notational simplicity, we denote ($\chi(\xi) = \chi_{(T,
\infty)}(\xi)$)
\begin{equation}
B(\xi) = 1 - s^{(1)}(\xi) = \int^{\infty}_{\xi} \Ai(t)dt.
\end{equation}

First, we express the integral operator
\begin{equation}
\chi(\xi) \overline{\overline{P}}(\xi, \eta) \chi(\eta) =
\begin{pmatrix}
\chi(\xi) \overline{\overline{S}}_4(\xi, \eta) \chi(\eta) &
\chi(\xi) \overline{\overline{SD}}_4(\xi, \eta) \chi(\eta) \\
\chi(\xi) \overline{\overline{IS}}_4(\xi, \eta) \chi(\eta) &
\chi(\xi) \overline{\overline{S}}_4(\eta, \xi,) \chi(\eta)
\end{pmatrix}
\end{equation}
by
\begin{equation}
\begin{pmatrix}
\chi(\xi) \frac{\partial}{\partial \xi} & 0 \\
0 & \chi(\xi)
\end{pmatrix}
\begin{pmatrix}
\overline{\overline{IS}}_4(\xi, \eta) \chi(\eta) &
\overline{\overline{S}}_4(\eta, \xi) \chi(\eta) \\
\overline{\overline{IS}}_4(\xi, \eta) \chi(\eta) &
\overline{\overline{S}}_4(\eta, \xi,) \chi(\eta)
\end{pmatrix},
\end{equation}
since by (\ref{eq:definition_of_S_4(xi,eta)}) --
(\ref{eq:definition_of_S_4(eta,xi)}) and taking limit,
\begin{align}
\frac{\partial}{\partial \xi} \overline{\overline{IS}}_4 (\xi, \eta)
= & \overline{\overline{S}}_4(\xi, \eta), \\
\frac{\partial}{\partial \xi} \overline{\overline{S}}_4 (\eta, \xi)
= & \overline{\overline{SD}}_4(\xi, \eta).
\end{align}
Then using (\ref{eq:det(I-AB)=det(I-BA)}) for $A$ bounded and $B$
trace class, upon suitably defining the Hilbert spaces our operators
$A$ and $B$ are acting on, we find
\begin{align}
& \det \left( I -
\begin{pmatrix}
\chi(\xi) \frac{\partial}{\partial \xi} & 0 \\
0 & \chi(\xi)
\end{pmatrix}
\begin{pmatrix}
\overline{\overline{IS}}_4(\xi, \eta) \chi(\eta) &
\overline{\overline{S}}_4(\eta, \xi) \chi(\eta) \\
\overline{\overline{IS}}_4(\xi, \eta) \chi(\eta) &
\overline{\overline{S}}_4(\eta, \xi,) \chi(\eta)
\end{pmatrix}
\right) \notag \\
= & \det \left( I -
\begin{pmatrix}
\overline{\overline{IS}}_4(\xi, \eta) \chi(\eta) &
\overline{\overline{S}}_4(\eta, \xi) \chi(\eta) \\
\overline{\overline{IS}}_4(\xi, \eta) \chi(\eta) &
\overline{\overline{S}}_4(\eta, \xi,) \chi(\eta)
\end{pmatrix}
\begin{pmatrix}
\chi(\eta) \frac{\partial}{\partial \eta} & 0 \\
0 & \chi(\eta)
\end{pmatrix}
\right) \notag \\
= & \det \left( I -
\begin{pmatrix}
\overline{\overline{IS}}_4(\xi, \eta) \chi(\eta)
\frac{\partial}{\partial \eta}&
\overline{\overline{S}}_4(\eta, \xi) \chi(\eta) \\
\overline{\overline{IS}}_4(\xi, \eta) \chi(\eta)
\frac{\partial}{\partial \eta} & \overline{\overline{S}}_4(\eta,
\xi,) \chi(\eta)
\end{pmatrix}
\right), \notag \\
\intertext{and by conjugation with $\begin{pmatrix} 1 & 0 \\ -1 & 1
\end{pmatrix}$, we get} = & \det \left( I -
\begin{pmatrix}
\overline{\overline{IS}}_4(\xi, \eta) \chi(\eta)
\frac{\partial}{\partial \eta} + \overline{\overline{S}}_4(\eta,
\xi) \chi(\eta) &
\overline{\overline{S}}_4(\eta, \xi) \chi(\eta) \\
0 & 0
\end{pmatrix}
\right) \notag \\
= & \det \left( I - \left( \overline{\overline{IS}}_4(\xi, \eta)
\chi(\eta) \frac{\partial}{\partial \eta} +
\overline{\overline{S}}_4(\eta, \xi) \chi(\eta) \right) \right).
\end{align}

Since
\begin{equation}
\int^{\infty}_T \overline{\overline{IS}}_4(\xi, \eta)
\frac{\partial}{\partial \eta} f(\eta) d\eta = \left.
\overline{\overline{IS}}_4(\xi, \eta) f(\eta) \right|^{\eta =
\infty}_{\eta = T} - \int^{\infty}_T \frac{\partial}{\partial \eta}
\overline{\overline{IS}}_4(\xi, \eta) f(\eta) d\eta,
\end{equation}
as an operator
\begin{equation}
\overline{\overline{IS}}_4(\xi, \eta) \chi(\eta)
\frac{\partial}{\partial \eta} = \overline{\overline{IS}}_4(\xi,
\infty) \delta_{\infty}(\eta) - \overline{\overline{IS}}_4(\xi, T)
\delta_{T}(\eta) - \frac{\partial}{\partial \eta}
\overline{\overline{IS}}_4(\xi, \eta) \chi(\eta),
\end{equation}
where $\delta_{\infty}$ and $\delta_{T}$ are (generalized) Dirac
functions. Then with the help of identity
\begin{equation}
\int^{\infty}_{\xi} K_{\Airy}(t,\eta)dt + \int^{\infty}_{\eta}
K_{\Airy}(\xi,t)dt = \int^{\infty}_{\xi} \Ai(t)dt
\int^{\infty}_{\xi} \Ai(t)dt,
\end{equation}
which can be proved directly from
(\ref{eq:definition_of_the_scalar_Airy_kernel}), we get
\begin{equation}
\begin{split}
I & - \left( \overline{\overline{IS}}_4(\xi, \eta) \chi(\eta)
\frac{\partial}{\partial \eta} + \overline{\overline{S}}_4(\eta,
\xi) \chi(\eta) \right) \\
= I & - \left( K_{\Airy}(\xi, \eta) - \frac{1}{2} B(\xi)\Ai(\eta)
+ \Ai(\eta) \right)\chi(\eta) \\
& + \left( \frac{1}{2} \int^{\infty}_T K_{\Airy}(\xi, t)dt -
\frac{1}{4} B(T)B(\xi) - \frac{1}{2} B(\xi) + \frac{1}{2} B(T)
\right) \delta_T(\eta) \\
& + \frac{1}{2} B(\xi) \delta_{\infty}(\eta).
\end{split}
\end{equation}

Now we denote $R(\xi, \eta)$ as the resolvent of $K_{\Airy}(\xi,
\eta)\chi(\eta)$, such that as integral operators
\begin{equation} \label{eq:definition_of_resolvent}
I + R(\xi, \eta) = (I-K_{\Airy}(\xi, \eta)\chi(\eta))^{-1},
\end{equation}
then
\begin{equation}
\begin{split}
& I - \left( \overline{\overline{IS}}_4(\xi, \eta) \chi(\eta)
\frac{\partial}{\partial \eta} + \overline{\overline{S}}_4(\eta,
\xi) \chi(\eta) \right) \\
= & (I-K_{\Airy}(\xi, \eta)\chi(\eta)) \left(
I - (I+R)(1 - \frac{1}{2} B(\xi))\Ai(\eta)\chi(\eta) \right. \\
& + (I+R) \left( \frac{1}{2} \int^{\infty}_T K_{\Airy}(\xi, t)dt -
\frac{1}{4} B(T)B(\xi) - \frac{1}{2} B(\xi) + \frac{1}{2} B(T)
\right) \delta_T(\eta) \\
& \left. + \frac{1}{2}(I+R)B(\xi) \delta_{\infty}(\eta)
\vphantom{\frac{1}{2}} \right).
\end{split}
\end{equation}

Again by the formula (\ref{eq:det(I-AB)=det(I-BA)}), in the form of
(formula (17) in \cite{Tracy-Widom96})
\begin{equation} \label{eq:another_form_of_det_identity}
\det \left( I - \sum^n_{k=1} \alpha_k \otimes \beta_k \right) = \det
\left( \delta_{j,k} - (\alpha_j, \beta_k) \right)_{j,k = 1, \dots,
n}
\end{equation}
we get
\begin{equation}
\begin{split}
\det & \left( I - (I+R)(1 - \frac{1}{2} B(\xi))\Ai(\eta)\chi(\eta)
\right. \\
& + (I+R) \left( \frac{1}{2} \int^{\infty}_T K_{\Airy}(\xi, t)dt -
\frac{1}{4} B(T)B(\xi) - \frac{1}{2} B(\xi) + \frac{1}{2} B(T)
\right) \delta_T(\eta) \\
& \left. + \frac{1}{2}(I+R)B(\xi) \delta_{\infty}(\eta)
\vphantom{\frac{1}{2}} \right)\\
= \det &
\begin{pmatrix}
1+\alpha_{11} & \alpha_{12} & \alpha_{13} \\
\alpha_{21} & 1+\alpha_{22} & \alpha_{23} \\
\alpha_{31} & \alpha_{32} & 1+\alpha_{33}
\end{pmatrix},
\end{split}
\end{equation}
where upon the definition
\begin{equation}
\langle f(\xi), g(\xi) \rangle_T = \int^{\infty}_T f(\xi)g(\xi)
d\xi,
\end{equation}
we have
\begin{align}
\alpha_{11} = & \left\langle (I+R)(1 - \frac{1}{2} B(\xi)),
-\Ai(\xi)
\right\rangle_T, \\
\alpha_{12} = & \left\langle (I+R) \left( \frac{1}{2}
\int^{\infty}_T K_{\Airy}(\xi, t)dt - \frac{1}{4} B(T)B(\xi) -
\frac{1}{2} B(\xi) + \frac{1}{2} B(T) \right), -\Ai(\xi)
\right\rangle_T, \\
\alpha_{13} = & \left\langle \frac{1}{2}(I+R)B(\xi), -\Ai(\xi)
\right\rangle_T, \\
\alpha_{21} = & \left. (I+R)(1 - \frac{1}{2} B(\xi)) \right|_{\xi =
T}, \\
\alpha_{22} = & \left. (I+R) \left( \frac{1}{2} \int^{\infty}_T
K_{\Airy}(\xi, t)dt - \frac{1}{4} B(T)B(\xi) - \frac{1}{2} B(\xi) +
\frac{1}{2} B(T) \right) \right|_{\xi = T}, \\
\alpha_{23} = & \left. \frac{1}{2}(I+R)B(\xi) \right|_{\xi =
T}, \\
\alpha_{31} = & \left. (I+R)(1 - \frac{1}{2} B(\xi)) \right|_{\xi =
\infty} = 1, \\
\alpha_{32} = & \left. (I+R) \left( \frac{1}{2} \int^{\infty}_T
K_{\Airy}(\xi, t)dt - \frac{1}{4} B(T)B(\xi) - \frac{1}{2} B(\xi) +
\frac{1}{2} B(T) \right) \right|_{\xi = \infty} \notag \\
= & \frac{1}{2} B(T). \\
\alpha_{33} = & \left. \frac{1}{2}(I+R)B(\xi) \right|_{\xi = \infty}
= 0,
\end{align}

If we take elementary row operations, we get
\begin{equation}
\begin{split}
& \det
\begin{pmatrix}
1+\alpha_{11} & \alpha_{12} & \alpha_{13} \\
\alpha_{21} & 1+\alpha_{22} & \alpha_{23} \\
\alpha_{31} & \alpha_{32} & 1+\alpha_{33}
\end{pmatrix} \\
= & \det
\begin{pmatrix}
1+\alpha_{11} - \alpha_{13} & \alpha_{12} - \frac{1}{2}B(T) \alpha_{13} & \alpha_{13} \\
\alpha_{21} - \alpha_{23} & 1+\alpha_{22} - \frac{1}{2}B(T) \alpha_{23} & \alpha_{23} \\
0 & 0 & 1
\end{pmatrix} \\
= & \det
\begin{pmatrix}
1+\beta_{11} & \beta_{12} \\
\beta_{21} & 1+ \beta_{22}
\end{pmatrix},
\end{split}
\end{equation}
where
\begin{align}
\beta_{11} = & \left\langle (I+R)(1 - B(\xi)), -\Ai(\xi)
\right\rangle_T, \\
\beta_{12} = & \left\langle \frac{1}{2} (I+R)\left( \int^{\infty}_T
K_{\Airy}(\xi, t)dt - B(T)B(\xi) - B(\xi) + B(T) \right),
-\Ai(\xi) \right\rangle_T, \\
\beta_{21} = & \left. (I+R)(1 - B(\xi))
\right|_{\xi=T}, \\
\beta_{22} = & \left. \frac{1}{2} (I+R)\left( \int^{\infty}_T
K_{\Airy}(\xi, t)dt - B(T)B(\xi) - B(\xi) + B(T) \right)
\right|_{\xi=T}.
\end{align}

Using (\ref{eq:definition_of_resolvent}) and
(\ref{eq:another_form_of_det_identity}), we observe ($s^{(1)}(\xi) = 1 -
B(\xi)$)
\begin{equation}
\begin{split}
\det & (I-K_{\Airy}(\xi, \eta)\chi(\eta)) \det
\begin{pmatrix}
1+\beta_{11} & \beta_{12} \\
\beta_{21} & 1+ \beta_{22}
\end{pmatrix} \\
= \det & \left( I-(K_{\Airy}(\xi, \eta)\chi(\eta) +
s^{(1)}(\xi)\Ai(\eta))\chi(\eta) \vphantom{\frac{1}{2}} \right. \\
& \left. + \frac{1}{2} \left( \int^{\infty}_T K_{\Airy}(\xi, t)dt -
B(T)B(\xi) - B(\xi) + B(T) \right) \delta_T(\eta) \right).
\end{split}
\end{equation}

If we denote $\tilde{R}(\xi, \eta)$ as the resolvent of
$(K_{\Airy}(\xi, \eta)\chi(\eta) + s^{(1)}(\xi)\Ai(\eta))\chi(\eta)$,
so that as operators
\begin{equation}
I + \tilde{R}(\xi, \eta) = \left(I + (K_{\Airy}(\xi,
\eta)\chi(\eta) + s^{(1)}(\xi)\Ai(\eta))\chi(\eta) \right)^{-1},
\end{equation}
and
\begin{equation}
Q(\xi) = (I+\tilde{R}) \left( \int^{\infty}_T K_{\Airy}(\xi, t)dt -
B(T)B(\xi) - B(\xi) + B(T) \right),
\end{equation}
then
\begin{equation}
\begin{split}
F_{\GSE1} = & \det \left( I-(K_{\Airy}(\xi, \eta)\chi(\eta) +
s^{(1)}(\xi)\Ai(\eta))\chi(\eta) \right) \det \left( I + \frac{1}{2}
Q(\xi)\delta_T(\eta) \right).
\end{split}
\end{equation}

To prove theorem \ref{theorem:F_GSE1=F_GOE}, we need only
\eqref{eq:simplified_F_GUE1} and \eqref{eq:F_GUE1=F^2_GOE_by_Forrester} with $\xi$ and $\eta$ swapped, and
\begin{equation}
\det \left( I + \frac{1}{2} Q(\xi)\delta_T(\eta) \right) = 1,
\end{equation}
which by (\ref{eq:another_form_of_det_identity}) is equivalent to
\begin{equation} \label{eq:Q(T)=0}
Q(T) = 0.
\end{equation}

If we take $f(\xi) = Q(\xi)+1$, then (\ref{eq:Q(T)=0}) is
\begin{multline}
\left(I - (K_{\Airy}(\xi, \eta)\chi(\eta) +
s^{(1)}(\xi)\Ai(\eta))\chi(\eta) \right) (f(\xi)-1) = \\
\int^{\infty}_T K_{\Airy}(\xi, t)dt - B(T)B(\xi) - B(\xi) + B(T),
\end{multline}
which is equivalent to
\begin{equation} \label{eq:integral equation}
\left(I - (K_{\Airy}(\xi, \eta)\chi(\eta) +
s^{(1)}(\xi)\Ai(\eta))\chi(\eta) \right) f(\xi) = s^{(1)}(\xi).
\end{equation}

The integral equation (\ref{eq:integral equation}) is solvable, and
the solution is
\begin{equation}
f(\xi) = \frac{(I+R) s^{(1)}(\xi)} {1 - \langle (I+R) s^{(1)}(\xi), \Ai(\xi)
\rangle_T}.
\end{equation}
Therefore to prove the theorem (\ref{theorem:F_GSE1=F_GOE}) we need
only to prove $f(T)=1$, which is equivalent to
\begin{equation}
(I+R) s^{(1)}(T) = 1 - \langle (I+R) s^{(1)}(\xi), \Ai(\xi) \rangle_T.
\end{equation}
This is a nontrivial result, but it can be derived by results in
\cite{Tracy-Widom96}, with \footnote{In section VII of
\cite{Tracy-Widom96} Tracy and Widom define function $\bar{q}$ and
$\bar{u}$ for both GOE and GSE. Our $(I+R) s^{(1)}(T)$ is equal to
$\sqrt{2}$ times their $\bar{q}$ in GOE and our $\langle (I+R)
s^{(1)}(\xi), \Ai(\xi) \rangle_T$ is equal to $2$ times their $\bar{u}$
in GOE.}
\begin{align}
(I+R) s^{(1)}(T) = & e^{-\int^{\infty}_{T} q(s)ds}, \label{eq:first_Tracy_Widom_result} \\
\langle (I+R) s^{(1)}(\xi), \Ai(\xi) \rangle_T = & 1-
e^{-\int^{\infty}_{T} q(s)ds}, \label{eq:second_Tracy_Widom_result}
\end{align}
where $q$ is the Painlev\'{e} II function described in \eqref{eq:Painleve_equation_II} and \eqref{eq:boundary_condition_of_Painleve}

We can give a proof of \eqref{eq:first_Tracy_Widom_result} and \eqref{eq:second_Tracy_Widom_result}, based on the method and results in \cite{Tracy-Widom94}. First, assume $T$ is fixed, then $(I+R) s^{(1)}$ is a function, and we have
\begin{equation}
\frac{d}{d\xi} (I+R) s^{(1)}(\xi) = (I+R) \frac{ds^{(1)}(\xi)}{d\xi} + \left[ \frac{d}{d\xi}, (1+R) \right] s^{(1)}(\xi).
\end{equation}
Since $\frac{d}{d\xi}s^{(1)}(\xi) = \Ai(\xi)$ and we have (2.13) in \cite{Tracy-Widom94}, which is
\begin{equation}
\left[ \frac{d}{d\xi}, (1+R) \right] = -(2+R)\Ai(\xi) \cdot (1-K^t)^{-1}(\Ai(\eta)\chi(\eta)) + R(\eta, T) \cdot \rho(T, \eta),
\end{equation}
where $\rho(x, y) = \delta(x-y) + R(x,y)$ is the distribution kernel of $1+R$, and $K^t$ is the transpose (as an operator) of $K_{\Airy}(\xi, \eta)\chi(\eta)$, we have
\begin{multline}
\frac{d}{d\xi} (I+R) s^{(1)}(\xi) = (1+R)\Ai(\xi) \\
- (1+R)\Ai(\xi) \cdot \langle (I+R) s^{(1)}(\xi), \Ai(\xi) \rangle_T + R(\xi, T) \cdot (1+R)s^{(1)}(T).
\end{multline}
If we regard $T$ as a parameter, then we have
\begin{equation} \label{eq:derivative_wrt_parameter_T}
\frac{d}{dT} (I+R) s^{(1)}(\xi; T) = -R(\xi, T) \cdot (1+R)s^{(1)}(T),
\end{equation}
because (2.16) in \cite{Tracy-Widom94} gives
\begin{equation}
\frac{1}{dT}(1+R) = R(\xi, T) \cdot \rho(T, \eta).
\end{equation}
Therefore, if we set $\xi = T$ and take the derivative with respect to the parameter $T$, we have
\begin{equation}
\begin{split}
\frac{d}{dT} ((1+R)s^{(1)}(T)) = & \left. \left( \frac{d}{d\xi} + \frac{d}{dT} \right) ((1+R)s^{(1)}(T)) \right\rvert_{\xi = T} \\
= & (1+R)\Ai(T) \cdot (1 - \langle (I+R) s^{(1)}(\xi), \Ai(\xi) \rangle_T).
\end{split}
\end{equation}
On the other hand, by \eqref{eq:derivative_wrt_parameter_T} we have
\begin{equation}
\begin{split}
\frac{d}{dT} \langle (I+R) s^{(1)}(\xi), \Ai(\xi) \rangle_T = & -(1+R)s^{(1)}(T) \cdot \Ai(T) +  \langle \frac{d}{dT} (I+R) s^{(1)}(\xi), \Ai(\xi) \rangle_T \\
= & -(1+R)s^{(1)}(T) \cdot \left( \Ai(T) + \int^{\infty}_T R(\xi, T)\Ai(\xi) d\xi \right) \\
= & -(1+R)s^{(1)}(T) \cdot (1+R)\Ai(T).
\end{split}
\end{equation}
(1.11) and (1.12) in \cite{Tracy-Widom94} give the result
\begin{equation}
(1+R)\Ai(T) = q(T),
\end{equation}
and now we if we denote $(I+R) s^{(1)}(T) = s_T$ and $\langle (I+R) s^{(1)}(\xi), \Ai(\xi) \rangle_T = w_T$, we have
\begin{equation}
\left\{
\begin{split}
\frac{d}{dT}s_T = & q(1-w_T) \\
\frac{d}{dT}(1-w_T) = & qs_T.
\end{split}
\right.
\end{equation}
Now we can get \eqref{eq:first_Tracy_Widom_result} and \eqref{eq:second_Tracy_Widom_result} by boundary conditions.

\section{Conjectures of phase transition in quaternionic and real spiked models}

In the complex spiked model, we have more complicated phase transition phenomenon for the limiting distribution of the largest sample eigenvalue, if the rank is greater than $1$. For example, if there are two spiked population eigenvalues $1+\alpha_1$ and $1+\alpha_2$, we have the phase diagram as $\alpha_1$ and $\alpha_2$ vary from $-1$ to $\infty$:
\begin{figure}[h]
\centering
\includegraphics{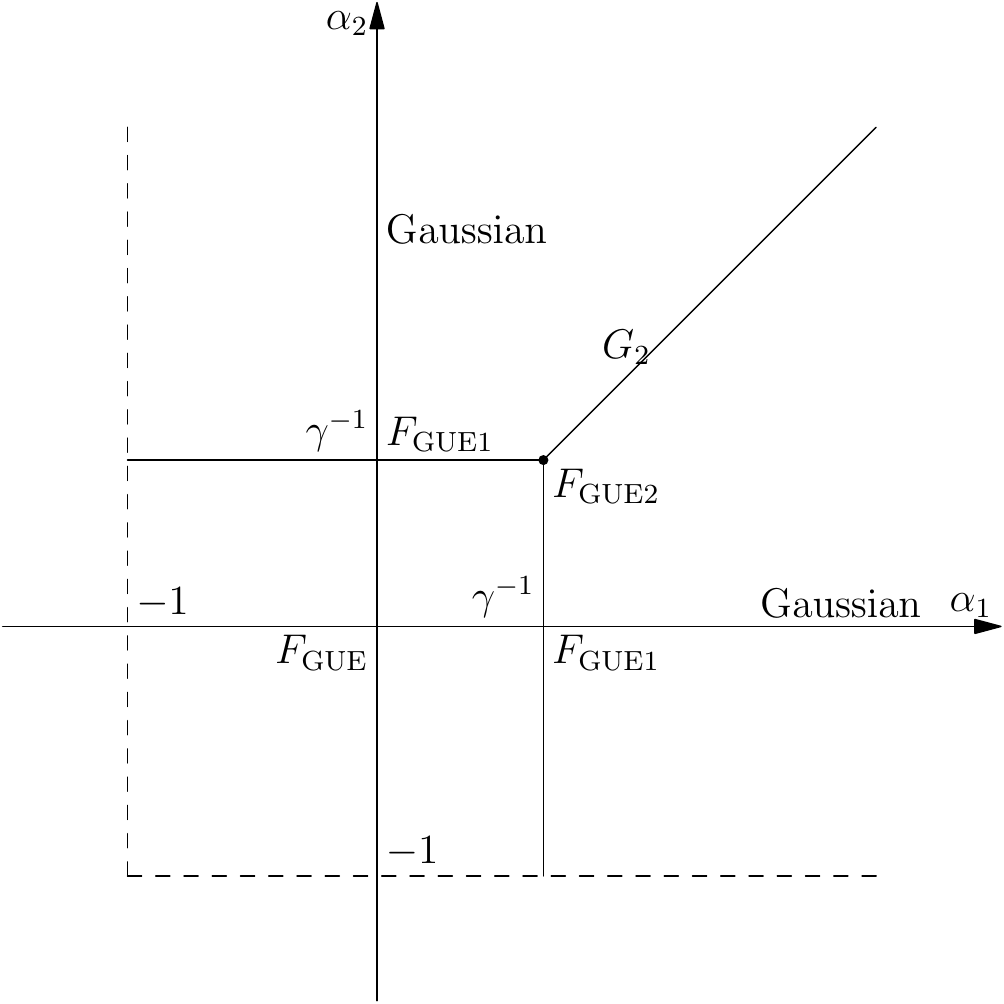}
\caption{Phase diagram of rank 2 complex spiked model}
\end{figure}

Analogously, we conjecture the phase diagram for the rank $2$ quaternionic spiked model for the limiting distribution of the largest sample eigenvalue, with spiked population eigenvalues $1+\alpha_1$ and $1+\alpha_2$.
\begin{figure}[h]
\centering
\includegraphics{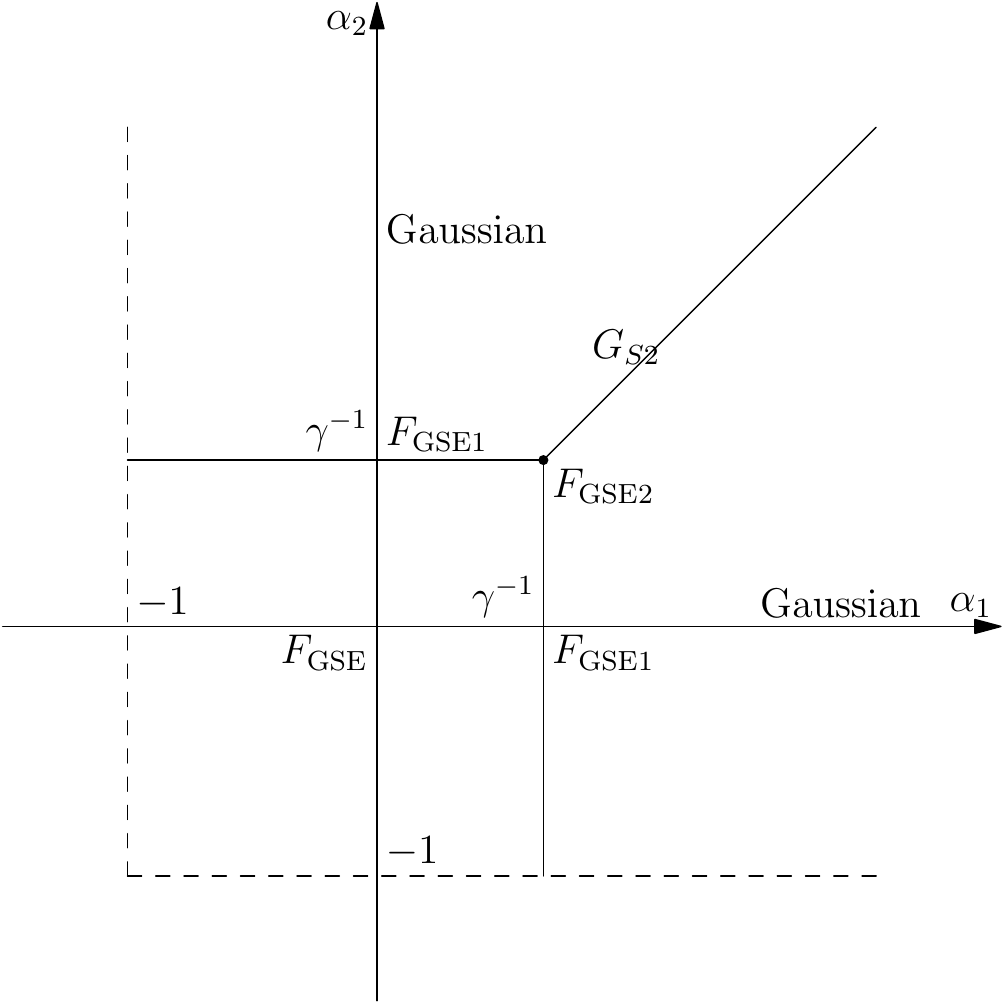}
\caption{Phase diagram of rank 2 quaternionic spiked model (conjecture)}
\end{figure}
The $G_{S1}$ is the distribution function of the largest eigenvalue of a $2 \times 2$ random quaternionic Hermitian matrix $\left( \begin{smallmatrix} a & c+id+je+kf \\ c-id-je-kf & b \end{smallmatrix} \right)$, where $a, \dots, f$ are independent normal random variables with mean $0$, the variance of $a$ and $b$ is $1$, and the variance of $c, \dots, f$ is $1/2$. Our conjecture for $G_{S2}$ is based on the pattern of $G_t$ for the rank $t$ complex spiked model: Actually $G_t$ is the distribution function of the largest eigenvalue of a $t \times t$ random Hermitian matrix $(a_{ij})_{1 \leq i,j \leq t}$, where $a_{ii}$, $\Re(a_{ij})$, $\Im(a_{ij})$ $(i<j)$ are independent normal random variables with mean $0$, the variance of diagonal entries is $1$, and the variance of real and imaginary parts of off-diagonal entries is $1/2$. Therefore we guess such distribution function for rank $2$ quaternionic spiked model should be the distribution function of of the largest eigenvalue of a $t \times t$ random Hermitian matrix. It is obvious that in the rank $1$ case, the Gaussian distribution satisfies the conjecture trivially, and the next distribution function is $G_{S2}$. But what is $F_{\GSE 2}$? Is $F^2_{\GSE 2}$ a Fredholm determinant? Our only clue is that the $F_{\GSE 2}$ should be similar to the $F^{\semi}_0$ defined in \cite{Prahofer-Spohn02a}.

For the real spiked model, even the rank $1$ case is speculative. We conjecture that if the only spiked population eigenvalue is $1+a$, then the limiting distribution of the largest sample eigenvalue has the pattern $F_{\GOE}$---$F_{\GSE}$---Gaussian. The Gaussian part has been proved in \cite{Paul08}, and other results are missing.

%% file: asymptotic_analysis.tex
\chapter{Asymptotic analysis}

\label{asymptotic_analysis}

In sections \ref{asymotptics_less}--\ref{asymptotics_more}, the convention of notations is the same as that in chapter \ref{complex_spiked_model}, e.g., $\psi_{r'}$ is defined by \eqref{eq:definition_of_phi_rprime}. In section \ref{asymptotics_Laguerre}, the convention of notations is the same as that in chapter \ref{quaternionic_spiked_model_rank_1}, e.g., $\psi_{r'}$ is defined by \eqref{eq:definition_of_psi}.

\input{asymptotics_less}

\input{asymptotics_equal}

\input{asymptotics_more}

\input{asymptotics_Laguerre}

%% file: asymptotics_less.tex
\section{Asymptotics of $\psi(p+q\xi)$, $\psi_{r'}(p+q\xi)$, $\varpsi(p+q\eta)$ and $\varpsi_{r'}(p+q\eta)$ when $a_{s'} < \gamma^{-1}$}

\label{asymotptics_less}

In this section, we assume $x = p+q\xi$ and $y = p+q\eta$, where $p = (1+\gamma^{-1})^2$ and $q = \frac{(1+\gamma)^{4/3}}{\gamma M^{2/3}}$.

For the asymptotic analysis, we define $\barSigmainf = \barSigmainf_1 \cup \barSigmainf_2 \cup \barSigmainf_3$, where
\begin{align}
\barSigmainf_1 = & \{ t e^{\frac{2\pi i}{3}} \mid t \geq 1 \}, \\
\barSigmainf_2 = & \{ e^{-t\pi i} \mid -\frac{4}{3} \leq t \leq -\frac{2}{3} \}, \\
\barSigmainf_3 = & \{ -t e^{\frac{4\pi i}{3}} \mid t \leq -1 \},
\end{align}
and ($c < 0$)
\begin{align}
\barSigmainf_{>c} = & \{ w \in \barSigmainf \mid \Re(w) > c \}, \\
\barSigmainf_{\leq c} = & \barSigmainf \setminus \barSigmainf_{>c}.
\end{align}
\begin{figure}[h]
\centering
\includegraphics{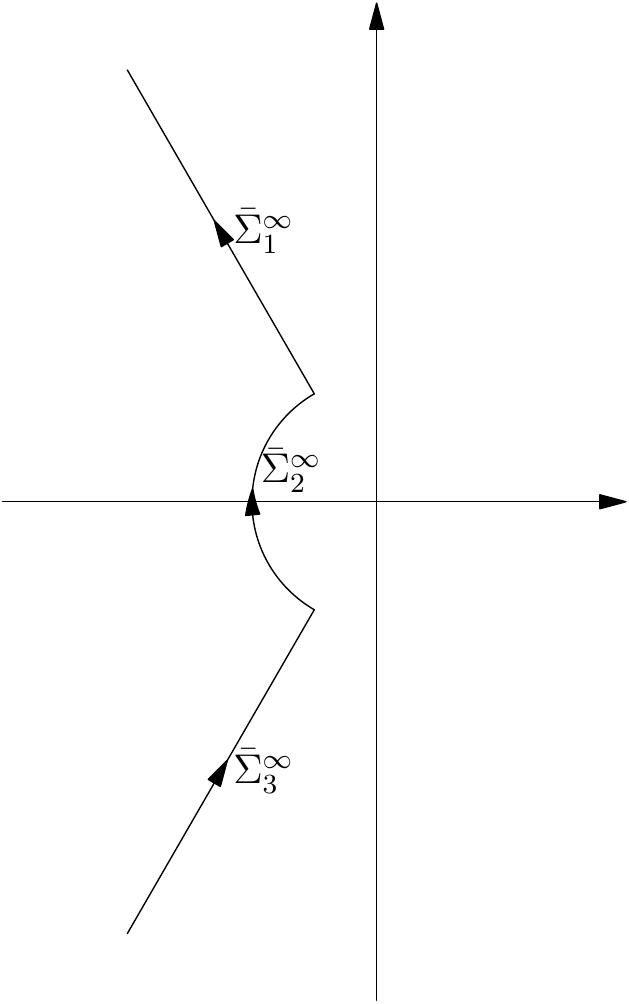}
\caption{$\barSigmainf$}
\end{figure}
We have
\begin{equation} \label{eq:one_integral_formula_of_Airy_function}
\frac{1}{2\pi i} \int_{\barSigmainf} e^{\xi u - \frac{u^3}{3}} du = \Ai(\xi),
\end{equation}
since by the substitution $u = iv$, we get ($\int^{\infty e^{\pi i/6}}_{\infty e^{5\pi i/6}}$ is defined below \ref{eq:definition_of_Airy_function_by_contour})
\begin{equation} \label{eq:simple_change_of_variable_for_Airy}
\frac{1}{2\pi i} \int_{\barSigmainf} e^{\xi u - \frac{u^3}{3}} du = \frac{1}{2\pi} \int^{\infty e^{\pi i/6}}_{\infty e^{5\pi i/6}} e^{i\xi v + \frac{iv^3}{3}} dv,
\end{equation}
which agrees with the integral definition of the Airy function, with the integration on the right hand side from $\infty e^{5\pi i/6}$ to $\infty e^{\pi i/6}$. By direct calculation, we also have the result that for any $T$, if $-c$ is large enough
\begin{equation}
\left\lvert \frac{1}{2\pi i} \int_{\barSigmainf_{\leq c}} e^{Tu - \frac{u^3}{3}} du \right\rvert < \frac{1}{-c}.
\end{equation}

Similarly, we define $\barGammainf = \barGammainf_1 \cup \barGammainf_2 \cup \barGammainf_3$, where
\begin{align}
\barGammainf_1 = & \{ -t e^{\frac{\pi i}{3}} \mid t \leq -1 \}, \\
\barGammainf_2 = & \{ e^{-t\pi i} \mid -\frac{1}{3} \leq t \leq \frac{1}{3} \}, \\
\barGammainf_3 = & \{ t e^{\frac{5\pi i}{3}} \mid t \geq 1 \}, 
\end{align}
and
\begin{align}
\barGammainf_{<c} = & \{ w \in \barGammainf \mid \Re(w) < c \}, \\
\barGammainf_{\geq c} = & \barGammainf \setminus \barSigmainf_{<c}. \label{eq:definition_of_xbarGammainf_geq_c}
\end{align}
\begin{figure}[h]
\centering
\includegraphics{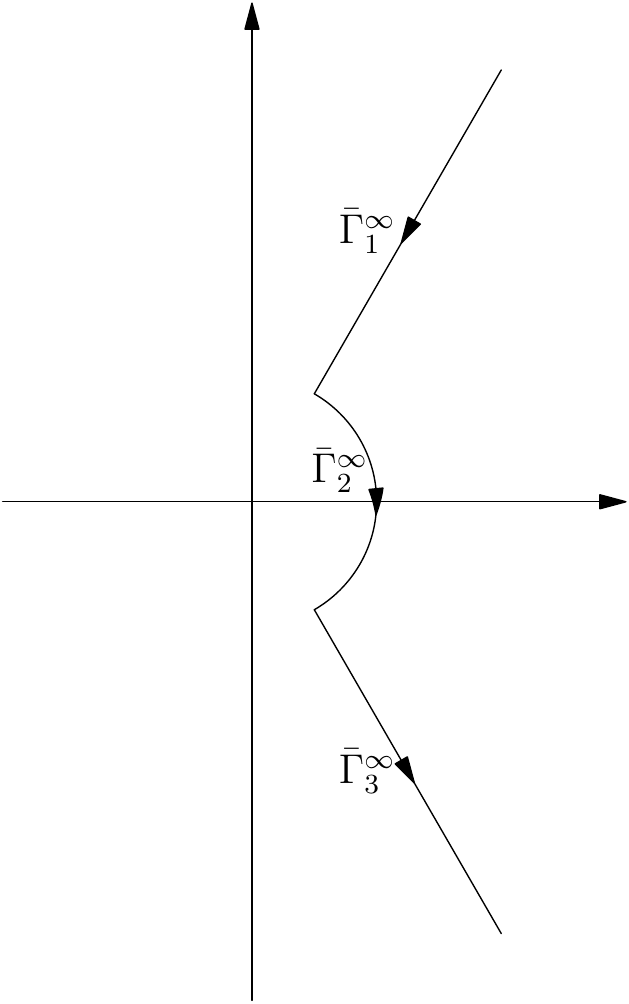}
\caption{$\barGammainf$}
\end{figure}
We have (similar to \eqref{eq:one_integral_formula_of_Airy_function})
\begin{equation}
\frac{1}{2\pi i} \int_{\barGammainf} e^{-\eta u + \frac{u^3}{3}} du = -\Ai(\xi),
\end{equation}
and for any $T$, if $c$ is large enough
\begin{equation}
\left\lvert \frac{1}{2\pi i} \int_{\barGammainf_{\geq c}} e^{-Tu + \frac{u^3}{3}} du \right\rvert < \frac{1}{c}.
\end{equation}

\subsection{Asymptotics of $\psi(p+q\xi)$ and $\psi_{r'}(p+q\xi)$}

We only analyze $\psi_{r'}(p+q\xi)$. The analysis of $\psi(p+q\xi)$ is similar and simpler, and we only give the result.

First we have
\begin{equation}
e^{Mxz} \frac{(z-1)^N}{z^M} = e^{-Mf(z) + \frac{(1+\gamma)^{4/3}}{\gamma}M^{1/3}\xi z}, \\
\end{equation}
where 
\begin{equation} \label{eq:definition_of_fz}
f(z) = -(1+\gamma^{-1})^2 z + \log z - \gamma^{-2}\log(z-1),
\end{equation}
and here and later, we do not need to concern ourselves about the ambiguity of the value of logarithmic functions. Now we can write \eqref{eq:definition_of_phi_rprime} as
\begin{multline} \label{eq:integral_representation_of_psi_by_f}
\psi_{r'}(x) = \frac{1}{2\pi i} \\
\oint_{\Sigma} e^{-Mf(z) + \frac{(1+\gamma)^{4/3}}{\gamma}M^{1/3}\xi z} \frac{z^{r-r'}}{(z-1)^r} \left( \prod^{s'-1}_{j=1} \left( z - \frac{1}{1+a_j} \right)^{r_j} \right) \left( z - \frac{1}{1+a_{s'}} \right)^{t'-1} dz.
\end{multline}

For $f(z)$ we have
\begin{itemize}
\item $\displaystyle f'(z) = \frac{((1+\gamma^{-1})z-1)^2}{z(z-1)}$, with the zero point $z = \frac{\gamma}{\gamma+1}$;

\item  $\displaystyle f''\left( \frac{\gamma}{\gamma+1} \right) = 0$;

\item  $\displaystyle f'''\left( \frac{\gamma}{\gamma+1} \right) = \frac{2(\gamma+1)^4}{\gamma^3} > 0$.
\end{itemize}
Hence locally around $z = \frac{\gamma}{\gamma+1}$,
\begin{equation} \label{eq:change_z_into_w}
f\left( \frac{\gamma}{\gamma+1} + w \right) = -\frac{\gamma+1}{\gamma} + \log\gamma - (1-\gamma^2)\log(\gamma+1) + \gamma^{-2}\pi i + \frac{(\gamma+1)^4}{3\gamma^3}w^3 + R_1(w),
\end{equation}
where
\begin{equation}
R_1(w) = \varO(w^4), \quad \text{as} \quad w \rightarrow 0.
\end{equation}
After the substitution $z = w + \frac{\gamma}{\gamma+1}$, we get by \eqref{eq:integral_representation_of_psi_by_f}
\begin{equation} \label{eq:substitution_of_z_by_w_psi_less}
\begin{split}
& \psi_{r'}(p+q\xi) \\
= & \frac{1}{2\pi i} \oint_{\barSigmaM} e^{M \left( \frac{\gamma+1}{\gamma} - \log\gamma + (1-\gamma^{-2})\log(\gamma+1) - \gamma^{-2}\pi i - \frac{(\gamma+1)^4}{3\gamma^3}w^3 - R_1(w) \right) + \frac{(1+\gamma)^{4/3}}{\gamma} M^{1/3}\xi \left( w+ \frac{\gamma}{\gamma+1} \right)} \\
& \frac{\left( w + \frac{\gamma}{\gamma+1} \right)^{r-r'}}{\left( w - \frac{1}{\gamma+1} \right)^r} \left( \prod^{s'-1}_{j=1} \left( w + \frac{a_j - \gamma^{-1}}{(1+\gamma^{-1})(1+a_j)} \right)^{r_j} \right) \left( w + \frac{a_{s'} - \gamma^{-1}}{(1+\gamma^{-1})(1+a_{s'})} \right)^{t'-1} dw \\
= & \frac{(-1)^N}{2\pi i} \frac{(\gamma+1)^{M-N}}{\gamma^M} e^{\frac{\gamma}{\gamma+1}Mx} \oint_{\barSigmaM} e^{\frac{(1+\gamma)^{4/3}}{\gamma} M^{1/3}\xi w - \frac{(\gamma+1)^4}{3\gamma^3}Mw^3 - MR_1(w)} \\
& \frac{\left( w + \frac{\gamma}{\gamma+1} \right)^{r-r'}}{\left( w - \frac{1}{\gamma+1} \right)^r} \left( \prod^{s'-1}_{j=1} \left( w + \frac{a_j - \gamma^{-1}}{(1+\gamma^{-1})(1+a_j)} \right)^{r_j} \right) \left( w + \frac{a_{s'} - \gamma^{-1}}{(1+\gamma^{-1})(1+a_{s'})} \right)^{t'-1} dw,
\end{split}
\end{equation}
where $\barSigmaM$ is a contour around $w = - \frac{\gamma}{\gamma+1}$, composed of $\barSigmaM_1$, $\barSigmaM_2$, $\barSigmaM_3$ and $\barSigmaM_4$, which are defined as
\begin{align}
\barSigmaM_1 = & \left\{ \left. t\frac{\gamma}{\gamma+1}e^{\frac{2\pi i}{3}} \right\rvert \frac{1}{(\gamma+1)^{1/3}}M^{-1/3} \leq t \leq 4 \right\}, \label{eq:definition_of_barGammaM_1} \\
\barSigmaM_2 = & \left\{ \left. \frac{\gamma}{(\gamma+1)^{4/3}}M^{-1/3}e^{-t\pi i} \right\rvert -\frac{4}{3} \leq t \leq -\frac{2}{3} \right\}, \\
\barSigmaM_3 = & \left\{ \left. (4-t)\frac{\gamma}{\gamma+1}e^{\frac{4\pi i}{3}} \right\rvert 0 \leq t \leq 4 - \frac{1}{(\gamma+1)^{1/3}}M^{-1/3} \right\}, \\
\barSigmaM_4 = & \left\{ \left. -2\frac{\gamma}{\gamma+1} - it \right\rvert -2\sqrt{3}\frac{\gamma}{(\gamma+1)} \leq t \leq 2\sqrt{3}\frac{\gamma}{(\gamma+1)} \right\}. \label{eq:definition_of_barGammaM_4}
\end{align}
\begin{figure}[h]
\centering
\includegraphics{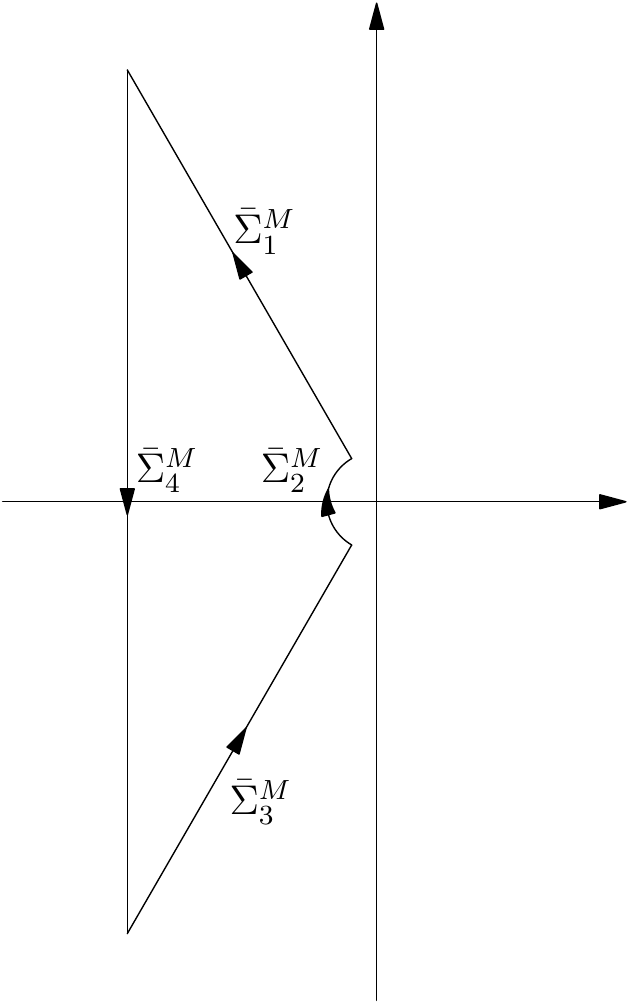}
\caption{$\barSigmaM$}
\end{figure}
For asymptotic analysis, we define
\begin{align}
\barSigmaM_{\local} = & \{ z \in \barSigmaM \mid \Re(z) > -M^{-10/39} \}, \\
\barSigmaM_{\remote} = & (\barSigmaM_1 \cup \barSigmaM_3) \setminus \barSigmaM_{\local}.
\end{align}

Now if we denote
\begin{multline}
F_M(\xi, w) = \frac{(\gamma+1)^{4/3}}{\gamma}M^{1/3} e^{\frac{(1+\gamma)^{4/3}}{\gamma} M^{1/3}\xi w - \frac{(\gamma+1)^4}{3\gamma^3}Mw^3 - MR_1(w)} \\
\frac{\left( w + \frac{\gamma}{\gamma+1} \right)^{r-r'}}{\left( w - \frac{1}{\gamma+1} \right)^r} \left( \prod^{s'-1}_{j=1} \left( w + \frac{a_j - \gamma^{-1}}{(1+\gamma^{-1})(1+a_j)} \right)^{r_j} \right) \left( w + \frac{a_{s'} - \gamma^{-1}}{(1+\gamma^{-1})(1+a_{s'})} \right)^{t'-1},
\end{multline}
we have by \eqref{eq:substitution_of_z_by_w_psi_less} 
\begin{equation} \label{eq:formula_of_variation_of_psi_less}
\frac{(\gamma+1)^{4/3}}{\gamma}M^{1/3} (-1)^N\frac{\gamma^M}{(\gamma+1)^{M-N}} e^{-\frac{\gamma}{\gamma+1}Mx} \psi_{r'}(p+q\xi) = \frac{1}{2\pi i} \oint_{\barSigmaM} F_M(\xi, w) dw,
\end{equation}
and establish several lemmas:

\begin{lemma} \label{lemma:first_lemma_for_Sigma_as_less_than_gammainv}
If $T$ is fixed and $M$ is large enough, 
\begin{equation}
\left\lvert \frac{1}{2\pi i} \int_{\barSigmaM_4} F_M(\xi, w)dw \right\rvert < \frac{1}{3} \frac{e^{-\xi/2}}{M^{1/40}},
\end{equation}
for any $\xi \geq T$.
\end{lemma}

\begin{proof}
By \eqref{eq:definition_of_fz} and \eqref{eq:change_z_into_w},
\begin{equation} \label{eq:formula_for_the_remainder_of_degree_3}
\begin{split}
& \frac{(\gamma+1)^4}{3\gamma^3}w^3 + R_1(w) \\
= & f\left( \frac{\gamma}{\gamma+1} + w \right) + \frac{\gamma}{\gamma+1} - \log\gamma + (1-\gamma^{-2})\log(\gamma+1) - \gamma^{-2}\pi i \\
= & -\left( \frac{\gamma+1}{\gamma} \right)^2 w + \log\left( \frac{\gamma+1}{\gamma}w + 1 \right) - \gamma^{-2}\log\left( (\gamma+1)w - 1 \right) - \gamma^{-2}\pi i.
\end{split}
\end{equation}
If $w \in \barSigmaM_4$, $\Re(w) = -2\frac{\gamma}{\gamma+1}$, and denote $\theta = \arg(w) \in [\frac{2\pi}{3}, \frac{4\pi}{3}]$, we have
\begin{equation}
\begin{split}
& \Re\left( \frac{(\gamma+1)^4}{3\gamma^3}w^3 + R_1(w) \right) \\
= & 2\frac{\gamma}{\gamma+1} + \log\left( \sqrt{(-1)^2 + (2\tan\theta)^2} \right) - \gamma^{-2}\log\left( \sqrt{(2\gamma+3)^2 + (2\gamma\tan\theta)^2} \right) \\
\geq & 2\frac{\gamma}{\gamma+1} + 0 - \gamma^{-2}\log\left( \sqrt{ 16\gamma^2 + 12\gamma + 9} \right),
\end{split}
\end{equation}
and we can prove that if $\gamma \geq 1$,
\begin{equation}
2\frac{\gamma}{\gamma+1} + 0 - \gamma^{-2}\log\left( \sqrt{ 16\gamma^2 + 12\gamma + 9} \right) > 2 - \log{\sqrt{37}} > 0.
\end{equation}
Therefore on $\barSigmaM_4$, if $\xi \geq T$, for $0 < \epsilon' < 2 - \log{\sqrt{37}}$ and $M$ large enough, 
\begin{equation}
\begin{split}
& \lvert F_M(\xi, w) \rvert \\
< & \frac{(\gamma+1)^{4/3}}{\gamma}M^{1/3} e^{-2(\xi-T)(\gamma+1)^{1/3}M^{1/3} + (\log{\sqrt{37}} -2 + 2T(\gamma+1)^{1/3}M^{-2/3})M} \\
& \times \left\lvert \frac{\left( w + \frac{\gamma}{\gamma+1} \right)^{r-r'}}{\left( w - \frac{1}{\gamma+1} \right)^r} \left( \prod^{s'-1}_{j=1} \left( w + \frac{a_j - \gamma^{-1}}{(1+\gamma^{-1})(1+a_j)} \right)^{r_j} \right) \left( w + \frac{a_{s'} - \gamma^{-1}}{(1+\gamma^{-1})(1+a_{s'})} \right)^{t'-1} \right\rvert \\
< & e^{-2(\xi-T)(\gamma+1)^{1/3}M^{1/3}}e^{(\log{\sqrt{37}} -2 + \epsilon')M}.
\end{split}
\end{equation}
If $M$ is large enough, 
\begin{align}
e^{(\log{\sqrt{37}} -2 + \epsilon')M} < & \frac{2\pi}{4\sqrt{3} \frac{\gamma}{\gamma+1}} \frac{1}{3} e^{-T/2}M^{-1/40}, \label{eq:first_scalar_inequality_for_lemma_one}\\
e^{-2(\xi-T)(\gamma+1)^{1/3}M^{1/3}} < & e^{T/2}e^{-\xi/2}, \label{eq:second_scalar_inequality_for_lemma_one}
\end{align}
and we get the result, since
\begin{equation} \label{eq:length_of_arc}
\left\lvert \frac{1}{2\pi i} \int_{\barSigmaM_4} F_M(\xi, w)dw \right\rvert \leq \frac{4\sqrt{3} \frac{\gamma}{\gamma+1}}{2\pi} \max_{w \in \barSigmaM_4}\lvert F_M(\xi, w) \rvert
\end{equation}
\end{proof}

\begin{lemma} \label{lemma:second_lemma_for_Sigma_as_less_than_gammainv}
If $T$ is fixed and $M$ is large enough, 
\begin{equation}
\left\lvert \frac{1}{2\pi i} \int_{\barSigmaM_{\remote}} F_M(\xi, w)dw \right\rvert < \frac{1}{3} \frac{e^{-\xi/2}}{M^{1/40}},
\end{equation}
for any $\xi \geq T$.
\end{lemma}

\begin{proof}
For $w \in \barSigmaM_{\remote}$, we denote $l = -\Re(w) = \frac{\lvert w \rvert}{2}$. Since $\arg(w) = \pm \frac{2\pi}{3}$, we get by \eqref{eq:formula_for_the_remainder_of_degree_3}
\begin{multline} \label{eq:modulus_in_variable_of_l}
\Re\left( \frac{(\gamma+1)^4}{3\gamma^3}w^3 + R_1(w) \right) = \frac{(\gamma+1)^2}{\gamma^2}l \\
+ \frac{1}{2}\log\left( 1 - 2\frac{\gamma+1}{\gamma}l + 4\left( \frac{\gamma+1}{\gamma}l \right)^2 \right) - \frac{\gamma^{-2}}{2}\log\left( 1 + 2(\gamma+1)l + 4(\gamma+1)^2l^2 \right).
\end{multline}
Then we take the derivative on both sides of \eqref{eq:modulus_in_variable_of_l}
\begin{multline}
\frac{d}{dl} \Re\left( \frac{(\gamma+1)^4}{3\gamma^3}w^3 + R_1(w) \right) = \\
8\frac{(\gamma+1)^4}{\gamma^3}l^2 \frac{1 + (\gamma-1)\frac{\gamma+1}{\gamma}l + 2\gamma\left( \frac{\gamma+1}{\gamma}l \right)^2}{\left( 1 - 2\frac{\gamma+1}{\gamma}l + 4\left( \frac{\gamma+1}{\gamma}l \right)^2 \right) \left( 1 + 2(\gamma+1)l + 4(\gamma+1)^2l^2 \right)},
\end{multline}
and are able to find a positive number $\epsilon''$, such that for $0 \leq l \leq 2\frac{\gamma}{\gamma+1}$,
\begin{equation}
8\frac{(\gamma+1)^4}{\gamma^3}l^2 \frac{1 + (\gamma-1)\frac{\gamma+1}{\gamma}l + 2\gamma\left( \frac{\gamma+1}{\gamma}l \right)^2}{\left( 1 - 2\frac{\gamma+1}{\gamma}l + 4\left( \frac{\gamma+1}{\gamma}l \right)^2 \right) \left( 1 + 2(\gamma+1)l + 4(\gamma+1)^2l^2 \right)} > 3\epsilon''l^2,
\end{equation}
and on the two right most points of $\barSigmaM_{\remote}$, $(-1+\sqrt{3})M^{-10/39}$ and $(-1-\sqrt{3})M^{-10/39}$,
\begin{equation}
\begin{split}
& \left. \Re\left( \frac{(\gamma+1)^4}{3\gamma^3}w^3 + R_1(w) \right) \right\rvert_{w = (-1 \pm \sqrt{3})M^{-10/39}} \\
= & \frac{8}{3} \frac{(\gamma+1)^4}{\gamma^3} M^{-10/13} + \varO(M^{-40/39}) \\
= & \frac{8}{3} \frac{(\gamma+1)^4}{\gamma^3} M^{-10/13} \left( 1 + \varO(M^{-10/39}) \right) \\
> & \int^{M^{10/39}}_0 3\epsilon''t^2 dt.
\end{split}
\end{equation}
Hence we know that for $w \in \barSigmaM_{\remote}$,
\begin{equation}
\Re\left( \frac{(\gamma+1)^4}{3\gamma^3}Mw^3 + MR_1(w) \right) > M \int^{M^{10/39}}_0 3\epsilon''t^2 dt = M\epsilon''l^3,
\end{equation}
and have the estimation that if $\xi \geq T$, for $0 < \epsilon''' < \epsilon''$ and $M$ large enough, ($l \geq M^{-10/39}$)
\begin{equation}
\begin{split}
& \lvert F_M(\xi, w) \rvert \\
< & \frac{(\gamma+1)^{4/3}}{\gamma}M^{1/3} e^{-(\xi-T)\frac{(\gamma+1)^{4/3}}{\gamma}M^{1/3}l - (\epsilon''l^3 + T\frac{(\gamma+1)^{4/3}}{\gamma}M^{-2/3}l)M} \\
& \times \left\lvert \frac{\left( w + \frac{\gamma}{\gamma+1} \right)^{r-r'}}{\left( w - \frac{1}{\gamma+1} \right)^r} \left( \prod^{s'-1}_{j=1} \left( w + \frac{a_j - \gamma^{-1}}{(1+\gamma^{-1})(1+a_j)} \right)^{r_j} \right) \left( w + \frac{a_{s'} - \gamma^{-1}}{(1+\gamma^{-1})(1+a_{s'})} \right)^{t'-1} \right\rvert \\
< & e^{-(\xi-T)\frac{(\gamma+1)^{4/3}}{\gamma}M^{1/13}}e^{-\epsilon'''M^{3/13}}.
\end{split}
\end{equation}
Now we get the result by similar inequalities as \eqref{eq:first_scalar_inequality_for_lemma_one}--\eqref{eq:length_of_arc}.
\end{proof}

If we define
\begin{equation} \label{eq:definition_of_Cbar_rprime}
\bar{C}_{r'} = \left( \prod^{s'-1}_{j=1} \left( \frac{a_j - \gamma^{-1}}{(1+\gamma^{-1})(1+a_j)} \right)^{r_j} \right) \left( \frac{a_{s'} - \gamma^{-1}}{(1+\gamma^{-1})(1+a_{s'})} \right)^{t'},
\end{equation}
we have
\begin{lemma} \label{lemma:fourth_lemma_for_Sigma_as_less_than_gammainv}
If $T$ is fixed and $M$ is large enough, 
\begin{equation}
\left\lvert \frac{1}{2\pi i}\int_{\barSigmaM_{\local}} F_M(\xi, w)dw  - (-1)^r \gamma^{r-r'}(\gamma+1)^{r'} \bar{C}_{r'-1} \Ai(\xi) \right\rvert < \frac{1}{3} \frac{e^{-\xi/2}}{M^{1/40}},
\end{equation}
for  any $\xi \geq T$.
\end{lemma}

\begin{proof}
On $\barSigmaM_{\local}$, $\lvert w \rvert < 2M^{-10/39}$, and $R_1(w) = \varO(M^{-40/39})$, so that
\begin{multline}
F_M(\xi, w) = \frac{\left( \frac{\gamma}{\gamma+1} \right)^{r-r'}}{\left( -\frac{1}{\gamma+1} \right)^r} \left( \prod^{s'-1}_{j=1} \left( \frac{a_j - \gamma^{-1}}{(1+\gamma^{-1})(1+a_j)} \right)^{r_j} \right) \left( \frac{a_{s'} - \gamma^{-1}}{(1+\gamma^{-1})(1+a_{s'})} \right)^{t'-1} \\
\frac{(\gamma+1)^{4/3}}{\gamma}M^{1/3} e^{\frac{(1+\gamma)^{4/3}}{\gamma} M^{1/3}\xi w - \frac{(\gamma+1)^4}{3\gamma^3}Mw^3} \left( 1 + \varO(M^{-1/39}) \right),
\end{multline}
and the $\varO(M^{-1/39})$ term is independent of $\xi$. After the substitution $u = \frac{(1+\gamma)^{4/3}}{\gamma} M^{1/3} w$, we get
\begin{multline} \label{eq:pointwise_convergence_to_barSigmainf}
\frac{1}{2\pi i} \int_{\barSigmaM_{\local}} F_M(\xi, w) dw = \\
\frac{(-1)^r \gamma^{r-r'}(\gamma+1)^{r'} \bar{C}_{r'-1}}{2\pi i} \int_{\barSigmainf_{>-\frac{(1+\gamma)^{4/3}}{\gamma} M^{1/13}}} e^{\xi u - \frac{u^3}{3}} du \left( 1 + \varO(M^{-1/39}) \right).
\end{multline}
On $\barSigmainf$, if $\xi \geq T$, $\lvert e^{(\xi-T) u} \rvert \leq e^{T/2}e^{-\xi/2}$, and we have
\begin{equation}
\begin{split}
& \left\lvert \frac{1}{2\pi i} \int_{\barSigmaM_{\local}} F_M(\xi, w)dw - (-1)^r \gamma^{r-r'}(\gamma+1)^{r'} \bar{C}_{r'-1} \Ai(\xi) \right\rvert \\
\leq & e^{T/2}e^{-\xi/2} \left\lvert \frac{(-1)^r \gamma^{r-r'}(\gamma+1)^{r'} \bar{C}_{r'-1}}{2\pi i} \int_{\barSigmainf_{\leq -\frac{(\gamma+1)^{4/3}}{\gamma}M^{1/13}}} \left\lvert e^{Tu-\frac{u^3}{3}} \right\rvert du (1 + \varO(M^{-1/39})) \right\rvert \\
& + e^{T/2}e^{-\xi/2} \left\lvert \frac{(-1)^r \gamma^{r-r'}(\gamma+1)^{r'}\bar{C}_{r'-1}}{2\pi i} \int_{\barSigmainf_{>-\frac{(\gamma+1)^{4/3}}{\gamma}M^{1/13}}} \left\lvert e^{Tu-\frac{u^3}{3}} \right\rvert du \varO(M^{-1/39}) \right\rvert,
\end{split}
\end{equation}
and we can get the result by direct calculation.
\end{proof}

By lemmas \ref{lemma:first_lemma_for_Sigma_as_less_than_gammainv}--\ref{lemma:fourth_lemma_for_Sigma_as_less_than_gammainv}, and \eqref{eq:formula_of_variation_of_psi_less}, we get the convergence result
\begin{multline} \label{eq:asymptotics_of_psi_rprime_less}
\left\lvert \frac{(\gamma+1)^{4/3}}{\gamma}M^{1/3} (-1)^N\frac{\gamma^M}{(\gamma+1)^{M-N}}e^{-\frac{\gamma}{\gamma+1}Mx} \psi_{r'}(p+q\xi) - \right. \\
\left.
\vphantom{\frac{\gamma^M}{(\gamma+1)^{M-N}}}
(-1)^r \gamma^{r-r'}(\gamma+1)^{r'} \bar{C}_{r'-1} \Ai(\xi) \right\rvert < \frac{e^{-\xi/2}}{M^{1/40}}.
\end{multline}
In the same way, we have the result for $\psi(p+q\xi)$
\begin{equation} \label{eq:asymptotics_of_psi_less}
\left\lvert \frac{(\gamma+1)^{4/3}}{\gamma}M^{1/3} (-1)^N\frac{\gamma^M}{(\gamma+1)^{M-N}}e^{-\frac{\gamma}{\gamma+1}Mx} \psi(p+q\xi) - (-\gamma)^r \Ai(\xi) \right\rvert < \frac{e^{-\xi/2}}{M^{1/40}}.
\end{equation}

\subsection{Asymptotics of $\varpsi(p+q\eta)$ and $\varpsi_{r'}(p+q\eta)$}

\label{asymptotics_of_varpsi_and+varpsi_rprime}
We only analyze $\varpsi_{r'}(p+q\eta)$, The analysis of $\psi(p+q\eta)$ is similar and simpler, and we only give the result.

We have
\begin{equation}
e^{-Myz} \frac{z^M}{(z-1)^N} = e^{Mf(z) - \frac{(1+\gamma)^{4/3}}{\gamma}M^{1/3}\eta z},
\end{equation}
where $f(z)$ is defined by \eqref{eq:definition_of_fz}. Now we can write \eqref{eq:definition_of_varphi_rprime} as
\begin{multline} \label{eq:integral_representation_of_varpsi_less}
\varpsi_{r'}(y) = \frac{1}{2\pi i} \\
\oint_{\Gamma} e^{Mf(z) - \frac{(1+\gamma)^{4/3}}{\gamma}M^{1/3}\eta z} \frac{(z-1)^r}{z^{r-r'+1}} \left( \prod^{s'-1}_{j=1} \left( z - \frac{1}{1+a_j} \right)^{-r_j} \right) \left( z - \frac{1}{1+a_{s'}} \right)^{-t'} dz.
\end{multline}
After the substitution $z = w + \frac{\gamma}{\gamma+1}$, we get
\begin{equation}
\begin{split}
& \varpsi_{r'}(p+q\eta) \\
= & \frac{1}{2\pi i} \oint_{\barGammaM} e^{M -\left( \frac{\gamma+1}{\gamma} + \log\gamma - (1-\gamma^{-2})\log(\gamma+1) + \gamma^{-2}\pi i + \frac{(\gamma+1)^4}{3\gamma^3}w^3 + R_1(w) \right) - \frac{(1+\gamma)^{4/3}}{\gamma} M^{1/3}\eta \left( w+ \frac{\gamma}{\gamma+1} \right)} \\
& \frac{\left( w - \frac{1}{\gamma+1} \right)^r}{\left( w + \frac{\gamma}{\gamma+1} \right)^{r-r'+1}} \left( \prod^{s'-1}_{j=1} \left( w + \frac{a_j - \gamma^{-1}}{(1+\gamma^{-1})(1+a_j)} \right)^{-r_j} \right) \left( w + \frac{a_{s'} - \gamma^{-1}}{(1+\gamma^{-1})(1+a_{s'})} \right)^{-t'} dw \\
= & \frac{(-1)^N}{2\pi i} \frac{\gamma^M}{(\gamma+1)^{M-N}} e^{-\frac{\gamma}{\gamma+1}Mx} \oint_{\barGammaM} e^{-\frac{(1+\gamma)^{4/3}}{\gamma} M^{1/3}\eta w + \frac{(\gamma+1)^4}{3\gamma^3}Mw^3 + MR_1(w)} \\
& \frac{\left( w - \frac{1}{\gamma+1} \right)^r}{\left( w + \frac{\gamma}{\gamma+1} \right)^{r-r'+1}} \left( \prod^{s'-1}_{j=1} \left( w + \frac{a_j - \gamma^{-1}}{(1+\gamma^{-1})(1+a_j)} \right)^{-r_j} \right) \left( w + \frac{a_{s'} - \gamma^{-1}}{(1+\gamma^{-1})(1+a_{s'})} \right)^{-t'} dw,
\end{split}
\end{equation}

where $\barGammaM$ is a contour containing $\frac{1}{\gamma+1}$ and $\frac{\gamma^{-1} - a_j}{(1+\gamma^{-1})(1+a_j)}$, ($j = 1, \dots, s$), composed of $\barGammaM_1$, $\barGammaM_2$, $\barGammaM_3$, $\barGammaM_4$, $\barGammaM_5$ and $\barGammaM_6$, which are defined as below, with the constant $C_{\rt}$ a large enough positive number, so that $\barGammaM$ contains all the poles if $M$ is large enough.
\begin{align}
\barGammaM_1 = & \left\{ \left. \left( \frac{2}{\gamma} - t \right)\frac{\gamma}{\gamma+1}e^{\frac{\pi i}{3}} \right\rvert 0 \leq t \leq \frac{2}{\gamma} - \frac{1}{(\gamma+1)^{1/3}}M^{-1/3} \right\}, \\
\barGammaM_2 = & \left\{ \left. \frac{\gamma}{(\gamma+1)^{4/3}}M^{-1/3}e^{-t\pi i} \right\rvert -\frac{\pi}{3} \leq t \leq \frac{\pi}{3} \right\}, \\
\barGammaM_3 = & \left\{ \left. t\frac{\gamma}{\gamma+1}e^{\frac{5\pi i}{3}} \right\rvert \frac{1}{(\gamma+1)^{1/3}}M^{-1/3} \leq t \leq \frac{2}{\gamma} \right\}, \\
\barGammaM_4 = & \left\{ \left. -t + \frac{\sqrt{3}}{\gamma+1}i \right\rvert -C_{\rt} \leq t \leq -\frac{1}{\gamma+1} \right\}, \\
\barGammaM_5 = & \left\{ \left. t - \frac{\sqrt{3}}{\gamma+1}i \right\rvert \frac{1}{\gamma+1} \leq t \leq C_{\rt} \right\}, \\
\barGammaM_4 = & \left\{ C_{\rt} + it \left\rvert -\frac{\sqrt{3}}{(\gamma+1)} \leq t \leq \frac{\sqrt{3}}{(\gamma+1)} \right. \right\}.
\end{align}
\begin{figure}[h]
\centering
\includegraphics{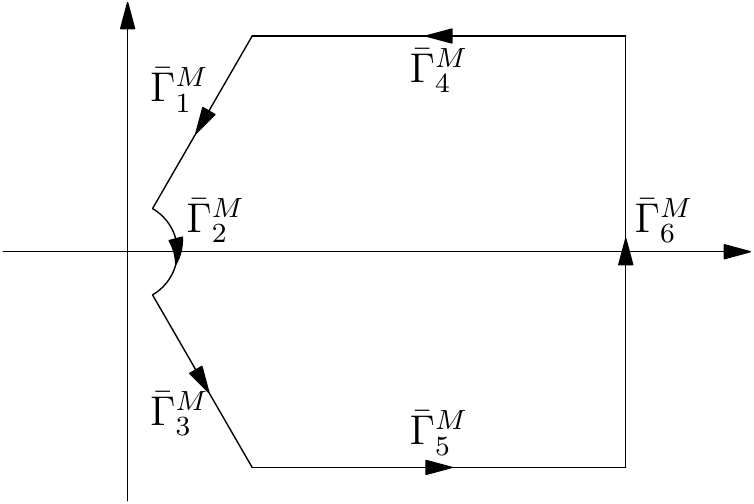}
\caption{$\barGammaM$}
\end{figure}
For asymptotic analysis, we define
\begin{align}
\barGammaM_{\local} = & \{ z \in \barGammaM \mid \Re(z) < M^{-10/39} \}, \label{eq:definition_of_barGammaM_local} \\
\barGammaM_{\remote} = & (\barGammaM_1 \cup \barGammaM_3) \setminus \barGammaM_{\local}.
\end{align}

If we denote
\begin{multline}
G_M(\eta, w) = \frac{(\gamma+1)^{4/3}}{\gamma}M^{1/3} e^{-\frac{(1+\gamma)^{4/3}}{\gamma} M^{1/3}\eta w + \frac{(\gamma+1)^4}{3\gamma^3}Mw^3 + MR_1(w)} \\
\frac{\left( w - \frac{1}{\gamma+1} \right)^r}{\left( w + \frac{\gamma}{\gamma+1} \right)^{r-r'+1}} \left( \prod^{s'-1}_{j=1} \left( w + \frac{a_j - \gamma^{-1}}{(1+\gamma^{-1})(1+a_j)} \right)^{-r_j} \right) \left( w + \frac{a_{s'} - \gamma^{-1}}{(1+\gamma^{-1})(1+a_{s'})} \right)^{-t'},
\end{multline}
we have
\begin{equation} \label{eq:variation_of_varpsi_less}
(-1)^N \frac{(\gamma+1)^{M-N}}{\gamma^M} e^{\frac{\gamma}{\gamma+1}My} \varpsi_{r'}(p+q\eta) = \frac{1}{2\pi i} \oint_{\barGammaM} G_M(\eta, w) dw,
\end{equation}
and establish several lemmas:
\begin{lemma} \label{lemma:first_for_varpsi_less}
If $T$ is fixed and $M$ is large enough, 
\begin{equation}
\left\lvert \frac{1}{2\pi i} \int_{\barGammaM_4 \cup \barGammaM_5 \cup \barGammaM_6} G_M(\eta, w)dw \right\rvert < \frac{1}{3} \frac{e^{-\eta/2}}{M^{1/40}},
\end{equation}
for any $\eta \geq T$.
\end{lemma}

\begin{lemma} \label{lemma:second_for_varpsi_less}
If $T$ is fixed and $M$ is large enough, 
\begin{equation}
\left\lvert \frac{1}{2\pi i} \int_{\barGammaM_{\remote}} G_M(\eta, w)dw \right\rvert < \frac{1}{3} \frac{e^{-\eta/2}}{M^{1/40}},
\end{equation}
for any $\eta \geq T$.
\end{lemma}

\begin{lemma} \label{lemma:last_for_varpsi_less}
If $T$ is fixed and $M$ is large enough, 
\begin{equation}
\left\lvert \frac{1}{2\pi i}\int_{\barGammaM_{\local}} F_M(\eta, w)dw  - \frac{(-1)^{r-1} (1+\gamma^{-1})}{\gamma^{r-r'}(\gamma+1)^{r'} \bar{C}_{r'}} \Ai(\eta) \right\rvert < \frac{1}{3} \frac{e^{-\eta/2}}{M^{1/40}},
\end{equation}
for any $\eta \geq T$.
\end{lemma}

Since the proofs of lemma \ref{lemma:first_for_varpsi_less}--\ref{lemma:last_for_varpsi_less} are similar to those of lemmas \ref{lemma:first_lemma_for_Sigma_as_less_than_gammainv}--\ref{lemma:fourth_lemma_for_Sigma_as_less_than_gammainv}, we only give an outline of the proof of lemma \ref{lemma:last_for_varpsi_less}.

\begin{proof}[Sketch of proof of lemma \ref{lemma:last_for_varpsi_less}]
On $\barGammaM_{\local}$, $\lvert w \rvert < 2M^{-10/39}$ and $R_1(w) = \varO(M^{-40/39})$, so that
\begin{multline}
G_M(\eta, w) = \frac{(-1)^r (1+\gamma^{-1})}{\gamma^{r-r'}(\gamma+1)^{r'} \bar{C}_{r'}} \\
\frac{(\gamma+1)^{4/3}}{\gamma}M^{1/3} e^{-\frac{(1+\gamma)^{4/3}}{\gamma} M^{1/3}\eta w + \frac{(\gamma+1)^4}{3\gamma^3}Mw^3} \left( 1 + \varO(M^{-1/39}) \right), 
\end{multline}
after the substitution $u = \frac{(1+\gamma)^{4/3}}{\gamma} M^{1/3} w$, we get
\begin{equation} 
\int_{\barGammaM_{\local}} G_M(\eta, w) dw = \frac{(-1)^r (1+\gamma^{-1})}{\gamma^{r-r'}(\gamma+1)^{r'} \bar{C}_{r'}} \int_{\barGammainf_{<\frac{(1+\gamma)^{4/3}}{\gamma} M^{1/3}}} e^{-\eta u + \frac{u^3}{3}} du \left( 1 + \varO(M^{-1/39}) \right).
\end{equation}
Also we have that on $\barGammainf$, if $\eta \geq T$, $\lvert e^{-(\eta-T) u} \leq e^{T/2}e^{-\eta/2} \rvert$. We can prove lemma \ref{lemma:last_for_varpsi_less} in the same way as proving lemma \ref{lemma:fourth_lemma_for_Sigma_as_less_than_gammainv}.
\end{proof}

By lemmas \ref{lemma:first_for_varpsi_less}--\ref{lemma:last_for_varpsi_less}, and \eqref{eq:variation_of_varpsi_less}, we get the convergence result
\begin{multline} \label{eq:asymptotics_of_varpsi_rprime_less}
\left\lvert \frac{(\gamma+1)^{4/3}}{\gamma}M^{1/3} (-1)^N\frac{(\gamma+1)^{M-N}}{\gamma^M} e^{\frac{\gamma}{\gamma+1}My} \varpsi_{r'}(p+q\eta) - \frac{(-1)^{r-1} (1+\gamma^{-1})}{\gamma^{r-r'}(\gamma+1)^{r'} \bar{C}_{r'}} \Ai(\xi) \right\rvert \\
< \frac{e^{-\eta/2}}{M^{1/40}}.
\end{multline}
And in the same way, we have the result for $\varpsi(p+q\xi)$
\begin{equation} \label{eq:asymptotics_of_varpsi_less}
\left\lvert \frac{(\gamma+1)^{4/3}}{\gamma}M^{1/3} (-1)^N\frac{(\gamma+1)^{M-N}}{\gamma^M} e^{\frac{\gamma}{\gamma+1}My} \varpsi(p+q\eta) - (-1)^{r-1}\gamma^{-r} \Ai(\xi) \right\rvert < \frac{e^{-\eta/2}}{M^{1/40}}.
\end{equation}

%% file: asymptotics_equal.tex
\section{Asymptotics of $\psi_{r'}(p+q\xi)$ and $\varpsi_{r'}(p+q\eta)$ when $a_{s'} = \gamma^{-1}$}

In this section, we still assume $x=p+q\xi$, $y=p+q\eta$, $p = (1+\gamma^{-1})^2$ and $q = \frac{(1+\gamma)^{4/3}}{\gamma M^{2/3}}$.

With $\barSigmainf$ defined in last section, we have
\begin{equation}
\frac{1}{2\pi i} \int_{\barSigmainf} e^{\xi u - \frac{u^3}{3}} u^{t'-1} du = (-1)^{t'-1}t^{(t')}(\xi),
\end{equation}
which can be proved by a simple change of variable similar to \eqref{eq:simple_change_of_variable_for_Airy}. For any $T$, if $-c$ is large enough
\begin{equation}
\left\lvert \frac{1}{2\pi i} \int_{\barSigmainf_{\leq c}} e^{Tu - \frac{u^3}{3}} u^{t'-1} du \right\rvert < \frac{1}{-c}.
\end{equation}

we define $\BbarGammainf = \BbarGammainf_1 \cup \BbarGammainf_2 \cup \BbarGammainf_3$, where
\begin{align}
\BbarGammainf_1 = & \{ -t e^{\frac{\pi i}{3}} \mid t \leq -\frac{1}{6} \}, \\
\BbarGammainf_2 = & \{ \frac{1}{6}e^{t\pi i} \mid \frac{1}{3} \leq t \leq \frac{5}{3} \}, \\
\BbarGammainf_3 = & \{ t e^{\frac{5\pi i}{3}} \mid t \geq \frac{1}{6} \}, 
\end{align}
\begin{figure}[h]
\centering
\includegraphics{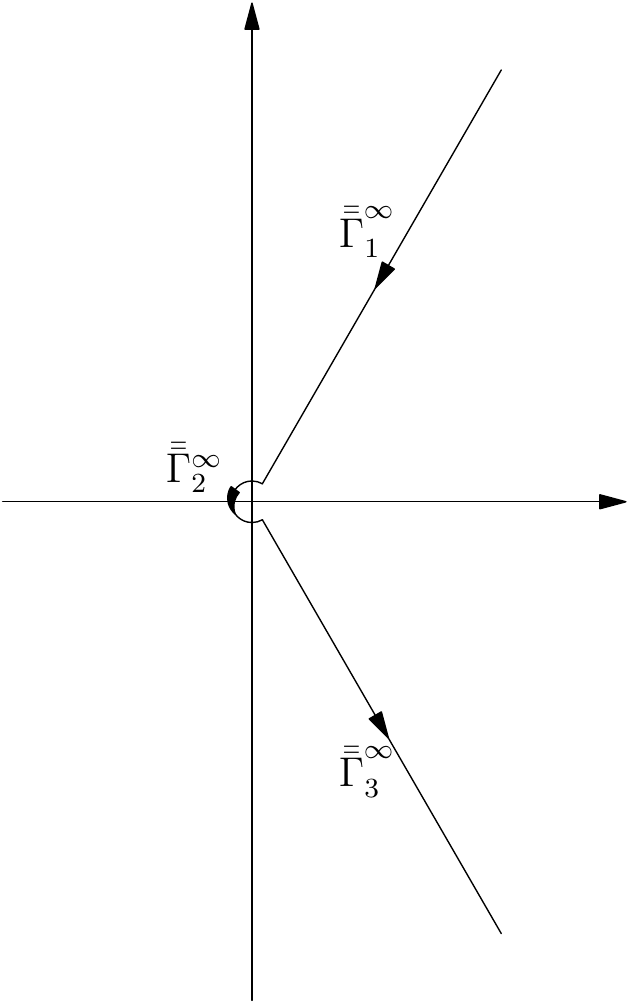}
\caption{$\BbarGammainf$}
\end{figure}
and
\begin{align}
\BbarGammainf_{<c} = & \{ w \in \BbarGammainf \mid \Re(w) < c \}, \\
\BbarGammainf_{\geq c} = & \BbarGammainf \setminus \barSigmainf_{<c}. \label{eq:definition_of_BbarGammainf_geq_c}
\end{align}
We have
\begin{equation}
\frac{1}{2\pi i} \int_{\BbarGammainf} e^{-\eta u + \frac{u^3}{3}} \frac{du}{u^{t'}} = (-1)^{t'-1}s^{(t')}(\xi),
\end{equation}
and for any $T$, if $c$ is large enough
\begin{equation}
\left\lvert \frac{1}{2\pi i} \int_{\BbarGammainf_{\geq c}} e^{-Tu + \frac{u^3}{3}} \frac{du}{u^{t'}} \right\rvert < \frac{1}{c}.
\end{equation}

\subsection{Asymptotics of $\psi_{r'}(p+q\xi)$}

Similar to \eqref{eq:integral_representation_of_psi_by_f}, we have ($f(z)$ is defined by \eqref{eq:definition_of_fz})
\begin{multline}
\psi_{r'}(x) = \frac{1}{2\pi i} \\
\oint_{\Sigma} e^{-Mf(z) + \frac{(1+\gamma)^{4/3}}{\gamma}M^{1/3}\xi z} \frac{z^{r-r'}}{(z-1)^r} \left( \prod^{s'-1}_{j=1} \left( z - \frac{1}{1+a_j} \right)^{r_j} \right) \left( z - \frac{\gamma}{\gamma+1} \right)^{t'-1} dz,
\end{multline}
and after the substitution $z = w + \frac{\gamma}{\gamma+1}$, we get
\begin{equation}
\begin{split}
& \psi_{r'}(p+q\xi) \\
= & \frac{1}{2\pi i} \oint_{\barSigmaM} e^{M \left( \frac{\gamma+1}{\gamma} - \log\gamma + (1-\gamma^{-2})\log(\gamma+1) - \gamma^{-2}\pi i - \frac{(\gamma+1)^4}{3\gamma^3}w^3 - R_1(w) \right) + \frac{(1+\gamma)^{4/3}}{\gamma} M^{1/3}\xi \left( w+ \frac{\gamma}{\gamma+1} \right)} \\
& \frac{\left( w + \frac{\gamma}{\gamma+1} \right)^{r-r'}}{\left( w - \frac{1}{\gamma+1} \right)^r} \left( \prod^{s'-1}_{j=1} \left( w + \frac{a_j - \gamma^{-1}}{(1+\gamma^{-1})(1+a_j)} \right)^{r_j} \right) w^{t'-1} dw \\
= & \frac{(-1)^N}{2\pi i} \frac{(\gamma+1)^{M-N}}{\gamma^M} e^{\frac{\gamma}{\gamma+1}Mx} \oint_{\barSigmaM} e^{\frac{(1+\gamma)^{4/3}}{\gamma} M^{1/3}\xi w - \frac{(\gamma+1)^4}{3\gamma^3}Mw^3 - MR_1(w)} \\
& \frac{\left( w + \frac{\gamma}{\gamma+1} \right)^{r-r'}}{\left( w - \frac{1}{\gamma+1} \right)^r} \left( \prod^{s'-1}_{j=1} \left( w + \frac{a_j - \gamma^{-1}}{(1+\gamma^{-1})(1+a_j)} \right)^{r_j} \right) w^{t'-1} dw,
\end{split}
\end{equation}
with $\barSigmaM$ defined by \eqref{eq:definition_of_barGammaM_1}--\eqref{eq:definition_of_barGammaM_4}.

Now we denote
\begin{multline}
F_{Mt'}(\xi, w) = \left( \frac{(\gamma+1)^{4/3}}{\gamma}M^{1/3} \right)^{t'} e^{\frac{(1+\gamma)^{4/3}}{\gamma} M^{1/3}\xi w - \frac{(\gamma+1)^4}{3\gamma^3}Mw^3 - MR_1(w)} \\
\frac{\left( w + \frac{\gamma}{\gamma+1} \right)^{r-r'}}{\left( w - \frac{1}{\gamma+1} \right)^r} \left( \prod^{s'-1}_{j=1} \left( w + \frac{a_j - \gamma^{-1}}{(1+\gamma^{-1})(1+a_j)} \right)^{r_j} \right) w^{t'-1},
\end{multline}
and have
\begin{equation}
(-1)^N \frac{\gamma^M}{(\gamma+1)^{M-N}} e^{-\frac{\gamma}{\gamma+1}Mx} \psi_{r'}(p+q\xi) = \left( \frac{(\gamma+1)^{4/3}}{\gamma}M^{1/3} \right)^{-t'} \frac{1}{2\pi i} \oint_{\barSigmaM} F_{Mt'}(\xi, w) dw,
\end{equation}

Similar to the $a_{s'} < \gamma^{-1}$ case, we have
\begin{lemma}
If $T$ is fixed and $M$ is large enough, 
\begin{equation}
\left\lvert \frac{1}{2\pi i}\int_{\barSigmaM} F_{Mt'}(\xi, w)dw  - (-1)^r \gamma^{r-r'}(\gamma+1)^r \bar{C}_{r'-t'}(-1)^{t'-1} t^{(t)}(\xi) \right\rvert < \frac{e^{-\xi/2}}{M^{1/40}},
\end{equation}
for any $\xi \geq T$.
\end{lemma}
\begin{proof}
Similar to the proofs of lemmas \ref{lemma:first_lemma_for_Sigma_as_less_than_gammainv}--\ref{lemma:fourth_lemma_for_Sigma_as_less_than_gammainv} together.
\end{proof}
Hence we have the convergence result
\begin{multline} \label{eq:asymptotics_of_psi_rprime_equal}
\left\lvert \left( \frac{(\gamma+1)^{4/3}}{\gamma}M^{1/3} \right)^{t'} (-1)^N \frac{\gamma^M}{(\gamma+1)^{M-N}} e^{-\frac{\gamma}{\gamma+1}Mx} e^{\xi/3} \psi_{r'}(p+q\xi) \right. \\
\left. 
\vphantom{\left( \frac{(\gamma+1)^{4/3}}{\gamma}M^{1/3} \right)^{t'}} 
-(-1)^r \gamma^{r-r'}(\gamma+1)^r \bar{C}_{r'-t'} (-1)^{t'-1}e^{\xi/3}t^{(t')}(\xi) \right\rvert < \frac{e^{-\xi/6}}{M^{1/40}}.
\end{multline}

\subsection{Asymptotics of $\varpsi_{r'}(p+q\eta)$}

\label{asymptotics_of_varpsi_rprime_equal}

Similar to \eqref{eq:integral_representation_of_varpsi_less}, we have
\begin{multline}
\varpsi_{r'}(y) = \frac{1}{2\pi i} \\
\oint_{\Gamma} e^{Mf(z) - \frac{(1+\gamma)^{4/3}}{\gamma}M^{1/3}\eta z} \frac{(z-1)^r}{z^{r-r'+1}} \left( \prod^{s'-1}_{j=1} \left( z - \frac{1}{1+a_j} \right)^{-r_j} \right) \left( z - \frac{\gamma}{\gamma+1} \right)^{-t'} dz,
\end{multline}
and after the substitution $z = w + \frac{\gamma}{\gamma+1}$, we get
\begin{equation}
\begin{split}
& \varpsi_{r'}(p+q\eta) \\
= & \frac{1}{2\pi i} \oint_{\BbarGammaM} e^{M -\left( \frac{\gamma+1}{\gamma} + \log\gamma - (1-\gamma^{-2})\log(\gamma+1) + \gamma^{-2}\pi i + \frac{(\gamma+1)^4}{3\gamma^3}w^3 + R_1(w) \right) - \frac{(1+\gamma)^{4/3}}{\gamma} M^{1/3}\eta \left( w+ \frac{\gamma}{\gamma+1} \right)} \\
& \frac{\left( w - \frac{1}{\gamma+1} \right)^r}{\left( w + \frac{\gamma}{\gamma+1} \right)^{r-r'+1}} \left( \prod^{s'-1}_{j=1} \left( w + \frac{a_j - \gamma^{-1}}{(1+\gamma^{-1})(1+a_j)} \right)^{-r_j} \right) w^{-t'} dw \\
= & \frac{(-1)^N}{2\pi i} \frac{\gamma^M}{(\gamma+1)^{M-N}} e^{-\frac{\gamma}{\gamma+1}Mx} \oint_{\BbarGammaM} e^{-\frac{(1+\gamma)^{4/3}}{\gamma} M^{1/3}\eta w + \frac{(\gamma+1)^4}{3\gamma^3}Mw^3 + MR_1(w)} \\
& \frac{\left( w - \frac{1}{\gamma+1} \right)^r}{\left( w + \frac{\gamma}{\gamma+1} \right)^{r-r'+1}} \left( \prod^{s'-1}_{j=1} \left( w + \frac{a_j - \gamma^{-1}}{(1+\gamma^{-1})(1+a_j)} \right)^{-r_j} \right) w^{-t'} dw,
\end{split}
\end{equation}
where $\BbarGammaM$ is slightly different from $\barGammaM$: $\BbarGammaM$ is composed of $\BbarGammaM_1$, \dots, $\BbarGammaM_6$, which are
\begin{align}
\BbarGammaM_1 = & \left\{ \left. \left( \frac{2}{\gamma} - t \right)\frac{\gamma}{\gamma+1}e^{\frac{\pi i}{3}} \right\rvert 0 \leq t \leq \frac{2}{\gamma} - \frac{1/6}{(\gamma+1)^{1/3}}M^{-1/3} \right\}, \\
\BbarGammaM_2 = & \left\{ \left. \frac{\gamma/6}{(\gamma+1)^{4/3}}M^{-1/3}e^{t\pi i} \right\rvert \frac{1}{3} \leq t \leq \frac{5}{3} \right\}, \\
\BbarGammaM_3 = & \left\{ \left. t\frac{\gamma}{\gamma+1}e^{\frac{5\pi i}{3}} \right\rvert \frac{1/6}{(\gamma+1)^{1/3}}M^{-1/3} \leq t \leq \frac{2}{\gamma} \right\}, \\
\BbarGammaM_* = & \barGammaM_*, \text{ for } * = 4,5,6.
\end{align}
\begin{figure}[h]
\centering
\includegraphics{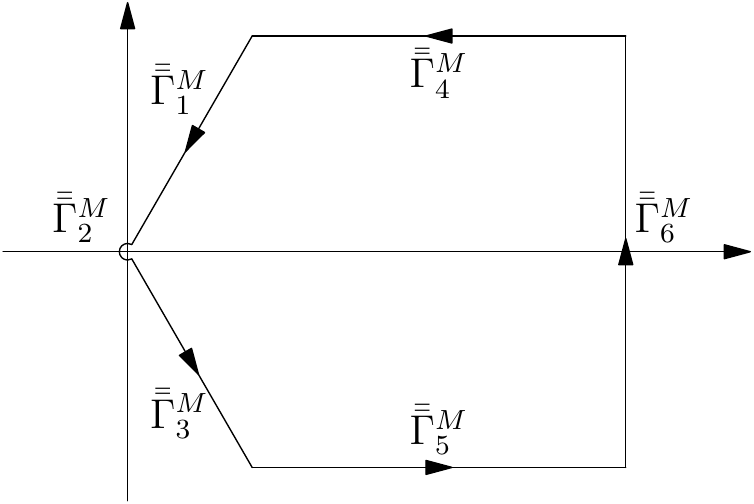}
\caption{$\BbarGammaM$}
\end{figure}
We also define
\begin{align}
\BbarGammaM_{\local} = & \{ z \in \BbarGammaM \mid \Re(z) < M^{-10/39} \}, \\
\BbarGammaM_{\remote} = & (\BbarGammaM_1 \cup \BbarGammaM_3) \setminus \BbarGammaM_{\local}.
\end{align}

If we denote
\begin{multline}
G_{Mt'}(\eta, w) = \left( \frac{(\gamma+1)^{4/3}}{\gamma}M^{1/3} \right)^{1-t'} e^{-\frac{(1+\gamma)^{4/3}}{\gamma} M^{1/3}\eta w + \frac{(\gamma+1)^4}{3\gamma^3}Mw^3 + MR_1(w)} \\
\frac{\left( w - \frac{1}{\gamma+1} \right)^r}{\left( w + \frac{\gamma}{\gamma+1} \right)^{r-r'+1}} \left( \prod^{s'-1}_{j=1} \left( w + \frac{a_j - \gamma^{-1}}{(1+\gamma^{-1})(1+a_j)} \right)^{-r_j} \right) w^{-t'},
\end{multline}
we have
\begin{equation}
(-1)^N \frac{(\gamma+1)^{M-N}}{\gamma^M} e^{\frac{\gamma}{\gamma+1}My} \varpsi_{r'}(p+q\eta) = \left( \frac{(\gamma+1)^{4/3}}{\gamma}M^{1/3} \right)^{t'-1} \frac{1}{2\pi i} \oint_{\BbarGammaM} G_{Mt'}(\eta, w) dw,
\end{equation}
and
\begin{lemma}
If $T$ is fixed and $M$ is large enough, 
\begin{equation}
\left\lvert \frac{1}{2\pi i}\int_{\BbarGammaM} G_{Mt'}(\eta, w)dw  - \frac{(-1)^r (1+\gamma^{-1})}{\gamma^{r-r'}(\gamma+1)^{r'}\bar{C}_{r'-t'}} (-1)^{t'-1} s^{(t)}(\eta) \right\rvert < \frac{e^{\eta/6}}{M^{1/40}},
\end{equation}
for any $\eta \geq T$.
\end{lemma}
\begin{proof}[Sketch of proof]
The integral of $G_{Mt'}(\eta, w)$ over $\BbarGammaM \setminus \BbarGammaM_{\local}$ is negligible, and we can estimate it as lemmas \ref{lemma:first_for_varpsi_less} and \ref{lemma:second_for_varpsi_less}, and get the same result. On $\BbarGammaM_{\local}$, if $\eta \geq T$, $\lvert e^{-(\eta-T) u} \rvert \leq e^{-T/6}e^{\eta/6}$, and we can carry out the proof like that of lemma \ref{lemma:last_for_varpsi_less}, with the upper bound of $\lvert e^{-(\eta-T) u} \rvert$ changed from $e^{T/2}e^{-\eta/2}$ to $e^{-T/6}e^{\eta/6}$.
\end{proof}

Therefore we have the convergence result
\begin{multline} \label{eq:asymptotics_of_varpsi_rprime_equal}
\left\lvert \left( \frac{(\gamma+1)^{4/3}}{\gamma}M^{1/3} \right)^{1-t'} (-1)^N \frac{(\gamma+1)^{M-N}}{\gamma^M} e^{\frac{\gamma}{\gamma+1}My} e^{-\eta/3} \varpsi_{r'}(p+q\eta) \right. \\
\left.
\vphantom{\left( \frac{(\gamma+1)^{4/3}}{\gamma}M^{1/3} \right)^{1-t'}}
-\frac{(-1)^r (1+\gamma^{-1})}{\gamma^{r-r'}(\gamma+1)^{r'}\bar{C}_{r'-t'}} (-1)^{t'-1}e^{-\eta/3}s^{(t')}(\eta) \right\rvert < \frac{e^{-\eta/6}}{M^{1/40}}.
\end{multline}

%% file: asymptotics_more.tex
\section{Asymptotics of $\psi(p+q\xi)$, $\psi_{r'}(p+q\xi)$, $\varpsi(p+q\eta)$ and $\varpsi_{r'}(p+q\eta)$ when $a_{s'} \leq a$}

\label{asymptotics_more}

In this section, $a$ is a member greater than $\gamma^{-1}$, and we assume $p = (1+a)\left( 1+\frac{1}{\gamma^2 a} \right)$ and  $q = (1+a)\sqrt{1 - \frac{1}{\gamma^2 a^2}} \frac{1}{\sqrt{M}}$. 

We define
\begin{align}
\tildeSigmainf = & \{ it -1 \mid -\infty \leq t \leq \infty \}, \\
\tildeSigmainf_{<c} = & \{ w \in \tildeSigmainf \mid \lvert w \rvert \leq c \}, \\
\tildeSigmainf_{\geq c} = & \tildeSigmainf \setminus \tildeSigmainf_{<c},
\end{align}
\begin{figure}[h]
\begin{minipage}[b]{0.5\linewidth}
\centering
\includegraphics{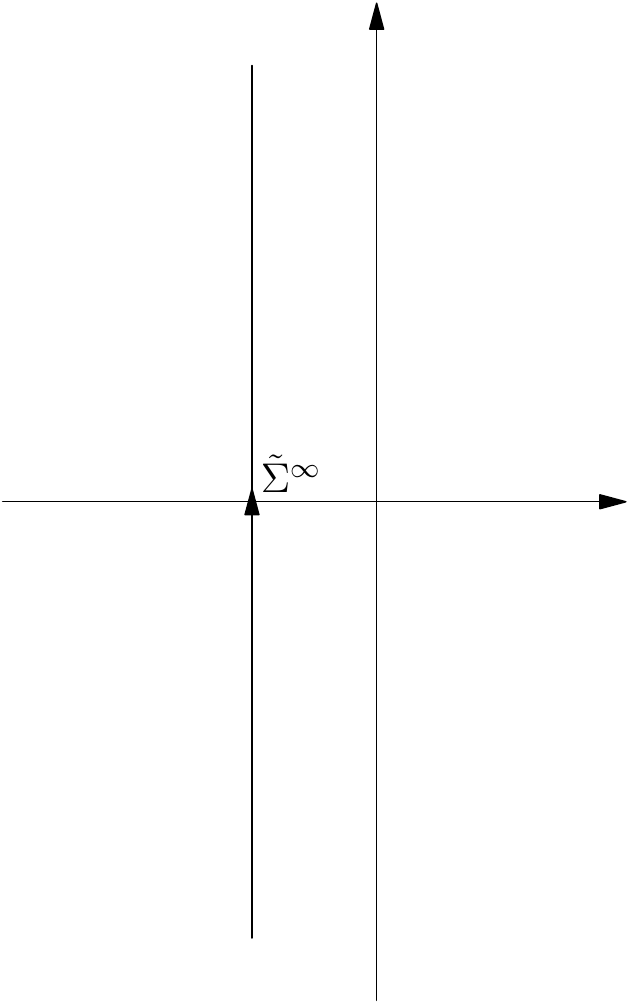}
\caption{$\tildeSigmainf$}
\end{minipage}
\hspace{0.5cm}
\begin{minipage}[b]{0.5\linewidth}
\centering
\includegraphics{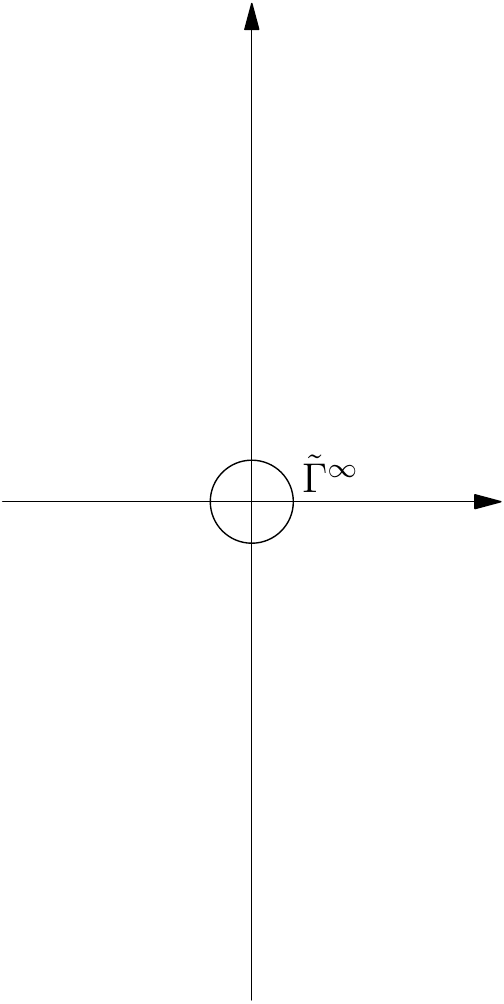}
\caption{$\tildeGammainf$}
\end{minipage}
\end{figure}
and for any $T$, if $c$ is large enough,
\begin{equation}
\left\lvert \frac{1}{2\pi i} \int_{\tildeSigmainf_{\geq c}} e^{T u + \frac{u^2}{2}} u^{t'-1} du \right\rvert < \frac{1}{c}.
\end{equation}
We also define
\begin{equation}
\tildeGammainf = \{ e^{it}/3 \mid 0 \leq t \leq 2\pi \}.
\end{equation}

We need two integral representations of Hermite polynomials. First, we have the explicit formulas for $H_n(x)$
\begin{align}
H_{2k}(x) = & (-1)^k (2k-1)!! \left( 1 + \sum^k_{l=1} \frac{(-2k)(-2k+2) \dots (-2k+2j-2)}{(2j)!} x^{2j} \right), \label{eq:even_Hermite_poly} \\
H_{2k+1}(x) = & (-1)^k (2k+1)!! \left( x + \sum^k_{l=1} \frac{(-2k)(-2k+2) \dots (-2k+2j-2)}{(2j)!} x^{2j+1} \right). \label{eq:odd_Hermite_poly}
\end{align}
On the other hand, ($v = u + \xi$)
\begin{align}
\frac{1}{2\pi i} \int_{\tildeSigmainf} e^{\xi u + \frac{u^2}{2}} u^{t'-1} du = & e^{\frac{\xi^2}{2}} \frac{1}{2\pi i} \int_{\tildeSigmainf} e^{\frac{(\xi + u)^2}{2}} u^{t'-1} du \notag \\
= & e^{-\frac{\xi^2}{2}} \frac{1}{2\pi i} \int^{\infty i + \xi}_{-\infty i + \xi} e^{\frac{v^2}{2}} (v - \xi)^{t'-1} dv \label{eq:change_u_into_v} \\
= & e^{-\frac{\xi^2}{2}} \frac{1}{2\pi i} \int_{\tildeSigmainf} e^{\frac{v^2}{2}} (v - \xi)^{t'-1} dv, \label{eq:change_into_parallel_arc}
\end{align}
where the integral in \eqref{eq:change_u_into_v} is along the vertical line parallel to $\tildeSigmainf$ through the point $v = \xi$, and the equivalency of \eqref{eq:change_u_into_v} and \eqref{eq:change_into_parallel_arc} is a simple application of the Cauchy integral formula. Now if we write
\begin{equation} \label{eq:evaluation_of_complex_integral}
\frac{1}{2\pi i} \int_{\tildeSigmainf} e^{\frac{v^2}{2}} (v - \xi)^{t'-1} dv = \sum^{t'-1}_{j=0} \binom{t'-1}{j} (-\xi)^{t'-j-1} \frac{1}{2\pi i} \int_{\tildeSigmainf} e^{\frac{v^2}{2}} v^j dv,
\end{equation}
and
\begin{equation} \label{eq:evaluation_of_coefficients}
\frac{1}{2\pi i} \int_{\tildeSigmainf} e^{\frac{v^2}{2}} v^j dv = \frac{i^{j}}{2\pi} \int^{\infty}_{-\infty} e^{-\frac{x^2}{2}} x^j dx = 
\begin{cases}
0 & \text{$j$ odd,} \\
\frac{(-1)^k}{\sqrt{2\pi}}(2k-1)!! & j = 2k.
\end{cases}
\end{equation}
Compare \eqref{eq:change_into_parallel_arc}, \eqref{eq:evaluation_of_complex_integral} and \eqref{eq:evaluation_of_coefficients} to \eqref{eq:even_Hermite_poly} and \eqref{eq:odd_Hermite_poly}, we get 
\begin{equation}
\frac{1}{2\pi i} \int_{\tildeSigmainf} e^{\xi u + \frac{u^2}{2}} u^{t'-1} du = (-1)^{t'-1} \frac{H_{t'-1}(\xi)}{\sqrt{2\pi}} e^{-\frac{\xi^2}{2}}.
\end{equation}

The second integral representation of Hermite polynomials is more familar:
\begin{equation}
\frac{1}{2\pi i} \int_{\tildeGammainf} e^{-\eta u - \frac{u^2}{2}} \frac{du}{u^{t'}} = \frac{(-1)^{t'-1}}{(t'-1)!}H_{t'-1}(\eta).
\end{equation}

\subsection{Asymptotics of $\psi(p+q\xi)$ and $\psi_{r'}(p+q\xi)$}

\label{asymptotics_of_psi_rprime_more}

We only consider $\psi_{r'}$ with $a_{s'} = a$. $\psi_{r'}$ with $a_{s'} < a$ and $\psi$ can be solved similarly, and we only give the results.

We have
\begin{equation}
e^{Mxz}\frac{(z-1)^N}{z^M} = e^{-Mg(z) + (1+a)\sqrt{\left( 1 - \frac{1}{\gamma^2 a^2} \right) M} \xi z},
\end{equation}
where
\begin{equation} \label{eq:definition_of_g_x}
g(z) = -(a+1)\left( 1 + \frac{1}{\gamma^2 a} \right)z + \log z - \gamma^{-2}\log(z-1).
\end{equation}
Now we can write \eqref{eq:definition_of_phi_rprime} as
\begin{multline} \label{eq:another_integral_formula_of_psi_rprime_more}
\psi_{r'}(x) = \frac{1}{2\pi i} \oint_{\Sigma} e^{-Mg(z) + (1+a)\sqrt{\left( 1 - \frac{1}{\gamma^2 a^2} \right) M} \xi z} \\
\frac{z^{r-r'}}{(z-1)^r} \left( \prod^{s'-1}_{j=1} \left( z - \frac{1}{1+a_j} \right)^{r_j} \right) \left( z - \frac{1}{1+a} \right)^{t'-1} dz.
\end{multline}

For $g(z)$, we have
\begin{itemize}
\item $\displaystyle g'(z) = -(a+1)\left( 1 + \frac{1}{\gamma^2 a} \right) + \frac{1}{z} - \frac{\gamma^{-2}}{z-1}$, with zero points $\displaystyle z= \frac{\gamma^2 a}{1 + \gamma^2 a}$ and $\displaystyle z = \frac{1}{1+a}$.

\item $\displaystyle g''(z) = -\frac{1}{z^2} + \frac{\gamma^{-2}}{(z-1)^2}$, $\displaystyle g''\left( \frac{\gamma^2 a}{1 + \gamma^2 a} \right) = (\gamma^{-1} + \gamma a)^2 \left( 1 - \frac{1}{\gamma^2 a^2} \right) > 0$ and $\displaystyle g''\left( \frac{1}{1+a} \right) = -(1+a)^2 \left( 1 - \frac{1}{\gamma^2 a^2} \right) < 0$.
\end{itemize}
Hence locally around $z = \frac{1}{1+a}$,
\begin{multline}
g\left( \frac{1}{1+a} + w \right) = -\left( 1 + \frac{1}{\gamma^2 a} \right) - (1 - \gamma^{-2})\log(1+a) - \gamma^{-2}\log a + \gamma^{-2}\pi i \\
- \frac{1}{2} (1+a)^2 \left( 1 - \frac{1}{\gamma^2 a^2} \right) w^2 + R_2(w),
\end{multline}
where
\begin{equation}
R_2(w) = \varO(w^3), \text{ as } w \rightarrow \infty.
\end{equation}
After the substitution $z = w + \frac{1}{1+a}$, we get by \eqref{eq:another_integral_formula_of_psi_rprime_more}
\begin{equation} \label{eq:substitution_of_the_integral_psi_more}
\begin{split}
& \psi_{r'}(p+q\xi) \\
= & \frac{1}{2\pi i} \oint_{\tildeSigmaM} e^{M \left( \left( 1 + \frac{1}{\gamma^2 a} \right) + (1 - \gamma^{-2})\log(1+a) + \gamma^{-2}\log a - \gamma^{-2}\pi i + \frac{1}{2} (1+a)^2 \left( 1 - \frac{1}{\gamma^2 a^2} \right) w^2 - R_2(w) \right)} \\
& e^{ (1+a)\sqrt{\left( 1 - \frac{1}{\gamma^2 a^2} \right) M} \xi \left( w + \frac{1}{1+a} \right)} \frac{\left( w + \frac{1}{1+a} \right)^{r-r'}}{\left( w - \frac{a}{1+a} \right)^r} \left( \prod^{s'-1}_{j=1} \left( w + \frac{a_j - a}{(1+a)(1+a_j)} \right)^{r_j} \right) w^{t'-1} dw \\
= & \frac{(-1)^N}{2\pi i} a^N (1+a)^{M-N} e^{\frac{M}{1+a}x} \oint_{\tildeSigmaM} e^{\frac{1}{2} (1+a)^2 \left( 1 - \frac{1}{\gamma^2 a^2} \right) Mw^2 - MR_2(w)} \\
& \frac{\left( w + \frac{1}{1+a} \right)^{r-r'}}{\left( w - \frac{a}{1+a} \right)^r} \left( \prod^{s'-1}_{j=1} \left( w + \frac{a_j - a}{(1+a)(1+a_j)} \right)^{r_j} \right) w^{t'-1} dw,
\end{split}
\end{equation}
where $\tildeSigmaM$ is a contour around  $w = -\frac{1}{1+a}$, composed of $\tildeSigmaM_1$, $\tildeSigmaM_2$, $\tildeSigmaM_3$ and $\tildeSigmaM_4$, which are defined as ($q = (1+a)\sqrt{1 - \frac{1}{\gamma^2 a^2}} \frac{1}{\sqrt{M}}$)
\begin{align}
\tildeSigmaM_1 = & \{ it - q^{-1}/M \mid -2 \leq t \leq 2 \}, \\
\tildeSigmaM_2 = & \{ 2i - t \mid q^{-1}/M \leq t \leq 4 \}, \\
\tildeSigmaM_3 = & \{ -4 - it \mid -2 \leq t \leq 2 \}, \\
\tildeSigmaM_4 = & \{ t - 2i \mid -4 \leq t \leq -q^{-1}/M \}, 
\end{align}
\begin{figure}[h]
\centering
\includegraphics{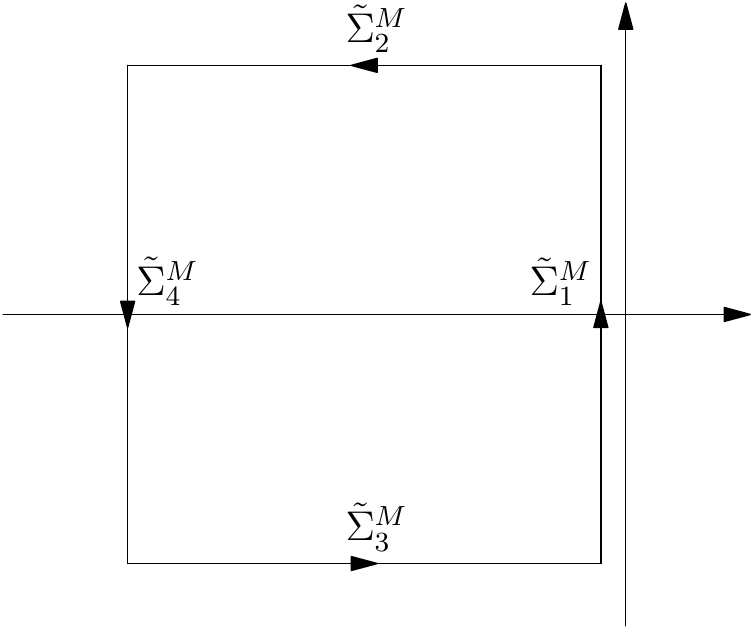}
\caption{$\tildeSigmaM$}
\end{figure}
and for the asymptotic analysis, we define
\begin{align}
\tildeSigmaM_{\local} = & \{ w \in \tildeSigmaM \mid \lvert w \rvert \leq M^{-2/5} \}, \\
\tildeSigmaM_{\remote} = & \tildeSigmaM_1 \setminus \tildeSigmaM_{\local}.
\end{align}
Now if we denote 
\begin{multline}
F_{Mat'}(w) = \\
\left( (1+a)\sqrt{\left( 1 - \frac{1}{\gamma^2 a^2} \right) M} \right)^{t'} e^{(1+a)\sqrt{\left( 1 - \frac{1}{\gamma^2 a^2} \right) M} \xi w + \frac{1}{2} (1+a)^2 \left( 1 - \frac{1}{\gamma^2 a^2} \right) Mw^2 - MR_2(w)} \\
\frac{\left( w + \frac{1}{1+a} \right)^{r-r'}}{\left( w - \frac{a}{1+a} \right)^r} \left( \prod^{s'-1}_{j=1} \left( w + \frac{a_j - a}{(1+a)(1+a_j)} \right)^{r_j} \right) w^{t'-1},
\end{multline}
We have 
\begin{equation} \label{eq:variation_of_integral_formula_of_psi_more}
\left( (1+a)\sqrt{\left( 1 - \frac{1}{\gamma^2 a^2} \right) M} \right)^{t'} \frac{e^{-\frac{M}{1+a}x}}{(-a)^N (1+a)^{M-N}} \psi_{r'}(p+q\xi) = \frac{1}{2\pi i} \oint_{\tildeSigmaM} F_{Mat'}(\xi, w) dw,
\end{equation}
and results similar to lemmas \ref{lemma:first_lemma_for_Sigma_as_less_than_gammainv}--\ref{lemma:fourth_lemma_for_Sigma_as_less_than_gammainv}:
                                                                     
\begin{lemma} \label{lemma:first_lemma_for_psi_more}
\begin{equation}
\left\lvert \frac{1}{2\pi i} \int_{\tildeSigmaM_2 \cup \tildeSigmaM_3 \cup \tildeSigmaM_4} F_{Mat'}(\xi, w)dw \right\rvert < \frac{1}{3} \frac{e^{-\xi}}{M^{1/10}},
\end{equation}
for any $\xi \geq T$.
\end{lemma}
                                         
\begin{lemma}                          
\begin{equation}
\left\lvert \frac{1}{2\pi i} \int_{\tildeSigmaM_{\remote}} F_{Mat'}(\xi, w)dw \right\rvert < \frac{1}{3} \frac{e^{-\xi}}{M^{1/10}},
\end{equation}
for any $\xi \geq T$.
\end{lemma}

\begin{lemma} \label{lemma:last_lemma_for_psi_more}
If $T$ is fixed and $M$ is large enough, 
\begin{equation}
\left\lvert \frac{1}{2\pi i}\int_{\tildeSigmaM_{\local}} F_{Mat'}(\xi, w)dw  - \frac{(1+a)^{r'}}{(-a)^r}\bar{C}_{a,r'-t'} (-1)^{t'-1}\frac{H_{t'-1}(x)}{\sqrt{2\pi}} e^{-\frac{\xi^2}{2}} \right\rvert < \frac{1}{3} \frac{e^{-\xi}}{M^{1/10}},
\end{equation}
for any $\xi \geq T$.
\end{lemma}

Since their proofs are similar to those of lemmas \ref{lemma:first_lemma_for_Sigma_as_less_than_gammainv}--\ref{lemma:fourth_lemma_for_Sigma_as_less_than_gammainv}, we only give the proof of lemma \ref{lemma:last_lemma_for_psi_more}.

\begin{proof}[Proof of lemma \ref{lemma:last_lemma_for_psi_more}]
On $\tildeSigmaM_{\local}$, $\lvert w \rvert < M^{-2/5}$, and $R_2(w) = \varO(M^{-6/5})$, so that
\begin{equation}
\begin{split}
 F_{Mat'}(\xi, w) 
= & \frac{(1+a)^{r'}}{(-a)^r} \left( \prod^{s'-1}_{j=1} \left( \frac{a_j-a}{(1+a)(1+a_l)} \right)^{r_j} \right) \left( (1+a)\sqrt{\left( 1 - \frac{1}{\gamma^2 a^2} \right) M} \right)^{t'} \\ 
& e^{(1+a)\sqrt{\left( 1 - \frac{1}{\gamma^2 a^2} \right) M} \xi w + \frac{1}{2} (1+a)^2 \left( 1 - \frac{1}{\gamma^2 a^2} \right) Mw^2} w^{t'-1} \left( 1 + \varO(M^{-1/5}) \right),
\end{split}
\end{equation}
and the $\varO(M^{-1/5})$ term is independent of $\xi$. After the substitution $u = (1+a)\sqrt{\left( 1 - \frac{1}{\gamma^2 a^2} \right) M} w$, we get
\begin{multline}
\frac{1}{2\pi i}\int_{\tildeSigmaM_{\local}} F_{Mat'}(\xi, w)dw = \\
\frac{(1+a)^{r'}}{(-a)^r}\bar{C}_{a,r'-t'} \frac{1}{2\pi i} \int_{\tildeSigmainf_{< (1+a)\sqrt{1 - \frac{1}{\gamma^2 a^2}}M^{1/10}}} e^{\xi u + \frac{u^2}{2}}u^{t'-1} du \left( 1 + \varO(M^{-1/5}) \right).
\end{multline}
On $\tildeSigmaM$, if $\xi \geq T$, $\lvert e^{(\xi-T) u} \rvert \leq e^T e^{-\xi}$, and we have
\begin{equation}
\begin{split}
& \left\lvert \frac{1}{2\pi i} \int_{\tildeSigmaM_{\local}} F_M(\xi, w)dw \chi(\xi) - \frac{(1+a)^{r'}}{(-a)^r}\bar{C}_{a,r'-t'} (-1)^{t'-1}\frac{H_{t'-1}(x)}{\sqrt{2\pi}} e^{-\frac{\xi^2}{2}} \right\rvert \\
\leq & e^T e^{-\xi} \left\lvert \frac{(1+a)^{r'}}{(-a)^r}\bar{C}_{a,r'-t'} \frac{1}{2\pi i} \int_{\tildeSigmainf_{\geq (1+a)\sqrt{1 - \frac{1}{\gamma^2 a^2}}M^{1/10}}} \left\lvert e^{Tu -\frac{u^2}{2}}u^{t'-1} \right\rvert du (1 + \varO(M^{-1/5})) \right\rvert \\
& + e^T e^{-\xi} \left\lvert \frac{(1+a)^{r'}}{(-a)^r}\bar{C}_{a,r'-t'} \frac{1}{2\pi i} \int_{\tildeSigmainf_{< (1+a)\sqrt{1 - \frac{1}{\gamma^2 a^2}}M^{1/10}}} \left\lvert e^{Tu -\frac{u^2}{2}}u^{t'-1} \right\rvert du \varO(M^{-1/5}) \right\rvert,
\end{split}
\end{equation}
and we can get the result by direct calculation.
\end{proof}

By lemmas \ref{lemma:first_lemma_for_psi_more}--\ref{lemma:last_lemma_for_psi_more} and \eqref{eq:variation_of_integral_formula_of_psi_more}, we have the convergence result for $\psi_{r'}$ with $a_{s'} = a$
\begin{multline} \label{eq:significant_psi_rprime_more}
\left\lvert \left( (1+a)\sqrt{\left( 1 - \frac{1}{\gamma^2 a^2} \right) M} \right)^{t'} \frac{e^{-\frac{M}{1+a}x}}{(-a)^N (1+a)^{M-N}} e^{2\xi/3} \psi_{r'}(p+q\xi) \right. \\
\left. 
\vphantom{\left( (1+a)\sqrt{\left( 1 - \frac{1}{\gamma^2 a^2} \right) M} \right)^t}
- \frac{(1+a)^{r'}}{(-a)^r}\bar{C}_{a,r'-t'} (-1)^{t'-1} e^{2\xi/3} \frac{H_{t'-1}(x)}{\sqrt{2\pi}} e^{-\frac{\xi^2}{2}} \right\rvert < \frac{e^{-\xi/3}}{M^{1/10}}.
\end{multline}

By the same method, we can get the convergence result for $\psi_{r'}(\xi)$ with $r' = 1, \dots, r-t$ and $\xi \geq T$
\begin{multline} \label{eq:insignificant_psi_rprime_more}
\left\lvert (1+a)\sqrt{\left( 1 - \frac{1}{\gamma^2 a^2} \right) M} \frac{e^{-\frac{M}{1+a}x}}{(-a)^N (1+a)^{M-N}} e^{2\xi/3} \psi_{r'}(p+q\xi) \right. \\
\left.
\vphantom{(1+a)\sqrt{\left( 1 - \frac{1}{\gamma^2 a^2} \right) M}}
- \frac{(1+a)^{r'}}{(-a)^r}\bar{C}_{a,r'-1} \frac{e^{\frac{2\xi}{3} - \frac{\xi^2}{2}}}{\sqrt{2\pi}} \right\rvert < \frac{e^{-\xi/3}}{M^{1/10}},
\end{multline}
and for $\psi(\xi)$ with $\xi \geq T$
\begin{equation} \label{eq:insignificant_psi_more}
\left\lvert (1+a)\sqrt{\left( 1 - \frac{1}{\gamma^2 a^2} \right) M} \frac{e^{-\frac{M}{1+a}x}}{(-a)^N (1+a)^{M-N}} e^{2\xi/3} \psi(p+q\xi) - (-a)^{-r} \frac{e^{\frac{2\xi}{3} - \frac{\xi^2}{2}}}{\sqrt{2\pi}} \right\rvert < \frac{e^{-\xi/3}}{M^{1/10}}.
\end{equation}

\subsection{Asymptotics of $\varpsi(p+q\eta)$ and $\varpsi_{r'}(p+q\eta)$}

We only consider $\varpsi_{r'}$ with $a_{s'} = a$, and only state the results for $\varpsi_{r'}$ with $a_{s'} < a$ and $\varpsi$, since they are simpler.

$\varpsi_{r'}(p+q\eta)$ is defined in \eqref{eq:definition_of_varphi_rprime} by a contour integral with poles $z = 1$ and $\frac{1}{1+a_j}$, $j = 1, \dots, s$ inside the contour. By the residue theorem, the value of $\varpsi_{r'}(p+q\eta)$ is the sum of residues at these poles. We will see that the contribution of poles other than $\frac{1}{1+a}$ are negligible, and we are going to calculate the residue at $\frac{1}{1+a}$. To consider these two kinds of poles separately, we deform $\Gamma$ into the sum of two disconnected contours $\Gamma_{\frac{1}{1+a}}$ and $\Gamma_{\rt}$, where $\Gamma_{\frac{1}{1+a}}$ includes $\frac{1}{1+a}$ and excludes other poles, and vice versa for $\Gamma_{\rt}$.

We have
\begin{equation}
e^{-Myz}\frac{(z-1)^N}{z^M} = e^{-Mg(z) + (1+a)\sqrt{\left( 1 - \frac{1}{\gamma^2 a^2} \right) M} \eta z},
\end{equation}
where $g(z)$ is defined in \eqref{eq:definition_of_g_x}. Then we can write \eqref{eq:definition_of_varphi_rprime} as
\begin{multline}
\varpsi_{r'}(y) = \frac{1}{2\pi i} \oint_{\Gamma_{\frac{1}{1+a}} \cup \Gamma_{\rt}} e^{-Mg(z) + (1+a)\sqrt{\left( 1 - \frac{1}{\gamma^2 a^2} \right) M} \eta z} \\
\frac{(z-1)^r}{z^{r-r'+1}} \left( \prod^{s'-1}_{j=1} \left( z - \frac{1}{1+a_j} \right)^{-r_j} \right) \left( z - \frac{1}{1+a} \right)^{-t'} dz,
\end{multline}
and after the substitution $z = w + \frac{1}{1+a}$, we get
\begin{equation}
\begin{split}
& \oint_{\Gamma_{\frac{1}{1+a}}} e^{Mg(z) - (1+a)\sqrt{\left( 1 - \frac{1}{\gamma^2 a^2} \right) M} \eta z} \frac{(z-1)^r}{z^{r-r'+1}} \left( \prod^{s'-1}_{j=1} \left( z - \frac{1}{1+a_j} \right)^{-r_j} \right) \left( z - \frac{1}{1+a} \right)^{-t'} dz \\
= & \oint_{\tildeGammaM_0} e^{M \left( -\left( 1 - \frac{1}{\gamma^2 a} \right) - (1 - \gamma^{-2})\log(1+a) - \gamma^{-2}\log a + \gamma^{-2}\pi i - \frac{1}{2} (1+a)^2 \left( 1 - \frac{1}{\gamma^2 a^2} \right) w^2 + R_2(w) \right)} \\
& e^{ -(1+a)\sqrt{\left( 1 - \frac{1}{\gamma^2 a^2} \right) M} \eta \left( w + \frac{1}{1+a} \right)} \frac{\left( w - \frac{a}{1+a} \right)^r}{\left( w + \frac{1}{1+a} \right)^{r-r'+1}} \left( \prod^{s'-1}_{j=1} \left( w + \frac{a_j - a}{(1+a)(1+a_j)} \right)^{-r_j} \right) \frac{dw}{w^{t'}} \\
= & \frac{(-1)^N}{a^N (1+a)^{M-N}} e^{-\frac{M}{1+a}y} \oint_{\tildeGammaM_0} e^{-(1+a)\sqrt{\left( 1 - \frac{1}{\gamma^2 a^2} \right) M} \eta w - \frac{1}{2} (1+a)^2 \left( 1 - \frac{1}{\gamma^2 a^2} \right) Mw^2 - MR_2(w)} \\
& \frac{\left( w - \frac{a}{1+a} \right)^r}{\left( w + \frac{1}{1+a} \right)^{r-r'+1}} \left( \prod^{s'-1}_{j=1} \left( w + \frac{a_j - a}{(1+a)(1+a_j)} \right)^{-r_j} \right) \frac{dw}{w^{t'}},
\end{split}
\end{equation}
where $\tildeGammaM_0$ is a contour around $w = 0$, defined as
\begin{equation}
\tildeGammaM_0 = \left\{ \left. \frac{e^{it}/3}{(1+a)\sqrt{\left( 1 - \frac{1}{\gamma^2 a^2} \right) M}} \right\rvert 0 \leq t \leq 2\pi \right\}.
\end{equation}
\begin{figure}[h]
\centering
\includegraphics{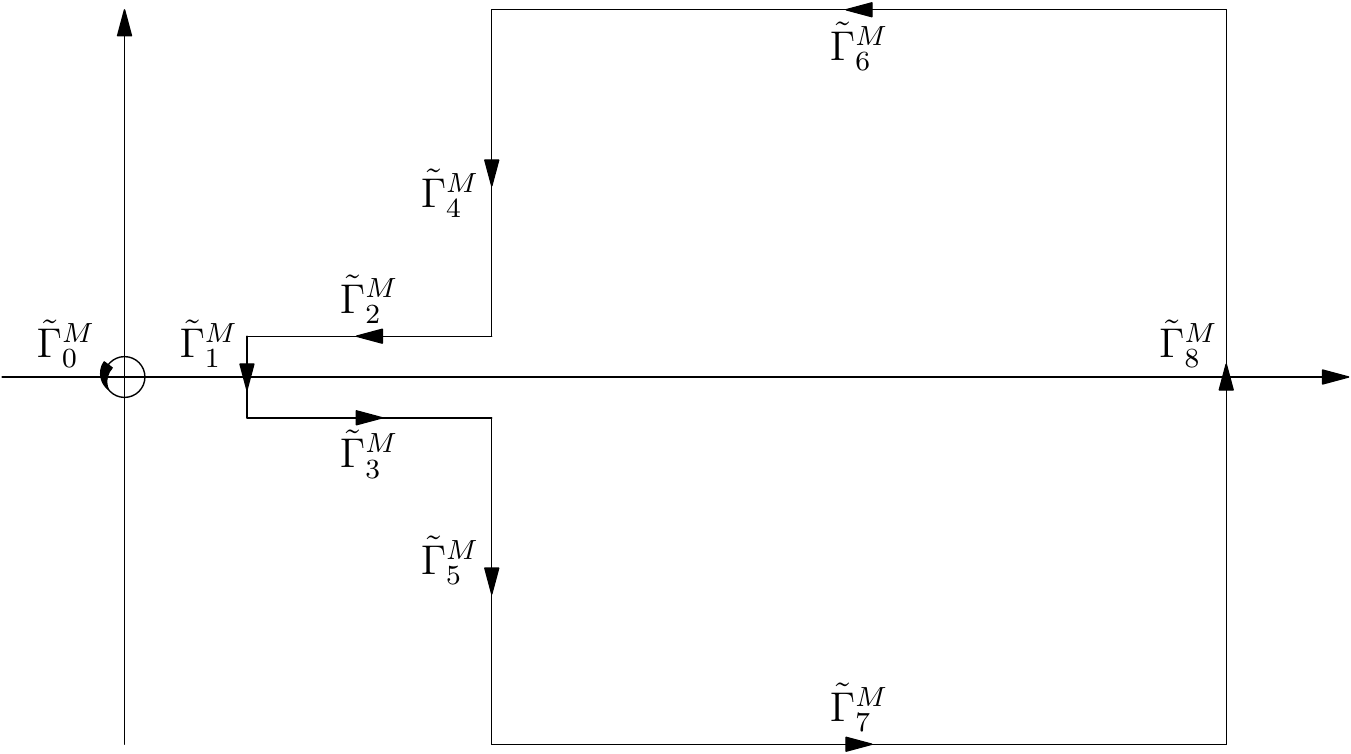}
\caption{$\tildeGammaM_0$ and $\tildeGammaM_{\rt}$}
\end{figure}

If we denote
\begin{multline} 
G_{Mat'}(\eta, w) = \\
\left( (1+a)\sqrt{\left( 1 - \frac{1}{\gamma^2 a^2} \right) M} \right)^{1-t'} e^{-(1+a)\sqrt{\left( 1 - \frac{1}{\gamma^2 a^2} \right) M} \eta w - \frac{1}{2} (1+a)^2 \left( 1 - \frac{1}{\gamma^2 a^2} \right) Mw^2 + MR_2(w)} \\
\frac{\left( w - \frac{a}{1+a} \right)^r}{\left( w + \frac{1}{1+a} \right)^{r-r'+1}} \left( \prod^{s'-1}_{j=1} \left( w + \frac{a_j - a}{(1+a)(1+a_j)} \right)^{-r_j} \right) \frac{1}{w^{t'}},
\end{multline}
We have
\begin{multline} \label{eq:variation_of_varpsi_more}
\left( (1+a)\sqrt{\left( 1 - \frac{1}{\gamma^2 a^2} \right) M} \right)^{1-t'} (-a)^N (1+a)^{M-N} e^{\frac{M}{1+a}y} \varpsi_{r'}(p+q\eta) = \\
\frac{1}{2\pi i} \oint_{\tildeGammaM_0} G_{Mat'}(\eta, w) dw + \oint_{\tildeGammaM_{\rt}} G_{Mat'}(\eta, w) dw,
\end{multline}
with the definition of $\tildeGammaM_{\rt}$ to be introduced later, and
\begin{lemma} \label{lemma:first_lemma_for_varphi_more}
If $T$ is fixed and $M$ is large enough,
\begin{equation}
\left\lvert \frac{1}{2\pi i} \oint_{\tildeGammaM_0} G_{Mat'}(\eta, w) dw - \frac{(-a)^r}{(1+a)^{r'-1}\bar{C}_{r'-t'}} \frac{(-1)^{t'-1}}{(t'-1)!} H_{t'-1}(\eta) \right\rvert < \frac{1}{2} \frac{e^{\eta/3}}{M^{1/3}},
\end{equation}
for any $\eta \geq T$.
\end{lemma}

\begin{proof}
On $\tildeGammaM_0$, $\lvert w \rvert = M^{-1/2}$, and $R_2(w) = \varO(M^{-3/2})$, so that
\begin{multline}
G_{Mat'}(\eta, w) = \\
\frac{(-a)^r}{(1+a)^{r'-1}} \left( \prod^{s'-1}_{j=1} \left( \frac{a_j - a}{(1+a)(1+a_j)} \right)^{-r_j} \right) \left( (1+a)\sqrt{\left( 1 - \frac{1}{\gamma^2 a^2} \right) M} \right)^{1-t'} \\
e^{-(1+a)\sqrt{\left( 1 - \frac{1}{\gamma^2 a^2} \right) M} \eta w - \frac{1}{2} (1+a)^2 \left( 1 - \frac{1}{\gamma^2 a^2} \right) Mw^2} \frac{dw}{w^{t'}} \left( 1 + \varO(M^{-1/2}) \right),
\end{multline}
and the $\varO(M^{-1/2})$ is independent of $\eta$. After the substitution $u = (1+a)\sqrt{\left( 1 - \frac{1}{\gamma^2 a^2} \right) M} w$, we get
\begin{equation}
\frac{1}{2\pi i} \oint_{\tildeGammaM_0} G_{Mat'}(\eta, w) dw = \frac{1}{2\pi i} \oint_{\tildeGammainf} e^{-\eta u - \frac{u^2}{2}} \frac{du}{u^{t'}} \left( 1 + \varO(M^{-1/2}) \right).
\end{equation}
On $\tildeGammainf$, if $\eta \geq T$, $\lvert e^{-(\eta-T)u} \rvert \leq e^{\lvert T/3 \rvert}e^{\eta/3}$, and we have
\begin{multline}
\left\lvert \frac{1}{2\pi i} \oint_{\tildeGammaM_0} G_{Mat'}(\eta, w) dw - \frac{(-a)^r}{(1+a)^{r'-1}\bar{C}_{r'-t'}} \frac{(-1)^{t'-1}}{(t'-1)!} H_{t'-1}(\eta) \right\rvert 
\leq \\
e^{\lvert T/3 \rvert}e^{\eta/3} \left\lvert \frac{(-a)^r}{(1+a)^{r'-1}\bar{C}_{r'-t'}} \frac{1}{2\pi i} \oint_{\tildeGammainf} \left\lvert e^{-\eta u - \frac{u^2}{2}} u^{-t'} \right\rvert du \varO(M^{-1/2}) \right\rvert,
\end{multline}
and we can get the result by direct calculation.
\end{proof}

Now we estimate the integral of $G_{Mat'}$ over $\tildeGammaM_{\rt}$, where $\tildeGammaM_{\rt}$ is defined as the union of $\tildeGammaM_1 \cup \dots \cup \tildeGammaM_8$, as
\begin{align}
\tildeGammaM_1 = & \left\{ \delta -it \left\lvert -\frac{\delta}{3} \leq t \leq \frac{\delta}{3} \right. \right\}, \\
\tildeGammaM_2 = & \left\{ -t + \frac{\delta}{3}i \left\lvert -\frac{\gamma^2 a^2 - 1}{(1+\gamma^2 a)(1+a)} \leq t \leq -\delta \right. \right\}, \\
\tildeGammaM_3 = & \left\{ t - \frac{\delta}{3}i \left\lvert \delta \leq t \leq \frac{\gamma^2 a^2 - 1}{(1+\gamma^2 a)(1+a)} \right. \right\}, \\
\tildeGammaM_4 = & \left\{ \frac{\gamma^2 a^2 - 1}{(1+\gamma^2 a)(1+a)} - it \left\lvert -4 \leq t \leq -\frac{\delta}{3} \right. \right\}, \\
\tildeGammaM_5 = & \left\{ \frac{\gamma^2 a^2 - 1}{(1+\gamma^2 a)(1+a)} - it \left\lvert \frac{\delta}{3} \leq t \leq 4 \right. \right\}, \\
\tildeGammaM_6 = & \left\{ -t + 4i \left\lvert -C_{\rt} \leq t \leq -\frac{\gamma^2 a^2 - 1}{(1+\gamma^2 a)(1+a)} \right. \right\}, \\
\tildeGammaM_7 = & \left\{ t - 4i \left\lvert \frac{\gamma^2 a^2 - 1}{(1+\gamma^2 a)(1+a)} \leq t \leq C_{\rt} \right. \right\}, \\
\tildeGammaM_8 = & \left\{ C_{\rt} + it \mid -4 \leq t \leq 4 \right\},
\end{align}
where $\delta$ is a very small positive number and $C_{\rt}$ is the same large positive number as that in the definition of $\barGammaM$. Then we can prove 

\begin{lemma} \label{lemma:second_lemma_for_varphi_more}
If $T$ is fixed and $M$ is large enough,
\begin{equation}
\left\lvert \frac{1}{2\pi i} \oint_{\tildeGammaM_{\rt}} G_{Mat'}(\eta, w) dw \right\rvert < \frac{1}{2} \frac{e^{\eta/3}}{M^{1/3}},
\end{equation}
for any $\eta \geq T$.
\end{lemma}
The proof is similar to those of lemmas \ref{lemma:first_lemma_for_Sigma_as_less_than_gammainv} and \ref{lemma:second_lemma_for_Sigma_as_less_than_gammainv}.

Hence by lemmas \ref{lemma:first_lemma_for_varphi_more} and \ref{lemma:second_lemma_for_varphi_more}, and \eqref{eq:variation_of_varpsi_more}, we have
\begin{multline} \label{eq:significant_varpsi_rprime_more}
\left\lvert \left( (1+a)\sqrt{\left( 1 - \frac{1}{\gamma^2 a^2} \right) M} \right)^{1-t'} (-a)^N (1+a)^{M-N} e^{\frac{M}{1+a}y} e^{-2\eta/3} \varpsi_{r'}(p+q\eta) \right. \\
\left.
\vphantom{\left( (1+a)\sqrt{\left( 1 - \frac{1}{\gamma^2 a^2} \right) M} \right)^{1-t'}}
- \frac{(-a)^r}{(1+a)^{r'-1}\bar{C}_{a,r'-t'}} \frac{(-1)^{t'-1}}{(t'-1)!} e^{-2\eta/3} H_{t'-1}(\eta) \right\rvert < \frac{e^{-\eta/3}}{M^{1/3}}.
\end{multline}

For $\varpsi_{r'}$ with $a_{s'} < a$ or $\varpsi$, we can choose $\tildeGammaM_{\rt}$ so that all poles are included in it, and we can get the estimation by results similar to lemma \ref{lemma:second_lemma_for_varphi_more}. The result is
\begin{equation} \label{eq:insignificant_part_of_varpsi_more}
\left\lvert (1+a)\sqrt{\left( 1 - \frac{1}{\gamma^2 a^2} \right) M} (-a)^N (1+a)^{M-N} e^{\frac{M}{1+a}y} e^{-2\eta/3} \varpsi_*(p+q\eta) \right\rvert < \frac{e^{-\eta/3}}{M^{1/3}},
\end{equation}
where $\varpsi_*$ stands for $\varpsi_{r'}$ with $a_{s'} < a$, or simply $\varpsi$.

%% file: asymptotics_Laguerre.tex
\section{Asymptotics of Laguerre polynomials and related functions}

\label{asymptotics_Laguerre}

In chapter \ref{quaternionic_spiked_model_rank_1}, we need asymptotic results of $\qLag{2N-2}$, $\qLag{2N-1}$ $\psi_{2N-1}$, $\psi'_{2N-1}$, etc. They can be expressed by linear combination of Laguerre polynomials and has the similar integral representation. Like \eqref{eq:first_integral_formula_of_Laguerre}, we have
\begin{equation}
\qLag{2N-1}(Mx) = \frac{e^{2Mx}}{2\pi i} \oint_{\Gamma} e^{-2Mxz} \frac{z^{2M-1}}{(z-1)^{2N}} dz,
\end{equation}
and by \eqref{eq:psi_2N-1_raw_formula}, we get
\begin{equation}
\begin{split}
\varphi_{2N-1}(x) = &  e^{\frac{a}{1+a}2Mx} - \frac{(1+a)^{2(M-N)+1}}{2\pi i}e^{2Mx} \oint_{\Gamma} e^{-2Mxz} \frac{1 + \left( a\frac{z}{z-1} \right)^{2N-1}}{1 + a\frac{z}{z-1}} \frac{z^{2(M-N)}}{z-1} dz \\
= & e^{\frac{a}{1+a}2Mx} - \frac{(1+a)^{2(M-N)+1}}{2\pi i}e^{2Mx} \oint_{\Gamma} e^{-2Mxz} \frac{z^{2(M-N)}}{z - \frac{1}{1+a}} dz \\
&  - \frac{(1+a)^{2(M-N)+1}a^{2N-1}}{2\pi i}e^{2Mx} \oint_{\Gamma} e^{-2Mxz} \frac{z^{2M}}{(z-1)^{2N}} \frac{z-1}{((1+a)z - 1)z} dz.
\end{split}
\end{equation}
If the pole $z = \frac{1}{1+a}$ is inside of $\Gamma$, then 
\begin{equation}
\frac{(1+a)^{2(M-N)+1}}{2\pi i}e^{2Mx} \oint_{\Gamma} e^{-2Mxz} \frac{z^{2(M-N)}}{z - \frac{1}{1+a}} dz = e^{\frac{a}{1+a}2Mx},
\end{equation}
and
\begin{equation} \label{eq:integral_rep_of_phi_when_lucky}
\varphi_{2N-1}(x) = - \frac{(1+a)^{2(M-N)+1}a^{2N-1}}{2\pi i}e^{2Mx} \oint_{\Gamma} e^{-2Mxz} \frac{z^{2M}}{(z-1)^{2N}} \frac{z-1}{((1+a)z - 1)z} dz.
\end{equation}
All other relevant functions can also be expressed by integrals by \eqref{eq:first_integral_formula_of_Laguerre}. We analyze three typical examples, and all other results can be derived similarly.

\subsection{$\qLag{2N-1}(2Mx)$ around $(1+\gamma^{-1})^2$}

We assume in $p = (1+\gamma^{-1})^2$ and $q = \frac{(\gamma+1)^{4/3}}{\gamma(2M)^{2/3}}$, and take $x = p+q\xi$. By methods in the analysis of  $\varpsi(p+q\xi)$ and $\varpsi_{r'}(p+q\xi)$ in subsection \ref{asymptotics_of_varpsi_and+varpsi_rprime}, if we denote
\begin{equation} \label{eq:def_of_Ttildepsi}
\Ttildepsi(\xi) = \frac{1}{2\pi i} \oint_{\Gamma} e^{-2Mxz} \frac{z^{2M-1}}{(z-1)^{2N}} dz,
\end{equation}
we have for $\xi \geq T$
\begin{equation}
\left\lvert \frac{(\gamma+1)^{4/3}}{\gamma}(2M)^{1/3} \frac{(\gamma+1)^{2(M-N)-1}}{\gamma^{2M}} e^{\frac{\gamma}{\gamma+1}2Mx}\Ttildepsi(\xi) - (-1)\Ai(\xi) \right\rvert < \frac{e^{-\xi/2}}{M^{1/40}}.
\end{equation}
The relation between $\Ttildepsi(\xi)$ and $\qLag{2N-1}(2Mx)$ is
\begin{equation} \label{eq:asymptotics_of_Laguerre_with_exp}
\qLag{2N-1}(2Mx) x^{M-N-1/2} e^{-Mx} = \frac{x^{M-N-1/2}}{e^{\frac{\gamma-1}{\gamma+1}Mx}} e^{\frac{\gamma}{\gamma+1}2Mx} \Ttildepsi(\xi).
\end{equation}
For $x = p+q\xi$, $\xi \geq T$, we have the point-wise with respect to $\xi$
\begin{equation} \label{eq:pointwise_convergence_for_Laguerre}
\lim_{M \rightarrow} \frac{e^{M-N}}{\left( \frac{\gamma+1}{\gamma} \right)^{2(M-N)-1}} \frac{x^{M-N-1/2}}{e^{\frac{\gamma-1}{\gamma+1}Mx}} = 1,
\end{equation}
and for $\xi \geq 0$,
\begin{equation} \label{eq:global_boundedness_for_Laguerre}
0 < \frac{e^{M-N}}{\left( \frac{\gamma+1}{\gamma} \right)^{2(M-N)-1}} \frac{x^{M-N-1/2}}{e^{\frac{\gamma-1}{\gamma+1}Mx}} \leq 1.
\end{equation}
Therefore we know that for $M$ large enough and $\xi \geq T$
\begin{multline} \label{eq:asymototics_Laguerre_subcritical}
\left\lvert \gamma^{-2N-1}(\gamma+1)^{4/3}(2M)^{1/3} e^{M-N} \qLag{2N-1}(2Mx) x^{M-N-1/2} e^{-Mx} - (-1)\Ai(\xi) \right\rvert \\
< \frac{e^{-\xi/2}}{M^{1/40}}.
\end{multline}

\subsection{$\psi_{2N-1}(x)$ around $(1+\gamma^{-1})^2$ when $a = \gamma^{-1}$}

We still take $p = (1+\gamma^{-1})^2$, $q = \frac{(\gamma+1)^{4/3}}{\gamma(2M)^{2/3}}$, and $z = p+q\xi$. We use the integral representation \eqref{eq:integral_rep_of_phi_when_lucky} of $\psi_{2N-1}(x)$, and we will make sure that the pole $z = \frac{1}{1+a}$ is inside $\Gamma$ when we deform it. Since $a = \gamma^{-1}$, we can write \eqref{eq:integral_rep_of_phi_when_lucky} as 
\begin{equation}
\psi_{2N-1}(x) = -\frac{(\gamma+1)^{2(M-N)+1}}{\gamma^{2M}} \frac{e^{2Mx}}{2\pi i} \oint_{\Gamma} e^{-2Mxz} \frac{z^{2M}}{(z-1)^{2N}} \frac{z-1}{((1+\gamma^{-1})z - 1)z} dz.
\end{equation}
If we denote
\begin{equation}
\Ttildepsi_{2N-1}(\xi) = \frac{1}{2\pi i} \oint_{\Gamma} e^{-2Mxz} \frac{z^{2M}}{(z-1)^{2N}} \frac{z-1}{((1+\gamma^{-1})z - 1)z} dz,
\end{equation}
by methods in the analysis of $\varpsi_{r'}(p+q\xi)$ in subsection \ref{asymptotics_of_varpsi_rprime_equal}, we have for $\xi \geq T$,
\begin{equation}
\left\lvert \frac{(\gamma+1)^{2(M-N)+1}}{\gamma^{2M}}e^{\frac{\gamma}{\gamma+1}2Mx} \Ttildepsi_{2N-1}(\xi) - (-1)s^{(1)}(\xi) \right\rvert < \frac{e^{\xi/6}}{M^{1/40}}.
\end{equation}
We can express $\psi_{2N-1}(x)$ as
\begin{equation}
\psi_{2N-1}(x) = \frac{x^{M-N+1/2}}{e^{\frac{\gamma-1}{\gamma+1}Mx}} \frac{(\gamma+1)^{2(M-N)+1}}{\gamma^{2M}}e^{\frac{\gamma}{\gamma+1}2Mx} \Ttildepsi_{2N-1}(\xi).
\end{equation}
Similar to \eqref{eq:pointwise_convergence_for_Laguerre} and \eqref{eq:global_boundedness_for_Laguerre}, we have point wise convergence with respect to $\xi$
\begin{equation}
\lim_{M \rightarrow \infty} \frac{e^{M-N}}{\left( \frac{\gamma+1}{\gamma} \right)^{2(M-N)+1}} \frac{x^{M-N+1/2}}{e^{\frac{\gamma-1}{\gamma+1}Mx}} = 1,
\end{equation}
and if $\xi \geq T$, then
\begin{equation}
0 < \frac{e^{M-N}}{\left( \frac{\gamma+1}{\gamma} \right)^{2(M-N)+1}} \frac{x^{M-N+1/2}}{e^{\frac{\gamma-1}{\gamma+1}Mx}} < 1 + \varO \left( \frac{1}{M} \right).
\end{equation}
Therefore, we have
\begin{equation} \label{eq:asymptotics_psi_2N-1_critical}
\left\lvert e^{M-N} \left( \frac{\gamma}{\gamma+1} \right)^{2(M-N)+1} \psi_{2N-1}(x) - s^{(1)}(\xi) \right\rvert < \frac{e^{\xi/6}}{M^{1/40}}.
\end{equation}

\subsection{$\qLag{2N-1}(2Mx)$ around $(1+a)\left( 1 + \frac{1}{\gamma^2 a} \right)$ when $a > \gamma^{-1}$}

We take $p = (1+a)\left( 1 + \frac{1}{\gamma^2 a} \right)$, $q = (1+a)\sqrt{1 - \frac{1}{\gamma^2 a^2}} \frac{1}{\sqrt{2M}}$ and $x = p + q\xi$. Like \eqref{eq:def_of_Ttildepsi}, we define
\begin{equation}
\Ttildepsi_a(\xi) = \frac{1}{2\pi i} \oint_{\Gamma} e^{-2Mxz} \frac{z^{2M-1}}{(z-1)^{2N}} dz.
\end{equation}

With $g(z)$ defined in \eqref{eq:definition_of_g_x}, we have
\begin{equation}
\Ttildepsi_a(\xi) = \frac{1}{2\pi i} \oint_{\Gamma} e^{2Mg(z) - (1+a)\sqrt{\left( 1 - \frac{1}{\gamma^2 a^2} \right) 2M} \xi z} \frac{dz}{z}.
\end{equation}
After the substitution $z = w + \frac{\gamma^2 a}{1 + \gamma^2 a}$, we get similar to \eqref{eq:substitution_of_the_integral_psi_more}, 
\begin{equation}
\begin{split}
& \oint_{\Gamma} e^{2Mg(z) - (1+a)\sqrt{\left( 1 - \frac{1}{\gamma^2 a^2} \right) 2M} \xi z} \frac{dz}{z} \\
= & \oint_{\TtildeGammaM} e^{-2M \left( (1 + a) + (1 - \gamma^{-2})\log(1+\gamma^2 a) - \log(\gamma6 2 a） - \gamma^{-2}\pi i － \frac{1}{2} (\gamma^{-1} + \gamma a)^2 \left( 1 - \frac{1}{\gamma^2 a^2} \right) w^2 - R_3(w) \right)} \\
& 
\phantom{\oint_{\TtildeGammaM}}
e^{ -(1+a)\sqrt{\left( 1 - \frac{1}{\gamma^2 a^2} \right) 2M} \xi \left( w + \frac{\gamma^2 a}{1 + \gamma^2 a} \right)} \frac{dw}{w + \frac{\gamma^2 a}{1 + \gamma^2 a}} \\
= & \frac{(\gamma^2 a)^{2M} e^{-\frac{\gamma^2 a}{1 + \gamma^2 a}2Mx}}{(1 + \gamma^2 a)^{2(M-N)}} \\
& \oint_{\TtildeGammaM} e^{(\gamma^{-1} + \gamma a)^2 \left( 1 - \frac{1}{\gamma^2 a^2} \right) Mw^2 -  -(1+a)\sqrt{\left( 1 - \frac{1}{\gamma^2 a^2} \right) 2M} \xi w + 2MR_3(w)} \frac{dw}{w + \frac{\gamma^2 a}{1 + \gamma^2 a}},
\end{split}
\end{equation}
where 
\begin{equation}
R_3(w) = \varO(w^3), \text{ as } w \rightarrow 0,
\end{equation}
and $\TtildeGammaM$ is a contour around $w = \frac{1}{1 + \gamma^2 a}$, composed of $\TtildeGammaM_1$, $\TtildeGammaM_2$, $\TtildeGammaM_3$ and $\TtildeGammaM_4$, which are defined as
\begin{align}
\TtildeGammaM_1 = & \{ -it + q^{-1}/(2M) \mid -2 \leq t \leq 2 \}, \\
\TtildeGammaM_2 = & \{ 4 - t + 2i \mid 0 \leq t \leq 4 - q^{-1}/M \}, \\
\TtildeGammaM_3 = & \{ 4 + it \mid -2 \leq t \leq 2 \}, \\
\TtildeGammaM_4 = & \{ t - 2i \mid q^{-1}/(2M) \leq t \leq 4 \}.
\end{align}
\begin{figure}[h]
\centering
\includegraphics{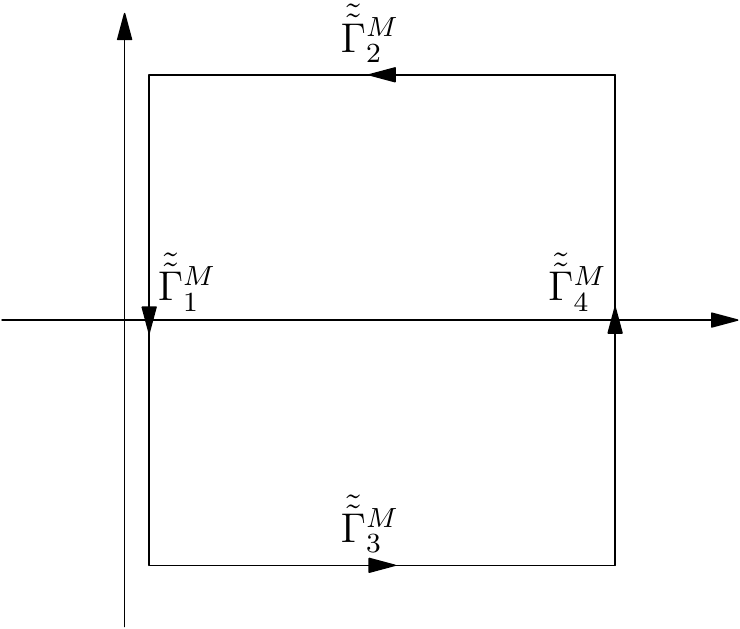}
\caption{$\TtildeGammaM$}
\end{figure}
By methods in the analysis of $\psi_{r'}(p+q\xi)$ in subsection \ref{asymptotics_of_psi_rprime_more}, we have for $\xi \geq T$,
\begin{equation}
\left\lvert \frac{(1 + \gamma^2 a)^{2(M-N)}}{(\gamma^2 a)^{2M}} \sqrt{(\gamma^2 a^2 -1)2M} e^{\frac{\gamma^2 a}{1 + \gamma^2 a}2Mx} \Ttildepsi_a(\xi) - \frac{-1}{\sqrt{2\pi}} e^{-\frac{1}{2} \left( \frac{\gamma + \gamma a}{1 + \gamma^2 a}\xi \right)^2} \right\rvert < \frac{e^{-\xi}}{M^{1/10}}.
\end{equation}
We have the result similar to \eqref{eq:asymptotics_of_Laguerre_with_exp}
\begin{equation}
\qLag{2N-1}(Mx)x^{M-N-1/2} e^{-\frac{\gamma^2 a + a + 2}{(a+1)(\gamma^2 a + 1)}Mx} = \frac{x^{M-N-1/2}}{e^{\frac{M-N}{(1+a)\left( 1 + \frac{1}{\gamma^2 a} \right)}x}} e^{\frac{\gamma^2 a}{1 + \gamma^2 a}2Mx} \Ttildepsi_a(\xi).
\end{equation}
For $x = p+q\xi$, $\xi \geq T$, we have the pointwise convergence with respect to $\xi$
\begin{equation}
\lim_{M \rightarrow \infty} \frac{e^{M-N}}{\left( (1+a) \left( 1 + \frac{1}{\gamma^2 a} \right) \right)^{M-N-1/2}} \frac{x^{M-N-1/2}}{e^{\frac{M-N}{(1+a) \left( 1 + \frac{1}{\gamma^2 a} \right)}x}} = e^{-\frac{1}{4}\frac{(\gamma^2 a^2 - 1)(\gamma^2 - 1)}{(\gamma^2 a + 1)^2}\xi^2},
\end{equation}
and for $\xi \geq 0$,
\begin{equation}
0 < \frac{e^{M-N}}{\left( (1+a) \left( 1 + \frac{1}{\gamma^2 a} \right) \right)^{M-N-1/2}} \frac{x^{M-N-1/2}}{e^{\frac{M-N}{(1+a) \left( 1 + \frac{1}{\gamma^2 a} \right)}x}} < 1.
\end{equation}
Therefore if $M$ is large enough,
\begin{multline} \label{eq:asymototics_Laguerre_supercritical}
\left\lvert \frac{(\gamma^2 a + 1)^{M-N+1/2} e^{M-N} \sqrt{(\gamma^2 a^2 - 1)2M}}{(\gamma^2 a)^{M+N+1/2} (a+1)^{M-N-1/2}}  \qLag{2N-1}(Mx)x^{M-N-1/2} e^{-\frac{\gamma^2 a + a + 2}{(a+1)(\gamma^2 a + 1)}Mx} \right. \\
\left.
\vphantom{\frac{(\gamma^2 a + 1)^{M-N+1/2} e^{M-N} \sqrt{(\gamma^2 a^2 - 1)2M}}{(\gamma^2 a)^{M+N+1/2} (a+1)^{M-N-1/2}}}
- \frac{-1}{\sqrt{2\pi}} e^{-\frac{1}{4} \frac{\gamma^4 a^2 + \gamma^2 a^2 + 4 \gamma^2 a + \gamma^2 + 1}{(\gamma^2 a + 1)^2}\xi^2} \right\rvert < \frac{e^{-\xi}}{M^{1/10}}.
\end{multline}